\documentclass[11pt,a4paper]{article}
\usepackage{amsmath}
\usepackage{amssymb}
\usepackage{amsthm}
\usepackage[dvipdfm]{geometry}
\usepackage{graphicx}
\usepackage{fancybox}
\usepackage{fancyhdr}
\usepackage{bm}
\usepackage{CJKutf8}
\newtheorem{theorem}{Theorem}[section]
\newtheorem{proposition}[theorem]{Proposition}

\theoremstyle{definition}
\newtheorem{definition}[theorem]{Definition}
\newtheorem{example}[theorem]{Example}
\newtheorem{remark}[theorem]{Remark}
\DeclareMathOperator{\shortestpath}{spath}
\DeclareMathOperator{\id}{id}

\DeclareMathOperator{\CNT}{CNT}
\DeclareMathOperator{\trace}{tr}
\DeclareMathOperator{\rank}{rank}
\DeclareMathOperator{\diag}{diag}
\DeclareMathOperator{\Vol}{Vol}
\DeclareMathOperator{\Aut}{Aut}
\DeclareMathOperator{\FI}{I}
\DeclareMathOperator{\II}{II}
\DeclareMathOperator{\Proj}{Proj}
\DeclareMathOperator{\image}{image}
\DeclareMathOperator{\Span}{span}
\newcommand{\Z}{{\mathbb Z}}
\newcommand{\R}{{\mathbb R}}
\newcommand{\N}{{\mathbb N}}
\newcommand{\inner}[2]{{\left\langle{#1},{#2}\right\rangle}}
\newcommand{\widevec}[1]{{\overrightarrow{#1}}}
\newcommand{\numberOf}[1]{{|#1|}}
\renewcommand{\v}[1]{{\bm{#1}}}
\numberwithin{equation}{section}
\numberwithin{figure}{section}
\numberwithin{table}{section}
\fancypagestyle{plain}{%
  \fancyhf{}
  \fancyhead{}
  \fancyhead[R]{}
  \fancyhead[L]{}
  \fancyfoot[R]{{\ttfamily 06 January, 2020}}
  \fancyfoot[C]{\thepage}
  \fancyfoot[L]{{\ttfamily naito@math.nagoya-u.ac.jp}}
  }
\begin{document}
%%%%\input{title.tex}
% -*- coding: utf-8 -*-
% \mbox{}\vspace{2\baselineskip}
\begin{flushleft}
  {\Large\bfseries A Short Lecture on Topological Crystallography\\
    and a Discrete Surface Theory}
\end{flushleft}
\begin{flushleft}
  {\large Hisashi NAITO} \\
  \mbox{}\\
  Graduate School of Mathematics, Nagoya University, Nagoya 464-8602, JAPAN\\
  {\ttfamily naito@math.nagoya-u.ac.jp}
\end{flushleft}
\par\vspace{\baselineskip}\par
%%%%% END of document
\thispagestyle{plain}
%%%%\input{abstract.tex}
% -*- coding: utf-8 -*-
\begin{abstract}
  In this note, we discuss topological crystallography, 
  which is a mathematical theory of crystal structures.
  The most symmetric structure among all placements of the graph is obtained by 
  a variational principle in topological crystallography.
  We also discuss a theory of trivalent discrete surfaces in 
  $3$-dimensional Euclidean space, 
  which are mathematical models of crystal/molecular structures.
\end{abstract}
%%%%% END of document
%%%%% \input{introduction.tex}
% -*- coding: utf-8 -*-
\section{Introduction}
Geometric analysis is a field of analysis on geometric objects such as manifolds, surfaces, and metric spaces.
Discrete geometric analysis is an analysis on {\em discrete} geometric objects, for example, graphs, 
and contains the spectral theory and the probability theory of graphs.
Topological crystallography and a discrete surface theory based on crystal/molecular structures are
also parts of discrete geometric analysis.
\par
Topological crystallography is a mathematical theory of crystal structures, 
which is founded by M.~Kotani and T.~Sunada \cite{MR1748964, MR1783793, MR2039958, MR3014418}.
In physics, chemistry, and mathematics, 
crystal structures are described by space groups, which denote symmetry of placements of atoms.
The usual description of crystals contains bonds between atoms in crystals.
However, space groups do not consider such atomic bonds.
Graphs are also natural notions to describe crystal structure, 
since vertices and edges of graphs correspond to atoms and atomic bonds in crystals, respectively.
On the other hand, one of the important notions of physical phenomena is the principle of the least action, 
which corresponds to the variational principle in mathematics.
That is to say, to describe physical phenomena, 
first we define an energy functional, which is called a Lagrangian in physics, 
and then we may find such phenomena as minimizers of the energy functional.
There is no direct relationship between descriptions of crystal structures by using space groups and the principle of 
the least action.
\par
Topological crystallography gives us a direct relationship between symmetry of crystal structures and the variational principle.
Precisely, for a given graph structure which describes a crystal, 
we define the energy of realizations of the graph 
(placements of vertices of the graph in suitable dimensional Euclidean space), 
and obtain a ``good'' structure as a minimizer of the energy.
Moreover, such structures give us most symmetric among all placements of the graph, 
which is proved by using the theory of random walk on graphs.
\par
Molecular structures can be also described by using graph theory.
The H\"uckel molecular orbital method, 
which is an important theory in physical chemistry, 
and the tight binding approximation for studying electronic states of crystals
can be regarded as spectral theories of graphs from mathematical viewpoints.
In this way, discrete geometric analysis can be applied to physics, chemistry, and related technologies.
\par
In the first few sections, we discuss topological crystallography including graph theory and geometry.
The most important bibliography of this part is T.~Sunada's lecture note \cite{MR2902247}.
The author discusses an introduction to topological crystallography along with it.
\par
\vspace{\baselineskip}
\par
On the other hand, 
we can regard 
some of crystal/molecular structures, for example, fullerenes and carbon nanotubes (see Section \ref{sec:carbon}), 
as surfaces, especially, as discrete surfaces.
Recently, $sp^2$-carbon structures (including fullerenes and nanotubes) are paid attention to in sciences and technologies, 
since they have rich $\pi$-electrons and hence rich physical properties.
In mathematical words, $sp^2$-carbon structures can be regarded as trivalent graphs in $\R^3$, 
and hence trivalent {\em discrete surfaces}.
There are many discrete surface theories in mathematics, 
but they are discretization or discrete analogue of continuous/smooth objects.
For example, discrete surfaces of triangulations are used in computer graphics, 
which is a discretization of smooth real objects.
In other words, conventional discrete surface theories are ``from continuous to discrete''.
In contrast, discrete surfaces, which describe crystal/molecular structures, 
are essentially discrete.
Even for the case of trivalent discrete surfaces, 
it is not easy to define curvatures of them.
\par
In the last few sections, 
we discuss a theory of trivalent discrete surfaces in $\R^3$, 
and also discuss subdivisions/convergences of them.
Hence, our discrete surface theory is ``from discrete to continuous''.
%%%%% END of document
%%%%%\input{preliminaries.tex}
% -*- coding: utf-8 -*-
\section{Preliminaries}
%%%%%\input{graphtheory.tex}
% -*- coding: utf-8 -*-
\subsection{Graph theory}
Here, we prepare a graph theory to describe topological crystals and discrete surfaces.
Definitions and notations are followed by standard text books of graph theory, for example \cite{Bondy, DSV}.

\begin{definition}% [Sunada {\cite{MR3014418}}]
  An ordered pair $X=(V, E)$ is called a {\em graph}, 
  if $V$ is a countable set and $E = \{(u, v) : u, \,v \in V\}$.
  An element $v \in V$ is called a vertex of $X$, 
  and an element $e \in E$ is called an edge of $X$.
  For each element $e = (u, v) \in E$, 
  we may also write $u = o(e)$ and $v = t(e)$, 
  which are called the {\em origin} and the {\em terminus} of $e$.
\end{definition}

\begin{definition}% [Sunada {\cite{MR3014418}}]
  A graph $X = (V, E)$ is called {\em finite}, 
  if the number of vertices $\numberOf{V}$ and the number of edges $\numberOf{E}$ are finite.
  \par
  For a vertex $v \in V$, write $E_v = \{(v, u) \in E\}$, which is the set of edges with $o(e) = v$.
  The number of edges emanating with $v \in V$ is called the {\em degree} $\deg(v) = \numberOf{E_v}$ of $v$.
  If $\deg(v)$ is finite for any $v \in V$, 
  then the graph $X$ is called {\em locally finite}.
  For any $e = (u, v) \in E$, $\overline{e} = (v, u) \in E$, 
  then $X$ is called {\em non-oriented}, otherwise $X$ is called {\em oriented}.
\end{definition}

In this note, we only consider non-oriented and locally finite graphs.
Moreover, we admit graphs which contain a loop $(u, u) \in E$ and multiple edges.

\begin{definition}% [Sunada {\cite{MR3014418}}]
  For a graph $X = (V, E)$, 
  successive edges \scalebox{0.8}[1.0]{$(u_1, u_2)(u_2, u_3) \cdots (u_{k-2}, u_{k-1})(u_{k-1}, u_k)$},  
  where $(u_i, u_j) \in E$, 
  is called a {\em path} between $u_1$ and $u_k$, 
  and if $u_1 = u_k$ and the path does not contains backtracking edges, it is called a {\em closed path}.
  A graph is {\em connected}, if there exists a path between arbitrary two vertices.
  A connected graph is called {\em tree}, if the graph contains no closed path.
\end{definition}

\begin{example}
  Both graphs in Fig.~\ref{fig:k40} are the same as each other, and are called $K_4$ graph.
  Each vertex of $K_4$ graph is connected with all of the other vertices.
  A graph with such a property is called a {\em complete graph}.
  The $K_4$ graph is the complete graph with $4$ vertices.
  \begin{figure}[htp]
    \centering
    \begin{tabular}{ccc}
      \includegraphics[bb=0 0 105 91,height=45pt]{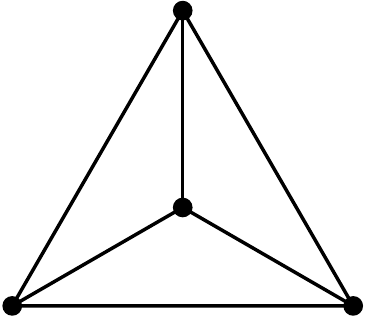}
      &\mbox{}
      &\includegraphics[bb=0 0 154 91,height=45pt]{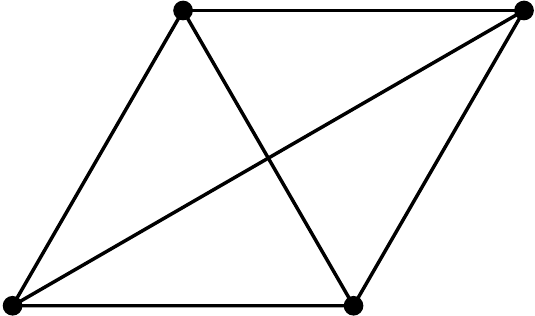}
    \end{tabular}
    \caption{$K_4$ graph, each is different figure of the same graph.}
    \label{fig:k40}
  \end{figure}
\end{example}

\begin{definition}% [Bondy--Murty \cite{Bondy}]
  Let $X = (V, E)$ be a finite graph with $V = \{v_i\}_{i=1}^n$.
  The {\em adjacency matrix} $A = A_X$ of $X$ is 
  an $n\times{n}$ matrix defined by 
  \begin{math}
    a_{ij} = \text{number of edges $(v_i, v_j)$}.
  \end{math}
\end{definition}
If a graph $X$ is non-oriented, then the adjacency matrix of $X$ is symmetric.

\begin{center}
  \begin{figure}
    \centering
    \begin{tabular}{cccc}
      \multicolumn{1}{l}{(a)}
      &\multicolumn{1}{l}{(b)}
      &\multicolumn{1}{l}{(c)}
      &\multicolumn{1}{l}{(d)}\\
      \includegraphics[bb=0 0 105 91,scale=0.5]{F/k4-crop.pdf}
      &\includegraphics[bb=0 0 122 107,scale=0.5]{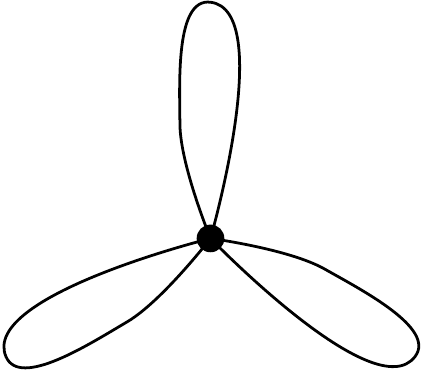}
      &\includegraphics[bb=0 0 120 67,scale=0.5]{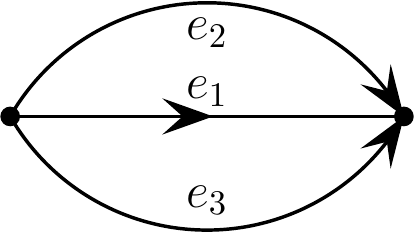}
      &\includegraphics[bb=0 0 115 117,scale=0.5]{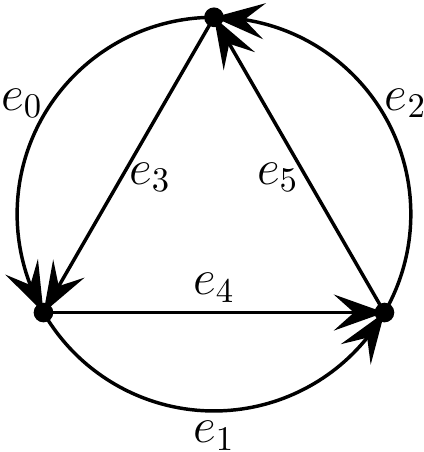}
      \\
      \begin{math}
        \begin{bmatrix}
          0 & 1 & 1 & 1 \\
          1 & 0 & 1 & 1 \\
          1 & 1 & 0 & 1 \\
          1 & 1 & 1 & 0
        \end{bmatrix}
      \end{math}
      &\begin{math}
        \begin{bmatrix}
          3 
        \end{bmatrix}
      \end{math}
      &\begin{math}
        \begin{bmatrix}
          3 & 0 \\
          0 & 3 \\
        \end{bmatrix}
      \end{math}
      &\begin{math}
        \begin{bmatrix}
          0 & 2 & 2 \\
          2 & 0 & 2 \\
          2 & 2 & 0 \\
        \end{bmatrix}
      \end{math}
      \\
    \end{tabular}
    \caption{Examples of graphs and their adjacency matrices.}
    \label{fig:adjacency}
  \end{figure}
\end{center}

\begin{remark}
  Let $A$ be an adjacency matrix of a graph $X$, 
  \begin{enumerate}
  \item 
    $(i, j)$-element of $A^k$ expresses 
    number of paths from $v_i$ to $v_j$ by $k$-steps
  \item 
    if $A^{n-1}$ ($n = \numberOf{V}$) is not block diagonal, 
    then $X$ is connected
  \item 
    if $X$ is simple and non-oriented, 
    then 
    \begin{math}
      \displaystyle
      (1/6) \trace(A^3)
    \end{math}
    expresses the number of triangles contained in $X$
  \end{enumerate}
\end{remark}

\begin{definition}% [Sunada {\cite{MR3014418}}]
  Let $X = (V, E)$ be a finite non-oriented graph.
  A tree $X_1 = (V, E_1)$ is called a {\em spanning tree} of $X$, 
  if it satisfies $E_1 \subset E$ and
  for any $e \in E\setminus{E_1}$, $(V, E_1 \cup \{e\})$ contains a closed path (see Fig.~\ref{fig:maxspanningtreeofK4}).
\end{definition}

\begin{figure}[htp]
  \centering
  \includegraphics[bb=0 0 290 91,height=45pt]{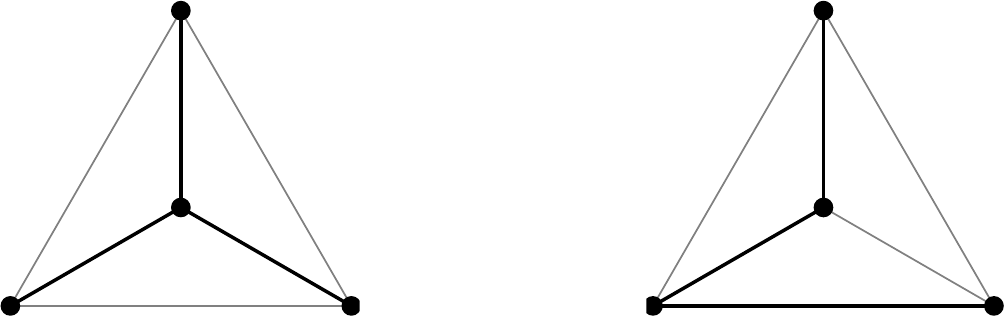}
  \caption{Each graph is a spanning tree of $K_4$ graph.}
  \label{fig:maxspanningtreeofK4}
\end{figure}

Now, we also consider homology groups of graphs.
A graph $X = (V, E)$ can be considered as a $1$-dimensional CW complex as follows:
the $0$-dimensional chain group $C_0$ is the $\Z$-module consisted by $V$, 
and the $1$-dimensional chain group $C_1$ is the $\Z$-module consisted by $E$.
The boundary operator $\partial \colon C_1 \longrightarrow C_0$ is defined by
\begin{displaymath}
  \partial(e) = t(e) - o(e), 
\end{displaymath}
where $o(e)$ and $t(e)$ are the origin and terminus of the edge $e$, 
namely, $o(e) = u$ and $t(e) = v$, if $e = (u, v)$.
The homology group $H_0(X, \Z)$ and $H_1(X, \Z)$ is defined by 
\begin{displaymath}
  H_1(X, \Z) = \ker \partial \subset C_1(X), 
  \quad
  H_0(X, \Z) = C_0(X)/\image \partial.
\end{displaymath}
The following proposition explains particular properties 
for the first homology group on graphs.
\begin{proposition}[Sunada {\cite{MR3014418}}]
  \label{claim:graph:homology1}
  Let $X = (V, E)$ be a locally finite graph. 
  Then, 
  any non-trivial closed path $e$ satisfies $[e] \not= [0] \in H_1(X, \Z)$.
  Conversely, for each non-zero element $h \in H_1(X, \Z)$, 
  there exists a closed path $e$ in $X$ such that $h = [e]$.
\end{proposition}
\begin{proof}
  Let $e = e_1 \cdots e_k$ ($e_i \in E$) a path in $X$.
  Then, we may write $e = e_1 + \cdots + e_k \in C_1(X, \Z)$, and
  $\partial(e) = t(e_1) - o(e_1) + t(e_2) - o(e_2) + \cdots + t(e_k) - o(e_k)$.
  Since $o(e_{i+1}) = t(e_i)$, we obtain 
  $\partial(e) = t(e_1) - o(e_k)$.
  Assuming $e$ is closed, that is $o(e_1) = t(e_k)$, we obtain $\partial(e) = 0$, 
  and $[e] \not= [0] \in H_1(X, \Z)$.
  \par
  Conversely, we take an $h \in H_1(X, \Z) = \ker \partial \subset C_1(X, \Z)$, 
  and let $C_1(X) = \Span\{e_i : i = 1, \ldots, n\}$, 
  then there exists $\alpha_i \in \Z$ such that 
  $h = \alpha_1 e_1 + \cdots + \alpha_n e_n$, 
  and $\partial h = 0$.
  The equation $\partial h = \sum \alpha_i (t(e_i) - o(e_i))$ implies 
  $h$ is the sum of closed paths (see \cite[p.41]{MR2902247}).
\end{proof}

Proposition \ref{claim:graph:homology1} implies that 
an elements of $H_1(X, \Z)$ corresponds to a closed path of $X$.
Hence we obtain a method for counting the rank of $H_1(X, \Z)$.

\begin{proposition}
  \label{claim:graph:homology2}
  Let $X = (V, E)$ be a finite non-oriented graph, and
  $X_1 = (V, E_1)$ be a spanning tree of $X$.
  Then, the first homology group $H_1(X, \Z)$ of $X$ satisfies $\rank H_1(X, \Z) = \numberOf{E} - \numberOf{E_1}$.
\end{proposition}
\begin{proof}
  Since $X_1$ is a tree, $X_1$ does not contain closed path.
  For each edge $e_0 = (u, v) \in E\setminus E_1$, 
  we may find unique path $e = e_1 \cdots e_k$ in $X_1$ with $o(e_1) = v$, $t(e_k) = u$, 
  and hence, $\widetilde{e_0} = e_0 e$ is a closed path in $X$.
  By Proposition \ref{claim:graph:homology1}, 
  we obtain $[\widetilde{e_0}] \in H_1(X, \Z)$.
  Therefore, for each $e_i \in E \setminus E_1$, there exists $[\widetilde{e_i}] \in H_1(X, \Z)$ by a similar manner, 
  and $\{[\widetilde{e_i}]\}$ are linearly independent.
\end{proof}

\begin{example}
  The rank of the first homology group of graphs in Fig.~\ref{fig:adjacency} are
  (a) $3$, (b) $3$, (c) $2$, and (d) $4$, respectively.
\end{example}

\begin{remark}
  An algorithm to find a spanning tree of a finite graph is well-known as 
  Kruskal's algorithm, which finds a spanning tree within $O(\numberOf{E} \log \numberOf{E})$ (see for example \cite{Aho}).
\end{remark}
%%%%% END of document
%%%%% \input{covering.tex}
% -*- coding: utf-8 -*-
\subsection{Covering spaces}
Definitions and notations are followed by standard text books of geometry and topology, for example \cite{SingerThorpe}.

\begin{definition}
  Let $X$ and $X_0$ be topological spaces.
  The space $X$ is a {\em covering space} of $X_0$ if
  there exists a surjective continuous map $p \colon X \longrightarrow X_0$, 
  which is called a {\em covering map}, 
  such that
  for each $x \in X_0$, 
  there exists an open neighbourhood $U$ of $x$
  and open sets $\{V_i\} \subset X$ satisfying 
  $p^{-1}(U) = \sqcup V_i$ with
  $p|_{V_i} \colon V_i \longrightarrow U$ homeomorphic.
\end{definition}

\begin{theorem}
  \mbox{}
  \begin{enumerate}
  \item 
    For any topological space $X_0$, 
    there exists the unique simply connected covering space $\tilde{X}$, 
    which is called the {\em universal covering} of $X_0$.
  \item 
    If $p \colon X \longrightarrow X_0$ is a covering map, 
    then there exists a transformation group $T$ on $X$ such that
    for any $\sigma \in T$, 
    $p \circ \sigma = p$.
    The group $T$ is called the {\em covering transformation group}.
  \item 
    The covering transformation group of $p \colon \tilde{X} \longrightarrow X$ is the fundamental group $\pi_1(X)$ of $X$.
  \end{enumerate}
  \begin{center}
    \includegraphics[bb=0 0 192 107]{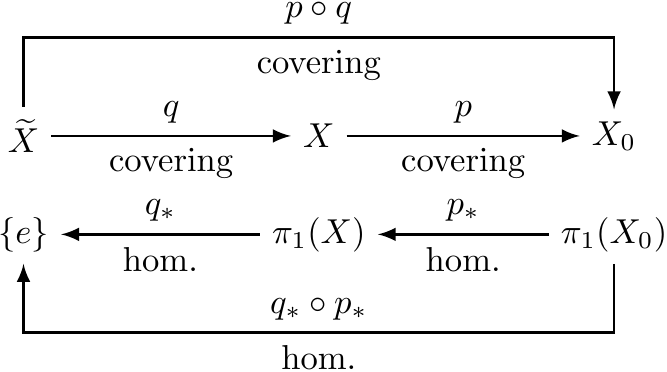}\\
  \end{center}
\end{theorem}

\begin{example}
  \mbox{}
  \begin{enumerate}
  \item 
    The real line $\R$ is a covering space of $S^1$, 
    since $p \colon \R \longrightarrow S^1$, 
    $p(x) = x \pmod{2\pi}$.
    Since $\R$ is simply connected and 
    the covering transformation group of $p$ is $\Z$,
    we obtain $\pi_1(S^1) \cong \Z$.
  \item 
    The 2-dimensional Euclidean space $\R^2$ is a covering space of 
    $T^2 = \R^2/\Z^2$, since 
    $p \colon \R^2 \longrightarrow T^2$, 
    $p(x, y) = (x \pmod{2\pi}, y \pmod{2\pi})$.
    Since $\R^2$ is simply connected and 
    the covering transformation group of $p$ is $\Z^2$,
    we obtain $\pi_1(T^2) \cong \Z^2$.
  \end{enumerate}
\end{example}

\begin{definition}[Sunada {\cite{MR3014418}}]
  A covering $p \colon X \longrightarrow X_0$ with an abelian covering transformation group
  is called an {\em abelian covering}.
  For any topological space $X_0$, there exists a {\em maximal abelian covering} space $X$, 
  since $H_1(X_0, \Z) = \pi_1(X_0)/[\pi_1(X_0), \pi_1(X_0)]$ is a maximal abelian subgroup of $\pi_1(X_0)$
\end{definition}

\begin{example}
  The universal covering graph of the $3$-bouquet graph $X$ (Fig.~\ref{fig:adjacency} (b)) is
  a tree graph with the degree $3$.
  The fundamental group $\pi_1(X)$ is the free group with $3$ elements, and
  the first homology group $H_1(Z, \Z)$ is $\Z^3$.
  For any normal subgroup $S \subset H_1(X, \Z)$, 
  there exists a graph $X_S$ such that 
  $X_S$ is a covering graph of $X$ with its covering transformation group $S$.
\end{example}
%%%%% END of document
%%%%% END of document
%%%%% \input{crystallography.tex}
% -*- coding: utf-8 -*-
\section{Topological crystals and their standard realization}

The classical description of crystallography is based on group theory, 
and they describe symmetries of crystals.
For example, in the classical crystallography, 
the diamond crystal is classified as the space group $F\!d\overline{3}m$ (see Section \ref{sec:spacegroup}), 
and the group does not contain information of chemical bonds of atoms in the crystal.
\par
The theory of topological crystals is developed by Kotani--Sunada \cite{MR1743611, MR1748964, MR1783793, MR2039958}.
The theory describes symmetries of crystals including chemical bonds of atoms, 
and it is based on variational problems.

%%%%% \input{topological-crystal.tex}
% -*- coding: utf-8 -*-
\subsection{Topological crystals and their realizations}
In this section, we assume that graphs are connected non-oriented locally finite, 
which may include self-loops and multiple edges.

\begin{definition}[Sunada \cite{MR3014418}]
  A connected non-oriented locally finite graph $X = (V, E)$, 
  which may include self-loops and multiple edges, 
  is called a {\em topological crystal} (or a {\em crystal lattice}), 
  if and only if 
  there exists an abelian group $G$ which acts freely on $X$.
  The topological crystal is {\em $d$-dimensional}
  if the rank of the abelian group $G$ is $d$.
\end{definition}

By this definition, for a topological crystal $X = (V, E)$, 
there exists a finite graph $X_0 = (V_0, E_0)$ satisfying $X/G = X_0$, 
for an abelian subgroup $G \subset H_1(X_0, \Z)$, 
and $X$ is a covering graph of $X_0$ whose covering transformation group is $G$.
On the contrary, for a given connected non-oriented finite graph $X_0 = (V_0, E_0)$ 
and an abelian subgroup $G \subset H_1(X_0, \Z)$, 
there exists a topological crystal $X$ with $X/G = X_0$, 
by taking a suitable covering graph.

\begin{definition}[Sunada \cite{MR3014418}]
  A topological crystal $X$ is called {\em maximal abelian}
  if and only if $G = H_1(X_0, \Z)$.
\end{definition}

\begin{example}
  A square lattice $X$ (Fig.~\ref{fig:standard} (a)) is a topological crystal
  whose base graph $X_0$ is the graph in Fig.~\ref{fig:basegraph:cubic} (a) (the $2$-bouquet graph).
  Since the covering transformation group $G$ is $G = H_1(X_0, \Z)$ and $\rank G = \rank H_1(X_0, \Z) = 2$, 
  the topological crystal $X$ is $2$-dimensional and maximal abelian.
\end{example}

\begin{example}
  A triangular lattice $X$ (Fig.~\ref{fig:standard} (b)) is a topological crystal
  whose base graph $X_0$ is the graph in Fig.~\ref{fig:adjacency} (b).
  Since the covering transformation group $G$ satisfies $\rank G = 2$
  but $H_1(X_0, \Z) = 3$, 
  the topological crystal $X$ is $2$-dimensional and not maximal abelian.
\end{example}

\begin{example}
  A hexagonal lattice $X$ (Fig.~\ref{fig:standard} (c)) is a topological crystal
  whose base graph $X_0$ is the graph in Fig.~\ref{fig:adjacency} (c).
  Since the covering transformation group $G$ is $G = H_1(X_0, \Z)$
  and $\rank G = \rank H_1(X_0, \Z) = 2$, 
  the topological crystal $X$ is $2$-dimensional and maximal abelian.
\end{example}

\begin{example}
  A kagome lattice $X$ (Fig.~\ref{fig:standard} (d)) is a topological crystal
  whose base graph $X_0$ is the graph in Fig.~\ref{fig:adjacency} (d).
  Since the covering transformation group $G$ satisfies $\rank G = 2$
  but $H_1(X_0, \Z) = 4$, 
  the topological crystal $X$ is $2$-dimensional and not maximal abelian.
\end{example}

\begin{example}
  A diamond lattice $X$ (Fig.~\ref{fig:fd3m}) is a topological crystal
  whose base graph $X_0$ is the graph in Fig.~\ref{fig:adjacency} (b) (the $3$-bouquet graph).
  Since the covering transformation group $G$ is $G = H_1(X_0, \Z)$
  and $\rank G = \rank H_1(X_0, \Z) = 3$, 
  the topological crystal $X$ is $3$-dimensional and maximal abelian.
\end{example}

\begin{figure}[htp]
  \centering
  \begin{tabular}{cccc}
    \multicolumn{1}{l}{(a)}
    &\multicolumn{1}{l}{(b)}
    &\multicolumn{1}{l}{(c)}
    &\multicolumn{1}{l}{(d)}\\
    \includegraphics[bb=0 0 158 157,height=100pt]{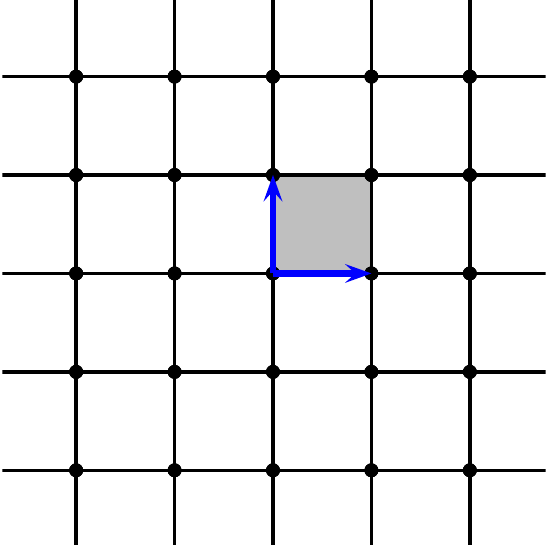}
    &\includegraphics[bb=0 0 158 157,height=100pt]{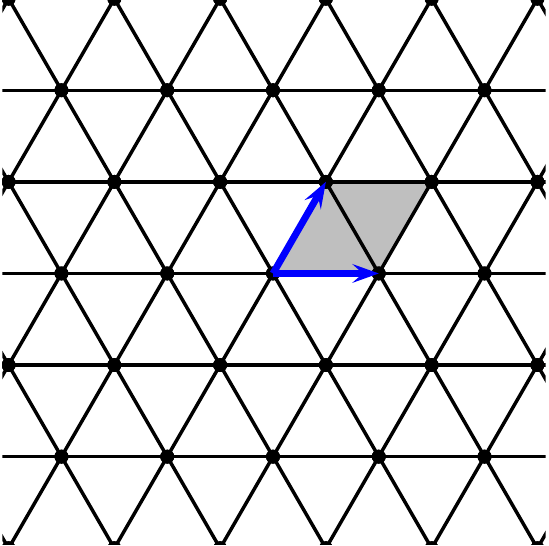}
    &\includegraphics[bb=0 0 158 157,height=100pt]{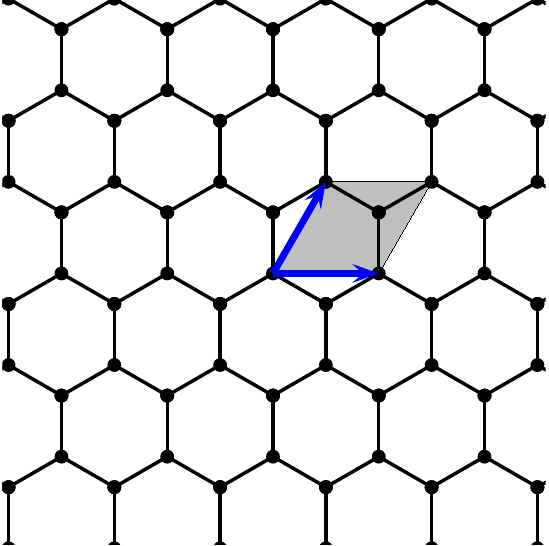}
    &\includegraphics[bb=0 0 158 157,height=100pt]{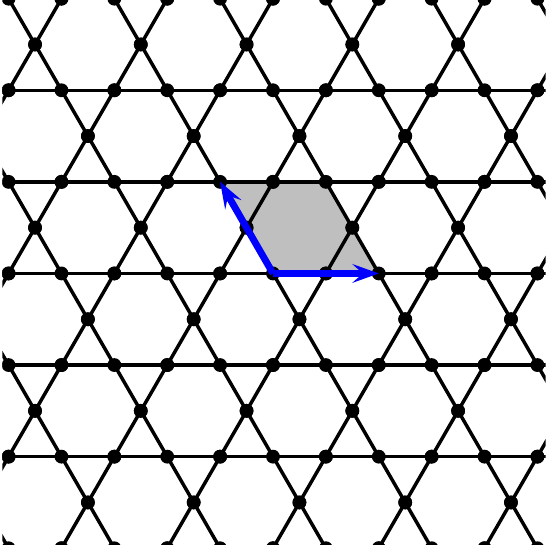}
  \end{tabular}
  \caption{Standard realizations of representative $2$-dimensional topological crystals.
    (a) A square lattice, 
    (b) a triangular lattice, 
    (c) a hexagonal lattice, 
    and
    (d) a kagome lattice.
    Blue vectors are basis of parallel translations.
  }
  \label{fig:standard}
\end{figure}

\begin{definition}[Sunada \cite{MR3014418}]
  Given $d$-dimensional topological crystal $X = (V, E)$, 
  a piecewise linear map $\Phi \colon X \longrightarrow \R^d$ 
  is called a {\em realization} of $X$.
  More precisely, first we define 
  $\Phi \colon V \longrightarrow \R^d$, 
  and define $\Phi(e)$ by linear interpolation between $\Phi(o(e))$ and $\Phi(t(e))$.
\end{definition}

\begin{definition}[Sunada \cite{MR3014418}]
  A realization $\Phi$ of a $d$-dimensional topological crystal $X$ is called
  a {\em periodic realization}, 
  if there exists an injective homomorphism $\rho \colon G \longrightarrow \R^d$ satisfying 
  \begin{displaymath}
    \Phi(gv) = \Phi(v) + \rho(g), \quad (v \in V, \, g \in G).
  \end{displaymath}
\end{definition}

\begin{example}
  Three realizations in Fig.~\ref{fig:hex-realization} are periodic realizations of a hexagonal lattice.
  These are different periodic realizations of the same graph.
  The realization (b) is the most symmetric, 
  and
  the main problem of this section is to explain the reason why nature selects (b) by mathematics.
\end{example}
\begin{figure}[htpb]
  \centering
  \begin{tabular}{ccccc}
    \multicolumn{1}{l}{(a)}
    &\mbox{}
    &\multicolumn{1}{l}{(b)}
    &\mbox{}
    &\multicolumn{1}{l}{(c)}\\
    \includegraphics[bb=0 0 401 401,height=90pt]{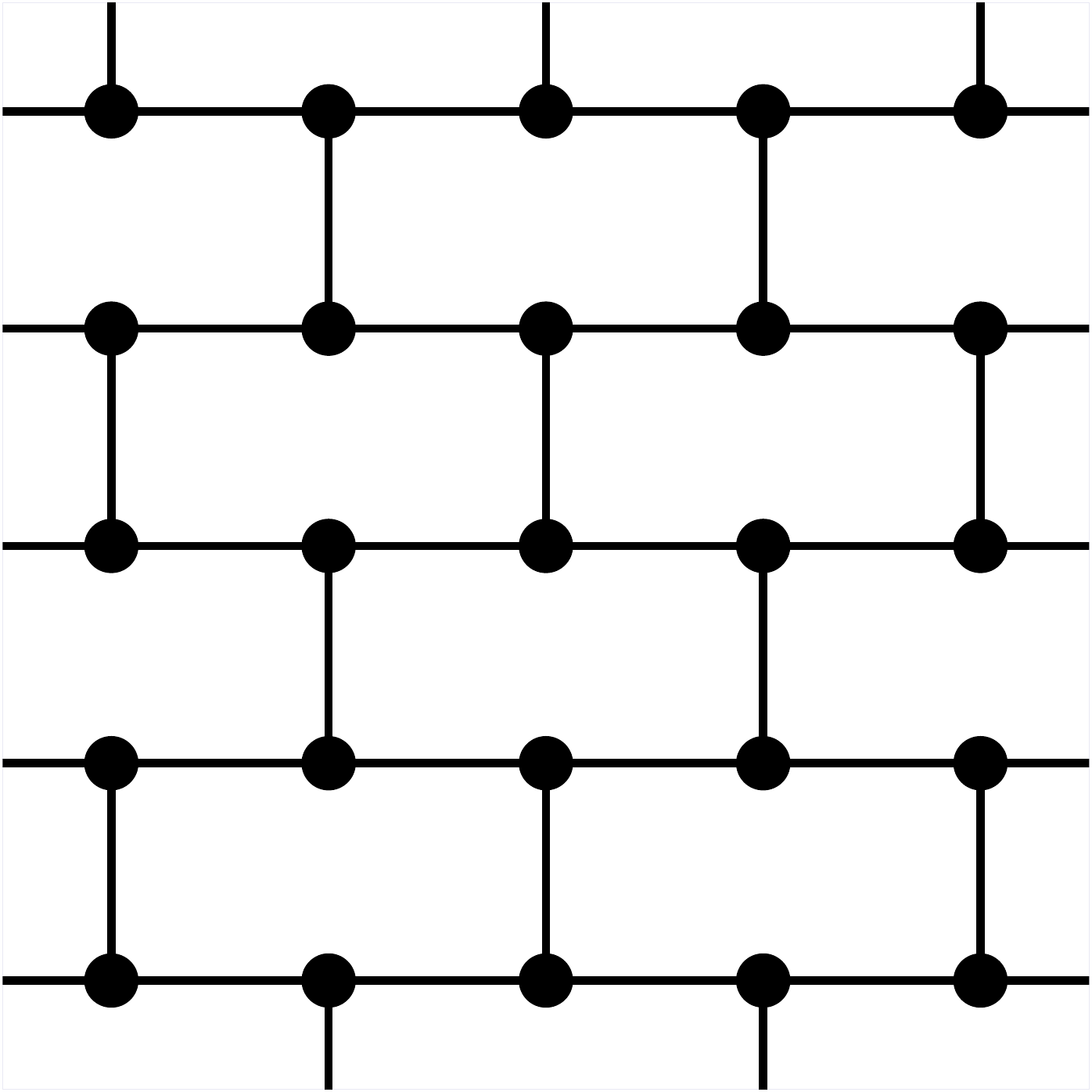}
    &\mbox{}
    &\includegraphics[bb=0 0 401 401,height=90pt]{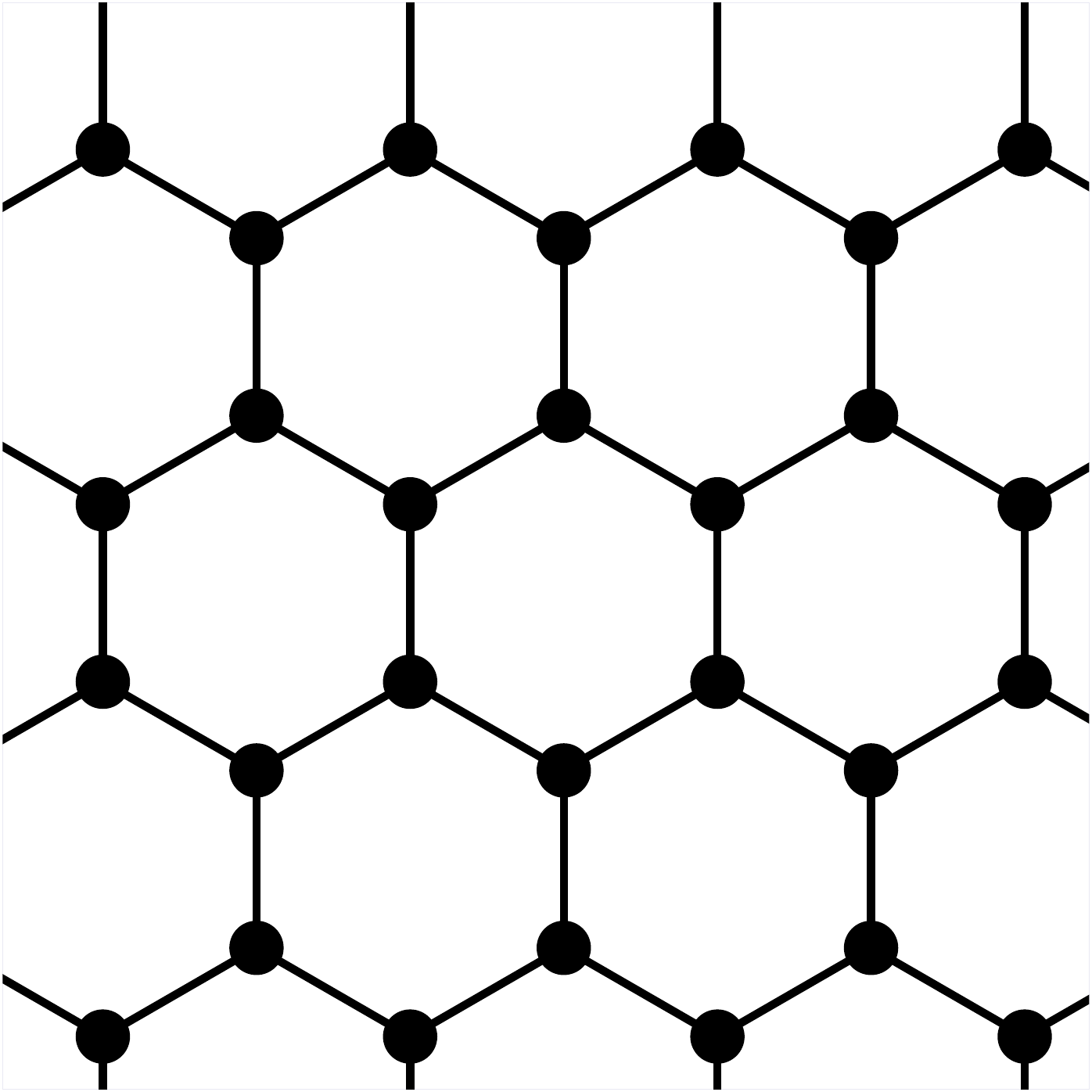}
    &\mbox{}
    &\includegraphics[bb=0 0 401 401,height=90pt]{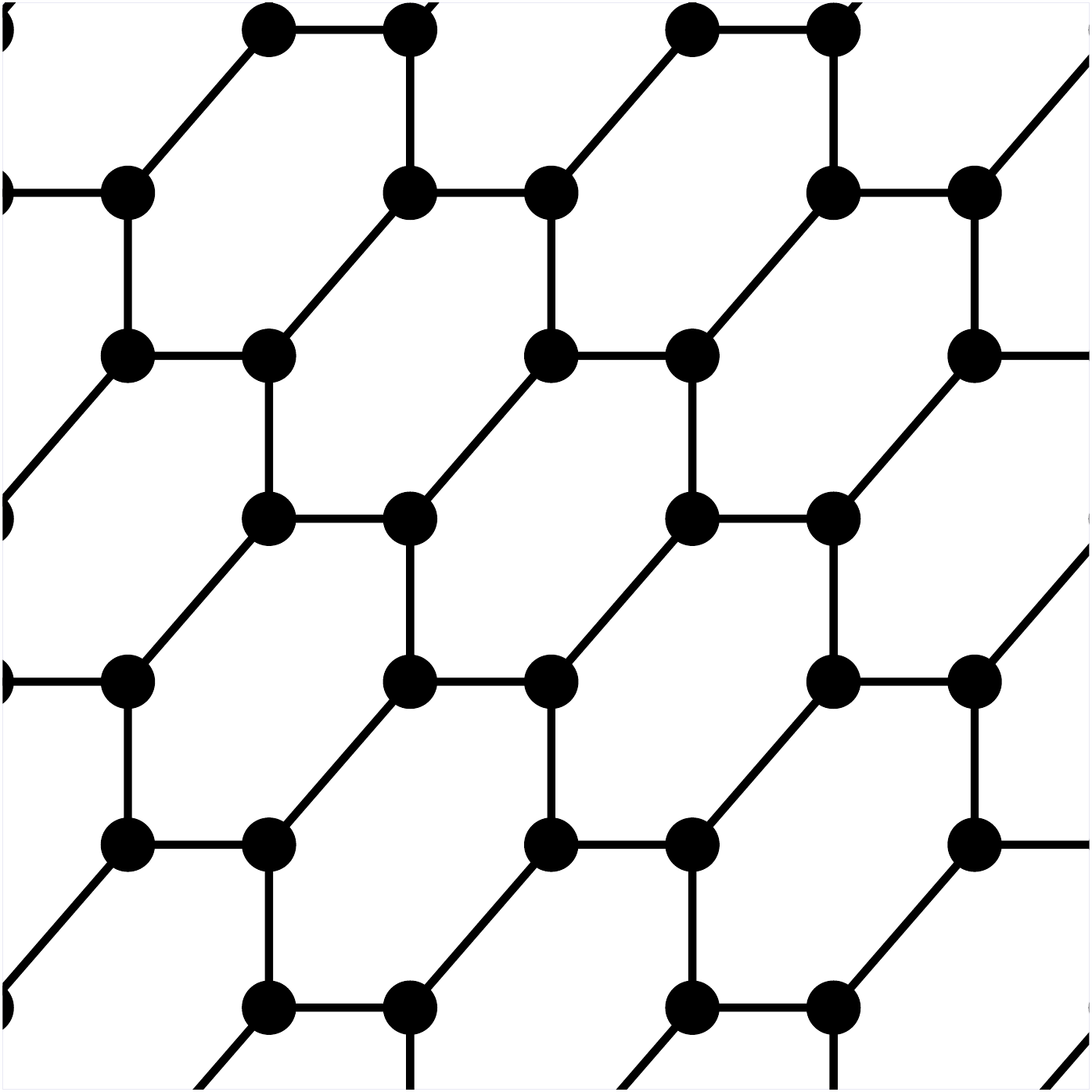}
  \end{tabular}
  \caption{
    Different realizations of the hexagonal lattice.
    These three lattices have the same topological structure.
  }
  \label{fig:hex-realization}
\end{figure}

\begin{definition}[Sunada \cite{MR3014418}]
  Let $X$ be a $d$-dimensional topological crystal with the base graph $X_0 = (V_0, E_0)$, 
  $G$ be an abelian group acting on $X$, 
  and $\Phi$ or $(\Phi, \rho)$ be a periodic realization of $X$, 
  where $\rho \colon G \longrightarrow GL(d, \R)$.
  The {\em energy} and the {\em normalized energy} of $\Phi$ are defined by
  \begin{equation}
    \label{eq:energy1}
    E(\Phi)
    =
    \sum_{e \in E_0} |\Phi(t(e)) - \Phi(o(e))|^2, 
  \end{equation}
  and
  \begin{equation}
    \label{eq:energy2}
    E(\Phi, \rho)
    =
    \Vol(\Gamma)^{2/d} \sum_{e \in E_0} |\Phi(t(e)) - \Phi(o(e))|^2, 
    \quad \Gamma = \rho(G), 
  \end{equation}
  respectively.
\end{definition}
The energy of $\Phi$ is a discrete analogue of the Dirichlet energy for smooth maps,
since $\Phi(t(e)) - \Phi(o(e))$ is a discretization of differential of smooth maps.

\begin{definition}[Sunada \cite{MR3014418}]
  \label{definition:harmonic}
  For a topological crystal $X$ with fixed lattice $\Gamma = \rho(G)$, 
  a critical point $\Phi$ of the energy $E$ is called a {\em harmonic realization} of $X$.
\end{definition}

In the followings, we abbreviate $\v{v} = \Phi(v)$ and $\v{e} = \Phi(e)$ for $v \in V$ and $e \in E$, 
as long as there are no misunderstandings.

\begin{proposition}[Sunada \cite{MR3014418}]
  For a topological crystal $X$, 
  a realization $\Phi$ is harmonic if and only if 
  \begin{equation}
    \label{eq:harmonic}
    \sum_{(u, v) \in (E_0)_v} (\v{u} - \v{v}) = \v{0}, 
    \quad \text{for all } v \in V_0.
  \end{equation}
\end{proposition}
That is to say, the sum of vectors creating edges emanating from each $v$ is zero, 
in other words, each vertex of $V$ satisfies the ``balancing condition''.
\begin{proof}
  First, we note that we have $E(\Phi) = E(\Phi, \rho)$ by fixing $\Gamma$ with $\Vol(\Gamma) = 1$.
  Let $\Phi_t \colon X \longrightarrow \R^d$ be a variation of $\Phi$ with $\Phi_0 = \Phi$.
  Differentiating $E(\Phi_t)$ by $t$ and using the ``integration by parts'', we may calculate as 
  \begin{displaymath}
    \begin{aligned}
      \frac{d}{dt} E(\Phi_t)
      &=
      \sum_{v \in V_0} \sum_{u \in V_v} \inner{\v{v}(t) - \v{u}(t)}{\v{v}'(t) - \v{u}'(t)}
      =
      2 \sum_{v \in V_0} \sum_{u \in V_v} \inner{\v{v}(t) - \v{u}(t)}{\v{u}'(t)}, 
    \end{aligned}
  \end{displaymath}
  where $\v{v}(t) = \Phi_t(v)$ and $\v{v}'(t) = \frac{d}{dt}\Phi_t(v)$.
  Therefore, we obtain 
  \begin{displaymath}
    \begin{aligned}
      \left.\frac{d}{dt} E(\Phi_t)\right|_{t=0}
      =
      2 \sum_{v \in V_0} \sum_{u \in V_v} \inner{\v{v} - \v{u}}{\v{x}_u}, 
      \quad
      \v{x}_u = \left.\Phi_t'(u)\right|_{t=0}, 
    \end{aligned}
  \end{displaymath}
  and get the result.
\end{proof}
\begin{remark}
  \mbox{}
  \begin{enumerate}
  \item 
    In the Definition \ref{definition:harmonic}, if we do not assume that the lattice $\Gamma$ is fixed, 
    then critical points admit $\Phi = 0$.
  \item 
    The equation (\ref{eq:harmonic}) is equivalent to
    a linear equation
    \begin{equation}
      \label{eq:harmonic1}
      \begin{aligned}
        - \deg(v) \v{v} + \sum_{u \in (V_0)_v} \v{u} &= 0, 
        \quad
        \text{for all } v \in V_0, 
        \\
        \Delta_X \Phi &= 0, 
      \end{aligned}
    \end{equation}
    where $\Delta_X = A_X - \diag(\deg(v))$ and is called the {\em Laplacian} of $X$.
    \par
    For a smooth map $u \colon \Omega \longrightarrow \R$ with $u|_{\partial \Omega} = 0$, 
    where $\Omega$ is a domain in $\R^N$, 
    the Dirichlet energy $E$ of $u$ is defined by 
    \begin{math}
      E(u) = \frac{1}{2}\int_{\Omega} |\nabla u|^2\,dV, 
    \end{math}
    and its Euler-Lagrange equation is $\Delta u = 0$ (the Laplace equation).
    This is a reason why the critical points of $E$ for topological crystals are called harmonic.
  \end{enumerate}
\end{remark}

\begin{example}
  Realizations (b) and (c) of Fig.~\ref{fig:hex-realization} are harmonic, 
  since each vertex $v \in V_0$ satisfies the balance condition (\ref{eq:harmonic1}), 
  but the realization (a) of Fig.~\ref{fig:hex-realization} is not harmonic.
  Hence, the harmonic condition (\ref{eq:harmonic}) or (\ref{eq:harmonic1}) are not suffice to 
  select (b) among three realizations in Fig.~\ref{fig:hex-realization}.
\end{example}

\begin{proposition}[Sunada \cite{MR3014418}]
  \label{claim:topological-crystal:harmonicuniq}
  Harmonic realizations of $X$ are unique up to affine transformations.
\end{proposition}
\begin{proof}
  First, note that the equation (\ref{eq:harmonic}) is invariant under affine transformations.
  Let $\{e_i\}_{i=1}^d$ be a $\Z$-basis of the abelian group $G$, which acts on $X$.
  Assume $\Phi_1(X)$ and $\Phi_2(X)$ are harmonic realizations of $X$ with respect to 
  lattice $\Gamma_1 = \rho_1(L)$ and $\Gamma_2 = \rho_2(L)$, respectively.
  Then, there exists an $A \in GL(d,\R)$ such that $\rho_1(g) = A\rho_2(g)$ for $g\in G$.
  Hence, we obtain that there exists $\v{b} \in \R^d$ such that $\Phi_1 = A\Phi_2 + \v{b}$.
\end{proof}

\begin{definition}[Sunada \cite{MR3014418}]
  \label{definition:standard}
  For a topological crystal $X$, 
  {\em standard} realizations of $X$ 
  are critical points among all realization $\Phi$ and $\Gamma$ with $\Vol(\Gamma) = 1$.
\end{definition}
A Standard realization are also called an {\em equilibrium placement} defined by 
Delgado-Friedrichs--O'Keeffe \cite{Delgado-Friedrichs-OKeeffe}.

\begin{theorem}[Kotani--Sunada \cite{MR1783793}, Sunada \cite{MR3014418}]
  \label{claim:standard}
  For any topological crystals $X$, 
  there exists the unique standard realization up to Euclidean motions.
\end{theorem}
Kotani--Sunada proved Theorem \ref{claim:standard} by using a theory of harmonic maps.
Eells--Sampson \cite{EellsSampson} proved the existence theorem of harmonic maps from compact Riemannian manifolds 
into non-positively curved Riemannian manifolds.
The energy (\ref{eq:energy1}) is the Dirichlet energy of maps from $1$-dimensional CW complex into a Euclidean space.
Hence, by Eells--Sampson's theorem, there exists a standard realization (an energy minimizing harmonic map) in each homotopy class.
Sunada also gave another proof of Theorem \ref{claim:standard} in his lecture note \cite{MR3014418}.
On the other hand, the existence of standard realizations can be also proved 
by showing the strong convexity of the energy (\ref{eq:energy1}).

\begin{theorem}[Sunada \cite{MR2375022, MR2902247, MR3014418}]
  \label{claim:standard1}
  For a $d$-dimensional topological crystal $X$, 
  a realization $\Phi$ is standard if and only if 
  \begin{align}
    \label{eq:harmonic:other}
    &\sum_{e \in E_0} \v{e} = \v{0}, 
    \\
    \label{eq:standard}
    &\sum_{e \in E_0} \inner{\v{x}}{\v{e}}\v{e} = c \v{x}
      \quad \text{for all } \v{x} \in \R^d 
      \text{ and for some } c > 0.
  \end{align}
\end{theorem}
\begin{proof}
  First, we define $T \colon \R^d \longrightarrow \R^d$ by 
  \begin{math}
    \displaystyle
    T \v{x} = \sum_{e \in E_0} \inner{\v{x}}{\v{e}}\v{e}.
  \end{math}
  Since 
  \begin{equation}
    \label{eq:topological:aa}
    \inner{T\v{x}}{\v{y}} 
    = \sum_{e \in E_0}  \inner{\v{x}}{\v{e}}\inner{\v{y}}{\v{e}}
    = \inner{\v{x}}{T \v{y}}, 
  \end{equation}
  we obtain that $T$ is symmetric.
  We would prove that $\Phi$ is a standard realization if and only if 
  there exists a positive constant $c > 0$ such that $T = c I$.
  \par
  On the other hand, for any symmetric matrix $T$ of size $d$ with positive eigenvalues, 
  there exists an orthogonal matrix $P$ such that
  $P^T T P = \diag(\lambda_1, \ldots, \lambda_d)$, 
  where $\lambda_j > 0$ are eigenvalues of $T$.
  The inequality of arithmetic and geometric means implies 
  \begin{displaymath}
    \trace T = \trace P^T T P \ge d (\det P^T T P)^{1/d} = d (\det T)^{1/d}, 
  \end{displaymath}
  and the equality holds if and only if $T = \lambda I_d$.
  \par
  Here, we write $E_0 = \{e_\alpha\}_{\alpha=1}^{\numberOf{E_0}}$, 
  and $\v{e}_i = (e_{\alpha 1}, \ldots, e_{\alpha d}) \in \R^d$.
  Since the equation (\ref{eq:standard}) is equivalent to 
  \begin{equation}
    \label{eq:standard1}
    \inner{T\v{x}}{\v{y}} = \sum_{e \in E_0} \inner{\v{e}}{\v{x}}\inner{\v{e}}{\v{y}}
    =
    c \inner{\v{x}}{\v{y}}, 
  \end{equation}
  taking an orthogonal basis $\{\v{x}_j\}_{j=1}^d$ of $\R^d$, and
  set $\v{x} = \v{x}_j$, and $\v{y} = \v{x}_k$, 
  we obtain 
  \begin{equation}
    \label{eq:topological:b}
    \sum_{\alpha=1}^{\numberOf{E_0}} e_{\alpha j} e_{\alpha k}
    =
    c \delta_{jk}, 
  \end{equation}
  and
  \begin{equation}
    \label{eq:topological:c}
    \Vol(\Gamma)^{2/d} E(\Phi, \rho)
    = 
    \sum_{e \in E_0} |\v{e}|^2
    =
    \sum_{\alpha=1}^{\numberOf{E_0}} \sum_{j=1}^d e_{\alpha j}^2
    =
    cd.
  \end{equation}
  \par
  Now, we assume that $(\Phi_1, \rho_1)$ is a standard realization of $X$
  and  $(\Phi_2, \rho_2)$ is a harmonic realization of $X$.
  By Proposition \ref{claim:topological-crystal:harmonicuniq}, 
  there exists an $A = (a_{ij}) \in GL(d, \R)$ and $\v{b} \in \R^d$ such that
  $\Phi_2 = A \Phi_1 + \v{b}$ and $\rho_1 = A \rho_2$.
  Then, we obtain 
  \begin{displaymath}
    \Vol(\Gamma_1) = \left|\det A\right|\Vol(\Gamma_2), 
    \quad
    f_{\alpha i} = \sum_{j=1}^d a_{ij} e_{\alpha j}, 
  \end{displaymath}
  and 
  \begin{displaymath}
    \begin{aligned}
      \Vol(\Gamma_2)^{2/d} E(\Phi_2, \rho_2)
      &=
      \sum_{i=1}^d \sum_{\alpha=1}^{\numberOf{E_0}} f_{\alpha i}^2
      =
      \sum_{i=1}^d \sum_{j, k=1}^d \sum_{\alpha=1}^{\numberOf{E_0}} a_{ij} a_{ik} e_{\alpha j} e_{\alpha_k}
      \\
      &=
      \sum_{i=1}^d \sum_{j, k=1}^d \sum_{\alpha=1}^{\numberOf{E_0}} a_{ij} a_{ik} \delta_{jk}
      =
      c \sum_{i=1}^d \sum_{j=1}^d a_{ij} a_{ij}
      =
      c \trace A^T A
      \\
      &\ge
      cd (\det A^T A)^{1/d}
      =
      cd (\det A)^{2/d}
      =
      cd (\Vol(\Gamma_2)/\Vol(\Gamma_1))^{2/d}
      \\
      &=
      cd \Vol(\Gamma_2) E(\Phi_1, \rho_1).
    \end{aligned}
  \end{displaymath}
  This implies that $E(\Phi_2, \rho_2) \ge E(\Phi_1, \rho_1)$
  if and only if the equation (\ref{eq:standard}) holds.
\end{proof}

\begin{theorem}[Sunada \cite{MR2375022, MR2902247, MR3014418}]
  \label{claim:standard:symmetry}
  Assume that $\Phi$ is a standard realization of a $d$-dimensional topological crystal.
  Then, each element $\sigma \in \Aut(X)$ extends as an element of $\Aut(\Phi(X)) \subset O(d) \ltimes \R^d$ 
  {\upshape(}Euclidean motion group of $\R^d${\upshape)}.
\end{theorem}

Theorem \ref{claim:standard:symmetry} means that
standard realizations, which are obtained by a variational principle, 
have maximum symmetry among all the realizations of a topological crystal.
\par
Recently, Kajigaya--Tanaka \cite{KajigayaTanaka} study the existence of discrete harmonic maps into Riemann surface of genus greater than one.

\begin{example}
  The realization (b) of Fig.~\ref{fig:hex-realization} is a standard realization of a hexagonal lattices, 
  whereas, the realization (c) of Fig.~\ref{fig:hex-realization} is not a standard.
\end{example}

\begin{example}
  Let $\triangle ABC$ be a triangle on a plane and $O$ be the barycenter of the triangle, 
  and consider a graph $G = (V, E)$ consisting $V = \{O, A, B, C\}$ and
  $E = \{(O, A), (O, B), (O, C)\}$.
  By a property of the barycenter of triangles, 
  we obtain $\widevec{OA} + \widevec{OB} + \widevec{OC} = \v{0}$.
  That is to say, the balancing condition (\ref{eq:harmonic}) holds for $O \in V$;
  however, the condition (\ref{eq:standard}) only holds for the case that $\triangle ABC$ is a regular triangle.
\end{example}

\begin{definition}[Sunada \cite{MR2375022, MR2902247, MR3014418}]
  A topological crystal $X$ of degree $n$ is called {\em strongly isotropic}, 
  if 
  there exists $g \in \Aut(X)$ such that $g(u) = v$ and $g(e_i) = f_{\sigma(i)}$, 
  for any $u, \, v \in V$, and for any permutation $\sigma \in {\mathfrak{S}}_n$, 
  where $E_u = \{e_i\}_{i=1}^n$ and $E_v = \{f_j\}_{j=1}^n$.
\end{definition}

\begin{theorem}[Sunada \cite{MR2375022}]
  $2$-dimensional strongly isotropic topological crystals are hexagonal lattices only.
  $3$-dimensional strongly isotropic topological crystals are diamond lattices and $K_4$ lattices 
  {\upshape(}and their mirror image{\upshape)} only.
\end{theorem}

\begin{remark}
  A square lattice does not have the strongly isotropic property.
  Let $X$ be a square lattice.
  Consider a vertex $v \in V$, $g = \id \in \Aut(X)$, and
  let $e_1,\,e_2,\,e_3,\,e_4 \in E_v$ be edges to north, west, south, and east.
  If $X$ has the strongly isotropic property, then any $\sigma \in {\mathfrak{S}}_4$, 
  $g(e_i) = e_{\sigma(i)}$ for $i = 1,2,3,4$.
  However, we exchange edges by the permutation $\sigma(1, 2, 3, 4) = (2, 1, 3, 4)$, 
  then, the graph structure could not preserved.
  Hence, a square lattice is not strongly isotropic.
\end{remark}
Graphenes and diamonds have nice physical properties (see Section \ref{sec:carbon}), 
and they are carbon structure of standard realizations of hexagonal and diamond lattices, which are strongly isotropic.
Hence, we may expect that $K_4$-carbons are also nice physical properties.

\begin{remark}
  Kotani--Sunada considered topological crystals from probabilistic motivations \cite{MR1743611, Kotani-Sunada-Proceedings, MR2039958, MR2375022}.
  A {\em random walk} on a graph $X = (V, E)$ is a stochastic process associated with $p \colon E \longrightarrow [0, 1]$ 
  satisfying $\sum_{e \in E_x} p(e) = 1$.
  In the case of $p(e) = 1/(\numberOf{E_x})$, the random walk is called {\em simple random walk}.
  The function $p$ is considered as {\em transition probability} from $o(e)$ to $t(e)$.
  Define
  \begin{math}
    p_X(n, x, y) 
    =
    \sum p(e_1)\cdots p(e_n), 
  \end{math}
  where summation over all paths with
  $e = e_1\cdots e_n$, $o(e_1) = x$, $t(e_n) = y\in V$, 
  is called $n$-step probability from $x$ to $y$.
  \par
  Let $X$ be a $d$-dimensional topological crystal, 
  Kotani--Sunada studied when the simple random walk on $X$ ``converges'' to a Brownian motion on $\R^d$ as the mesh of $X$ becomes finer, 
  and proved that if a realization of $X$ is standard, then there exists constants $C$ depending only on $X$
  such that 
  \begin{displaymath}
    \begin{aligned}
      &\frac{1}{\deg(y)}p_X(n, x, y) 
      \sim \frac{C}{(4\pi n)^{d/2}}
      \left(1 + c_1(x, y)n^{-1} + O(n^{-2})\right)
      \text{ as } n\to\infty, 
      \\
      &c_1(x, y)
      =
      -\frac{C}{4}|\v{x} - \v{y}|^2 + g(x) + g(y) + c, 
      \text{ for certain } g \text{ and } c, 
    \end{aligned}
  \end{displaymath}
  which means that 
  $p_X(n, x, y)$ ``converges'' to the heat kernel $p_{\R^d}(t, \v{x}, \v{y})$ as $n \uparrow \infty$.
\end{remark}
%%%%% END of document

%%%%% \input{standard.tex}
% -*- coding: utf-8 -*-
\subsection{Explicit constructions of standard realizations}
\label{sec:standard}
In this section, we demonstrate how to construct a standard realization from given base graph explicitly.
\par
Let $X_0 = (V_0, E_0)$ be a finite graph with $d = \rank H_1(X_0, \Z)$.
We define a natural inner product on $d$-dimensional vector space $C_1(X_0, \R)$ by 
\begin{displaymath}
  \inner{e_1}{e_2} 
  = 
  \left\{
    \begin{alignedat}{3}
      &1 &\quad &\text{if } e_1 = e_2, \\
      &-1 &\quad &\text{if } e_1 = \overline{e_2}, \\
      &0 &\quad &\text{otherwise}, 
    \end{alignedat}
  \right.
\end{displaymath}
for $e_1, \, e_2 \in E_0$. 
By using the inner product, we may identify $C_1(X_0, \R)$ to $\R^{\numberOf{E_0}}$, 
hence we may also identify $H_1(X_0, \R)$ to $\R^d$.
\par
Let $X = (V, E)$ be the maximum abelian covering of $X_0$,
and $\pi \colon X \longrightarrow X_0$ be the covering map.
Fix a vertex $v_0 \in V_0$, and 
define 
\begin{math}
  \Phi \colon X \longrightarrow H_1(X_0, \R)
\end{math}
by 
\begin{equation}
  \label{eq:standard:1}
  \Phi(v) = P(\pi(e_1)) + \cdots + P(\pi(e_n)), 
\end{equation}
where $e = e_1\cdots e_n$ is a path in $X$ connecting $v_0 = o(e_1)$ and $v = t(e_n)$, 
and 
\begin{math}
  P \colon C_1(X_0, \R) \longrightarrow H_1(X_0, \R)
\end{math}
is the orthogonal projection.
\begin{proposition}[Sunada \cite{MR3014418}]
  \label{claim:standard:h}
  The map $\Phi \colon X \longrightarrow H_1(X_0, \R)$ defined by {\upshape(}\ref{eq:standard:1}{\upshape)} is 
  a harmonic realization of $X$, namely,
  \begin{equation}
    \label{eq:standard:0}
    \sum_{e \in (E_0)_v} P(e) = 0 \in H_1(X_0, \R).
  \end{equation}
\end{proposition}
\begin{proof}
  First, we prove that 
  \begin{equation}
    \label{eq:standard:p1}
    \sum_{e \in (E_0)_v} \inner{e}{c} = 0
  \end{equation}
  for an arbitrary closed path $c = e_1\cdots e_n$ in $X_0$.
  If $c$ does not contain an edge whose origin or terminus is $v$, 
  then the equation (\ref{eq:standard:p1}) obviously holds.
  Let $e_j$ and $e_{j+1}$ be edges in $c$ satisfying $t(e_j) = o(e_{j+1}) = v$, 
  then $\inner{e}{e_j + e_{j+1}} = 1 - 1 = 0$.
  Hence, we obtain (\ref{eq:standard:p1}).
  \par
  The equality (\ref{eq:standard:p1}) implies that
  \begin{displaymath}
    \sum_{e \in (E_0)_v} e \in H_1(X_0, \R)^\perp, 
  \end{displaymath}
  since $H_1(X_0, \R)$ is generated by closed paths in $X_0$.
  Therefore, we obtain
  \begin{displaymath}
    0 = P\left(\sum_{e\in (E_0)_v} e\right)
    = \sum_{e \in (E_0)_v} P(e), 
  \end{displaymath}
  and hence we get (\ref{eq:standard:0}).
\end{proof}
\begin{proposition}[Sunada \cite{MR3014418}]
  \label{claim:standard:s}
  The map $\Phi \colon X \longrightarrow H_1(X_0, \R)$ defined by {\upshape(}\ref{eq:standard:1}{\upshape)} is 
  a standard realization of $X$, namely, 
  there exists a constant $c > 0$ such that 
  \begin{equation}
    \label{eq:standard:3}
    \sum_{e \in E_0} \left(\inner{P(e)}{x}\right)^2 = c |x|^2, 
    \quad 
    x \in H_1(X_0, \R).
  \end{equation}
\end{proposition}
\begin{proof}
  Since the set of oriented edges $E_0^o := \{e_i\}_{i=1}^n$ is an orthonormal basis of $C_1(X_0, \R)$, 
  we obtain 
  \begin{displaymath}
    \sum_{e \in E_0^o} \left(\inner{P(e)}{x}\right)^2
    = 
    \sum_{e \in E_0^o} \left(\inner{e}{x}\right)^2
    =
    |x|^2, 
  \end{displaymath}
  and
  \begin{displaymath}
    \sum_{e \in E_0} \left(\inner{P(e)}{x}\right)^2 
    = 
    \sum_{e \in E_0^o} \left(\inner{P(e)}{x}\right)^2 
    +
    \sum_{\overline{e} \in E_0^o} \left(\inner{P(e)}{x}\right)^2 
    =
    2 \sum_{e \in E_0^o} \left(\inner{P(e)}{x}\right)^2, 
  \end{displaymath}
  hence we get (\ref{eq:standard:3})
\end{proof}

By the above arguments, the realization $\Phi$ of $X_0$ is into $H_1(X_0, \R)/H_1(X_0, \Z)$ 
with the period lattice $\Gamma$.
The torus $H_1(X_0, \R)/H_1(X_0, \Z)$ is called an {\em Albanese torus}.
Therefore, to calculate explicit coordinates of standard realizations, 
we should compute correspondences between the Albanese torus and $\R^d/\Z^d$.

%%%%% \input{max_abelian.tex}
% -*- coding: utf-8 -*-
\subsubsection{Explicit algorithm in cases of maximum abelian coverings}
\label{sec:explicit:max}
Now we explain explicit algorithm to obtain a standard realization of 
a $d$-dimensional topological crystal $X$, 
which is a maximum abelian covering of $X_0 = (V_0, E_0)$.
This method is followed by Sunada \cite{MR3014418} and Naito \cite{MR1500145}.
In the followings, set $b = \rank H_1(X_0, \Z)$.

\paragraph*{Step 1}
First, compute a spanning tree $X_1 = (V_0, E_1)$ of $X_0$ by Kruskal's algorithm, 
and set $E_0\setminus E_1 = \{e_i\}_{i=1}^b$ 
and $E_1 = \{e_i\}_{i=b+1}^{\numberOf{E}}$.
Then, we may select a $\Z$-basis $\{{\alpha}_i\}_{i=1}^b$ of $H_1(X_0, \Z)$ as follows.
For each edge $e_i \in E_0 \setminus E_1$, we may find a path $p_i$ in $E_1$ such that
$o(p_i) = t(e_i)$ and $t(p_i) = o(e_i)$.
The path $p_i e_i \in E_0$ is a closed path in $E_0$, 
and hence by Propositions \ref{claim:graph:homology1} and \ref{claim:graph:homology2}, 
we may set ${\alpha}_i = [p_i e_i]$.

\paragraph*{Step 2}
Since $\{{\alpha}_i\}_{i=1}^b$ is a $\Z$-basis of $H_1(X_0, \Z)$, 
for each edge $e \in E_0$ there exists $a_i(e) \in \R$ such that
\begin{equation}
  \label{eq:standard:explicit:1:1}
  P(e) = \sum_{i=1}^b a_i(e) {\alpha}_i \in H_1(X_0, \R).
\end{equation}
Since $e \in C_1(X_0, \R)$ and $P$ is the orthogonal projection from $C_1(X_0, \R)$ onto $H_1(X_0, \R)$.
$P(e)$ satisfies 
\begin{equation}
  \label{eq:standard:explicit:1:2}
  \inner{P(e) - e}{\alpha_j} = 0, \text{ for any } j.
\end{equation}
Substituting (\ref{eq:standard:explicit:1:1}) into (\ref{eq:standard:explicit:1:2}), 
we obtain
\begin{equation}
  \label{eq:standard:explicit:1:3}
  \sum_{i=1}^b a_i(e) \inner{\alpha_i}{\alpha_j}
  =
  \inner{e}{\alpha_j}.
\end{equation}
Set $A = (\inner{\alpha_i}{\alpha_j}) \in GL(b, \R)$, 
and $\v{a}(e) = (a_i(e))^T$, $\v{b}(e) = (\inner{e}{\alpha_i})^T \in \R^b$, 
then (\ref{eq:standard:explicit:1:3}) is written as
\begin{equation}
  \label{eq:standard:explicit:1:4}
  \v{a}(e) = A^{-1} \v{b}(e).
\end{equation}
\par
We get $\v{a}(e)$ for each $e \in E_0$, 
then we obtain the realization 
\begin{equation}
  \label{eq:standard:explicit:1:3.5}
  P(e) = \v{e} = \sum_{i=1}^b a_i(e) \alpha_i, \text{ in } H_1(X_0, \R).
\end{equation}
On the other hand, we easily calculate $\v{b}(e)$ and $A$, 
since $e$ and $\alpha_i$ are given by linear combinations of $\{e_i\}_{i=1}^b$ and $\{e_i\}_{i=b=1}^{\numberOf{E}}$.
Therefore, by (\ref{eq:standard:explicit:1:4}), we obtain $\v{a}(e)$.
We remark that the matrix $A$ is the Gram matrix of the basis $\{\alpha_i\}_{i=1}^b$.
Taking an orthonormal basis $\{\v{x}_i\}_{i=1}^b$ of $H_1(X_0, \R)$ and 
write
\begin{equation}
  \label{eq:standard:explicit:1:5}
  \alpha_i = \sum_{j=1}^b \beta_{ij} \v{x}_j, 
\end{equation}
then we obtain the expression of the realization in the Cartesian coordinates of $H_1(X_0, \R) \equiv \R^b$ as
\begin{equation}
  \label{eq:standard:explicit:1:6}
  \v{e}
  = 
  \sum_{i=1}^b a_i(e) \alpha_i 
  =
  \sum_{i=1}^b \left(\sum_{j=1}^b a_i(e) \beta_{ij}\right) \v{x}_j
  \text{ for } e \in E_0.
\end{equation}
To obtain the relation (\ref{eq:standard:explicit:1:5}), 
we may use the Cholesky decomposition.
The Cholesky decomposition, which is a special case of $LU$ decomposition, 
gives us the decomposition 
$A = X^T X$ for any positive definite symmetric matrix $A$ by 
an upper triangular matrix $X$ (see for example \cite{MatrixCookbook}).

\paragraph*{Step 3}
Fix a vertex $v_0 \in V_0$, and set $\v{v_0} = \v{0}$ (origin of $\R^b$).
For each vertex $v_j \in V_0$, 
we find the shortest path $e = e_{j1} \cdots e_{jk} \in E_1$ with $o(e_{j1}) = v_0$ and $t(e_{jk}) = v_j$, 
which is a shortest path in the spanning tree finding in Step 1 connecting $v_0$ and $v_j$.
By using (\ref{eq:standard:explicit:1:3.5}), 
we obtain 
\begin{equation}
  \label{eq:standard:explicit:1:7}
  \v{v}_j = 
  \sum_{i=1}^k \v{e}_{ji}
  =
  \sum_{i=1}^k \sum_{k=1}^b a_k(e_{ji}) \alpha_k.
\end{equation}
In the above, we realize edges in the spanning tree.
Hence, to complete calculation, we compute realizations of edges which are not contained in the spanning tree.
For each $e_\ell \in E_0 \setminus E_1$, 
we define $\v{w}_\ell \in \R^b$ by
\begin{equation}
  \label{eq:standard:explicit:1:8}
  \v{w}_\ell = 
  \v{v}(e_\ell) + \v{e}_\ell, 
\end{equation}
where $v(e_\ell) = o(e_\ell)$.
\par
Vertices $\{\v{v}_j\}_{j=1}^{\numberOf{V}} \sqcup \{\v{w}_\ell\}_{\ell=1}^b \subset \R^b$
(or edges $\{\v{e}_i\}_{i=1}^{\numberOf{E_0}}$) 
with the period lattice $\{\alpha_i\}_{i=1}^b$ 
give us a standard realization of $X$ with period lattice $\Gamma$.
The set of realizations of edges $\{\v{e}_j\}_{j=1}^{\numberOf{E}}$ is called the {\em building block}.
In other words, Information of adjacency of the graph and the building block give us a standard realization.

\begin{remark}
  Dijkstra's algorithm gives us shortest paths from a vertex to any other vertices 
  within $O(\numberOf{E} + \numberOf{V}\log\numberOf{V})$ (see for example \cite{Aho}).
\end{remark}

\begin{example}[{\bfseries Square lattices in $\R^2$}, Fig.~\ref{fig:standard} (a), Sunada {\cite[Section 8.3]{MR3014418}}]
  \label{example:standard:square}
  The base graph $X_0 = (V_0, E_0)$ of square lattices in $\R^2$ is the $2$-bouquet graph (Fig.~\ref{fig:basegraph:cubic} (a)), 
  and $\rank H_1(X_0, \R) = 2$.
  Write $V_0 = \{v_0\}$ and $E_0 = \{e_1, e_2\}$, as in Fig.~\ref{fig:basegraph:cubic} (a), 
  then a spanning tree of $X_0$ is $X_1 = (V_0, \{\emptyset\})$, namely, $E_1 = \{\emptyset\}$.
  Hence, we may take $\alpha_1 = e_1$ and $\alpha_2 = e_2$ as a $\Z$-basis of $H_1(X_0, \Z)$, 
  and obtain
  \begin{displaymath}
    \begin{aligned}
      A 
      &= 
      \begin{bmatrix}
        \inner{\alpha_1}{\alpha_1} & \inner{\alpha_1}{\alpha_2} \\
        \inner{\alpha_2}{\alpha_1} & \inner{\alpha_2}{\alpha_2} \\
      \end{bmatrix}
      =
      \begin{bmatrix}
        1 & 0 \\ 0 & 1
      \end{bmatrix}, 
      \quad
      A^{-1} = A, 
      \\
      \begin{bmatrix}
        \v{b}(e_1) & \v{b}(e_2)
      \end{bmatrix}
      &=
      \begin{bmatrix}
        \inner{e_1}{\alpha_1} & \inner{e_2}{\alpha_1} \\
        \inner{e_1}{\alpha_2} & \inner{e_2}{\alpha_2} 
      \end{bmatrix}
      =
      \begin{bmatrix}
        1 & 0 \\
        0 & 1
      \end{bmatrix},
      \\
      \begin{bmatrix}
        \v{a}(e_1)  & \v{a}(e_2)
      \end{bmatrix}
      &= A
      \begin{bmatrix}
        \v{b}(e_1) & \v{b}(e_2)
      \end{bmatrix}
      =
      \begin{bmatrix}
        1 & 0 \\
        0 & 1
      \end{bmatrix}.
    \end{aligned}
  \end{displaymath}
  On the other hand, 
  the shortest paths from $\v{v}_0$ to other vertices are 
  \begin{displaymath}
    \shortestpath(\v{v}_0, \v{w}_i) = (\v{v}_0 \v{w}_i), 
    \quad
    i = 1, \, 2.
  \end{displaymath}
  Since $\{\alpha_i\}_{i=1}^2$ is orthonormal, hence, we obtain 
  \begin{displaymath}
    \v{v}_0 
    = 
    \begin{bmatrix}
      0 \\ 0
    \end{bmatrix}, 
    \quad
    \v{w}_1 
    = \v{v}_0 + a(e_1)
    =
    \begin{bmatrix}
      1 \\ 0
    \end{bmatrix}, 
    \quad
    \v{w}_2
    = \v{v}_0 + a(e_2)
    =
    \begin{bmatrix}
      0 \\ 1
    \end{bmatrix}, \\
  \end{displaymath}
  and the period lattice is 
  \begin{displaymath}
    \begin{bmatrix}
      \v{x}_1 & \v{x}_2
    \end{bmatrix}
    =
    \begin{bmatrix}
      1 & 0 \\ 0 & 1
    \end{bmatrix}.
  \end{displaymath}
  The above datas allows us to write figure in Fig.~\ref{fig:standard} (a).
\end{example}

\begin{example}[{\bfseries Hyper-cubic lattice in $\R^n$}, Sunada {\cite[Section 8.3]{MR3014418}}]
  \label{example:standard:cubic}
  A generalization of Example \ref{example:standard:square} is 
  hyper-cubic lattices in $\R^n$.
  In case of $n = 3$, it is called cubic lattices.
  The base graph $X_0 = (V_0, E_0)$ of hyper-cubic lattices is the $n$-bouquet graph (Fig.~\ref{fig:basegraph:cubic} (b)),  
  namely $V_0 = \{v\}$, $E_0 = \{e_i\}_{i=1}^n$, as in Fig.~\ref{fig:basegraph:cubic} (b). 
  Since a spanning tree of $X_0$ is $X_1 = (V_0, \{\emptyset\})$, 
  we may take an orthonormal $\Z$-basis of $H_1(X_0, \Z)$
  by $\{\alpha_i\}_{i=1}^n$, where $\alpha_i = e_i$.
  By similar calculations, we obtain 
  \begin{displaymath}
    A 
    = 
    A^{-1}
    =
    \begin{bmatrix}
      a(e_i)
    \end{bmatrix}
    = 
    \begin{bmatrix}
      b(e_i)
    \end{bmatrix}
    =
    E_n
    \quad
    (\text{the identity matrix of size $n$}).
  \end{displaymath}
  Hence, we obtain 
  \begin{displaymath}
    \v{v}_0 
    = 
    \v{0}, 
    \quad
    \v{w}_i
    =
    \v{x}_i 
    \quad
    (\text{standard $i$-th unit vector of $\R^n$})
    \quad i = 1, \ldots, n, 
  \end{displaymath}
  and the period lattice is 
  \begin{math}
    \begin{bmatrix}
      \v{x}_i
    \end{bmatrix}
    =
    E_n.
  \end{math}
  A standard realization of hyper-cubic lattices is
  an orthonormal lattice in $\R^n$.
\end{example}

\begin{figure}[htbp]
  \centering
  \begin{tabular}{ccc}
    \multicolumn{1}{l}{(a)}
    &\mbox{}
    &\multicolumn{1}{l}{(b)}
    \\
    \includegraphics[bb=0 0 51 175,scale=0.75]{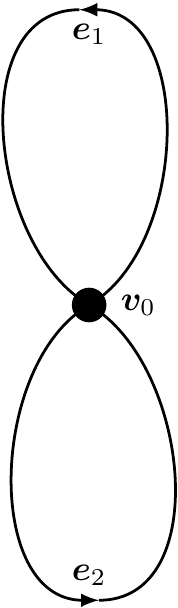}
    &\mbox{}
    &\includegraphics[bb=0 0 156 143,scale=0.75]{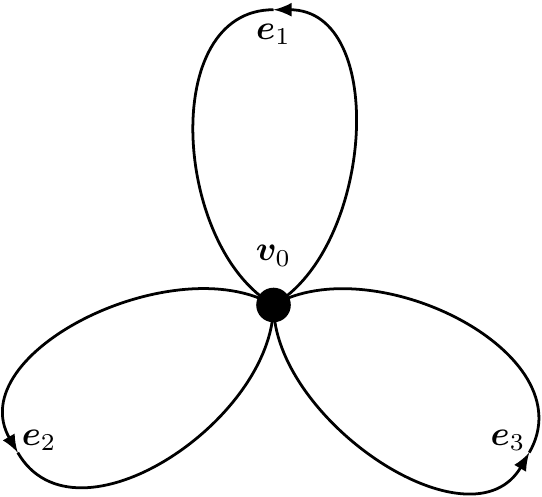}
  \end{tabular}
  \caption{
    (a) The $2$-bouquet graph, which is the base graph of square lattices, 
    (b) the $3$-bouquet graph, which is the base graph of cubic lattices.
  }
  \label{fig:basegraph:cubic}
\end{figure}

\begin{example}[{\bfseries Hexagonal lattices}, Fig.~\ref{fig:standard} (c), Sunada {\cite[Section 8.3]{MR3014418}}]
  \label{example:standard:hexagonal}
  The base graph $X_0 = (V_0, E_0)$ of hexagonal lattices in $\R^2$ is 
  the graph with two vertices and three edges connecting both vertices (Fig.~\ref{fig:basegraph:hex} (a)), 
  and $\rank H_1(X_0, \R) = 2$.
  Write $V_0 = \{v_0, v_1\}$ and $E_0 = \{e_1, e_2, e_3\}$ as in Fig.~\ref{fig:basegraph:hex}, 
  then a spanning tree of $X_0$ is $X_1 = (V_0, \{e_3\})$.
  Hence, we may take $\alpha_1 = e_1 - e_3$ and $\alpha_2 = e_2 - e_3$ as a $\Z$-basis of $H_1(X_0, \Z)$, 
  and obtain
  \begin{displaymath}
    \begin{aligned}
      A 
      &= 
      \begin{bmatrix}
        2 & 1 \\ 1 & 2
      \end{bmatrix}, 
      \quad
      A^{-1}
      =
      \frac{1}{3}
      \begin{bmatrix}
        2 & -1 \\ -1 & 2
      \end{bmatrix}
      \\
      \begin{bmatrix}
        \v{b}(e_1) & \v{b}(e_2) & \v{b}(e_3)
      \end{bmatrix}
      &=
      \begin{bmatrix}
        1 & 0 & -1 \\
        0 & 1 & -1 
      \end{bmatrix},
      \\
      \begin{bmatrix}
        \v{a}(e_1)  & \v{a}(e_2) & \v{a}(e_3)
      \end{bmatrix}
      &= 
      \frac{1}{3}
      \begin{bmatrix}
        2 & -1 & -1 \\
        -1 & 2 & -1
      \end{bmatrix}.
    \end{aligned}
  \end{displaymath}
  The shortest paths from $\v{v}_0$ to other vertices are 
  \begin{displaymath}
    \shortestpath(\v{v}_0, \v{w}_i) = (\v{v}_0 \v{w}_i)
    \quad
    i = 1, \, 2.
  \end{displaymath}
  Since $\{\alpha_i\}_{i=1}^2$ is not orthonormal, we choice the basis as 
  \begin{displaymath}
    \begin{bmatrix}
      \alpha_1 \\ \alpha_2 
    \end{bmatrix}
    =
    \begin{bmatrix}
      \sqrt{2} & 0 \\
      1/\sqrt{2} & \sqrt{3/2}
    \end{bmatrix}
    =: X, 
  \end{displaymath}
  then we obtain 
  \begin{displaymath}
    \begin{aligned}
      &\v{v}_0 
      = 
      \v{0}, 
      \\
      &\v{w}_1 
      =
      \frac{2}{3} \alpha_1 - \frac{1}{3}\alpha_2
      =
      \begin{bmatrix}
        1/\sqrt{2} \\ -1/\sqrt{6}
      \end{bmatrix}, 
      \\
      &\v{w}_2
      =
      -\frac{1}{3}\alpha_1 + \frac{2}{3}\alpha_2 
      =
      \begin{bmatrix}
        0 \\ \sqrt{2/3}
      \end{bmatrix}, 
      \\
      &\v{w}_3
      =
      -\frac{1}{3} \alpha_1 - \frac{1}{3}\alpha_2
      =
      \begin{bmatrix}
        -1/\sqrt{2} \\ -1/\sqrt{6}
      \end{bmatrix}, 
    \end{aligned}
  \end{displaymath}
  and the period lattice is 
  \begin{displaymath}
    \begin{bmatrix}
      \v{x}_1 & \v{x}_2
    \end{bmatrix}
    =
    X.
  \end{displaymath}
  Note that $\inner{\v{w}_i}{\v{w}_j} = (-1/2)|\v{w}_i|\,|\v{w}_j|$ ($i \not= j$) are satisfied.
  The above datas allow us to write figure in Fig.~\ref{fig:standard} (c).
\end{example}

\begin{figure}[htbp]
  \centering
  \begin{tabular}{ccc}
    \multicolumn{1}{l}{(a)}
    &\mbox{}
    &\multicolumn{1}{l}{(b)}
    \\
    \includegraphics[bb=0 0 109 140,scale=0.75]{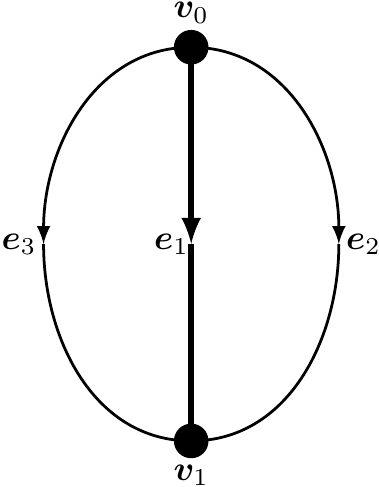}
    &\mbox{}
    &\includegraphics[bb=0 0 109 140,scale=0.75]{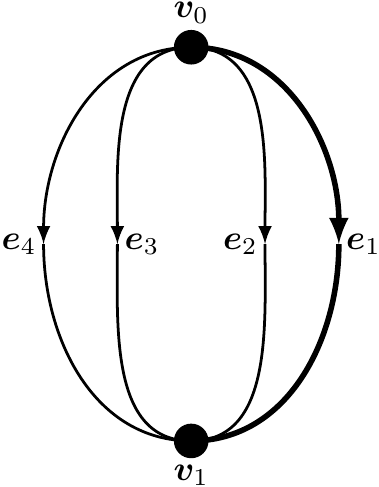}
  \end{tabular}
  \caption{
    (a) The base graph of hexagonal lattices, 
    (b) the base graph of diamond lattices.
    Thick edges consist a spanning tree of them.
  }
  \label{fig:basegraph:hex}
\end{figure}

\begin{example}[{\bfseries Diamond lattices}, Sunada {\cite[Section 8.3]{MR3014418}}]
  \label{example:standard:diamond}
  The base graph $X_0 = (V_0, E_0)$ of a diamond lattices in $\R^3$ is 
  the graph with two vertices and four edges connecting both vertices, 
  and $\rank H_1(X_0, \R) = 3$ (see Fig.~\ref{fig:diamondasstandard}).
  Write $V_0 = \{v_0, v_1\}$ and $E_0 = \{e_1, e_2, e_3, e_4\}$, where $e_i = (v_0, v_1)$, 
  then a spanning tree of $X_0$ is $X_1 = (V_0, \{e_4\})$.
  Hence, we may take $\alpha_i = e_i - e_4$ ($i = 1,\,2,\,3$) as a $\Z$-basis of $H_1(X_0, \Z)$, 
  and obtain
  \begin{displaymath}
    \begin{aligned}
      A 
      &= 
      \begin{bmatrix}
        2 & 1 & 1 \\ 1 & 2 & 1 \\ 1 & 1 & 2
      \end{bmatrix}, 
      \quad
      A^{-1}
      =
      \frac{1}{4}
      \begin{bmatrix}
        3 & -1 & -1 \\ -1 & 3 & -1 \\ -1 & -1 & 3
      \end{bmatrix}
      \\
      \begin{bmatrix}
        \v{b}(e_1) & \v{b}(e_2) & \v{b}(e_3) & \v{b}(e_4)
      \end{bmatrix}
      &=
      \begin{bmatrix}
        1 & 0 & 0 & -1 \\
        0 & 1 & 0 & -1 \\
        0 & 0 & 1 & -1
      \end{bmatrix},
      \\
      \begin{bmatrix}
        \v{a}(e_1)  & \v{a}(e_2) & \v{a}(e_3) & \v{a}(e_4)
      \end{bmatrix}
      &= 
      \frac{1}{4}
      \begin{bmatrix}
        3 & -1 & -1 & -1 \\
        -1 & 3 & -1 & -1 \\
        -1 & -1 & 3 & -1 \\
      \end{bmatrix}, 
      \\
      \begin{bmatrix}
        \alpha_1 \\ \alpha_2 \\ \alpha_3
      \end{bmatrix}
      &=
      \begin{bmatrix}
        \sqrt{2} & 0 & 0 \\
        1/\sqrt{2} & \sqrt{3/2} & 0 \\
        1/\sqrt{2} & 1/\sqrt{6} & 2/\sqrt{3} \\
      \end{bmatrix}
      =: X.
    \end{aligned}
  \end{displaymath}
  The shortest paths from $\v{v}_0$ to other vertices are 
  \begin{displaymath}
    \shortestpath(\v{v}_0, \v{w}_i) = (\v{v}_0 \v{w}_i)
    \quad
    i = 1, \, 2, \, 3.
  \end{displaymath}
  Hence, we obtain 
  \begin{displaymath}
    \begin{aligned}
      &
      \v{v}_0 
      = 
      \v{0}, 
      \\
      &\v{w}_1 
      =
      \frac{3}{4} \alpha_1 - \frac{1}{4}\alpha_2 - \frac{1}{4}\alpha_3
      =
      \begin{bmatrix}
        1/\sqrt{2} \\ -1/\sqrt{6} \\ -1/(2\sqrt{3})
      \end{bmatrix}, 
      \\
      &\v{w}_2
      =
      -\frac{1}{4}\alpha_1 + \frac{3}{4}\alpha_2 - \frac{1}{4}\alpha_3
      =
      \begin{bmatrix}
        0 \\ \sqrt{2/3} \\ -1/(2\sqrt{3})
      \end{bmatrix}, 
      \\
      &\v{w}_3
      =
      -\frac{1}{4}\alpha_1 - \frac{1}{4}\alpha_2 + \frac{3}{4}\alpha_3
      =
      \begin{bmatrix}
        0 \\ 0 \\ 2/\sqrt{3}
      \end{bmatrix}, 
      \\
      &\v{w}_4
      =
      -\frac{1}{4} \alpha_1 - \frac{1}{4}\alpha_2 - \frac{1}{4}\alpha_3
      =
      \begin{bmatrix}
        -1/\sqrt{2} \\ -1/\sqrt{6} \\ -1/(2\sqrt{3})
      \end{bmatrix}.
    \end{aligned}
  \end{displaymath}
  The period lattice is 
  \begin{displaymath}
    \begin{bmatrix}
      \v{x}_1 & \v{x}_2 & \v{x}_3
    \end{bmatrix}
    =
    X.
  \end{displaymath}
  Note that $\inner{\v{w}_i}{\v{w}_j} = (-1/3)|\v{w}_i|\,|\v{w}_j|$ ($i \not= j$) are satisfied.
\end{example}

\begin{example}[{\bfseries Gyroid lattices ($K_4$ lattices)}, Sunada {\cite[Section 8.3]{MR3014418}}]
  \label{example:standard:k4}
  The base graph $X_0 = (V_0, E_0)$ of a gyroid lattices in $\R^3$ is 
  the $K_4$ graph, which is the complete graph of four vertices, 
  and $\rank H_1(X_0, \R) = 3$.
  Write $V_0 = \{v_i\}_{i=1}^4$ and $E_0 = \{e_i\}_{i=1}^4$ as in Fig.~\ref{fig:basegraph:k4} (a), 
  and take a spanning tree $X_1$ of $X_0$ as in Fig.~\ref{fig:basegraph:k4} (b).
  Hence, we may take 
  \begin{displaymath}
    \alpha_1 = e_1 + e_4 - e_2, 
    \quad
    \alpha_2 = e_2 + e_5 - e_3, 
    \quad
    \alpha_3 = e_3 + e_6 - e_1
  \end{displaymath}
  as a $\Z$-basis of $H_1(X_0, \Z)$, 
  and obtain
  \begin{displaymath}
    \begin{aligned}
      A 
      &= 
      \begin{bmatrix}
        3 & -1 & -1 \\ -1 & 3 & -1 \\ -1 & -1 & 3
      \end{bmatrix}, 
      \quad
      A^{-1}
      =
      \frac{1}{4}
      \begin{bmatrix}
        2 & 1 & 1 \\ 1 & 2 & 1 \\ 1 & 1 & 2
      \end{bmatrix}
      \\
      \v{b}
      &=
      \begin{bmatrix}
        1 & -1 & 0 & 1 & 0 & 0 \\
        0 & 1 & -1 & 0 & 1 & 0 \\
        -1 & 0 & 1 & 0 & 0 & 1 \\
      \end{bmatrix},
      \quad
      \v{a} 
      = 
      \frac{1}{4}
      \begin{bmatrix}
        1 & -1 & 0 & 2 & 1 & 1 \\
        0 & 1 & -1 & 1 & 2 & 1 \\
        -1 & 0 & 1 & 1 & 1 & 2 \\
      \end{bmatrix},
      \\
      \begin{bmatrix}
        \alpha_1 \\ \alpha_2 \\ \alpha_3
      \end{bmatrix}
      &=
      \begin{bmatrix}
        \sqrt{3} & 0 & 0 \\
        -1/\sqrt{3} & 2\sqrt{2/3} & 0 \\
        -1/\sqrt{3} & -\sqrt{2/3} & \sqrt{2} \\
      \end{bmatrix}
      =: X, 
    \end{aligned}
  \end{displaymath}
  Let $\{w_i\}_{i=1}^3$ be as in Fig.~\ref{fig:basegraph:k4} (b), 
  then the shortest paths from $\v{v}_0$ to other vertices are 
  \begin{displaymath}
    \shortestpath(\v{v}_0, \v{v}_i) = (\v{v}_0 \v{v}_i), 
    \quad
    \shortestpath(\v{v}_0, \v{w}_i) = (\v{v}_0 \v{v}_i) (\v{v}_i \v{w}_i), 
    \quad
    i = 1, \, 2, \, 3.
  \end{displaymath}
  Hence, we obtain 
  \begin{displaymath}
    \begin{aligned}
      \v{v}_0 
      &= 
      \v{0}, 
      \\
      \v{v}_1 
      &=
      \frac{1}{4} \alpha_1 - \frac{1}{4}\alpha_3
      =
      \begin{bmatrix}
        1//\sqrt{3}\\
        1/(2\sqrt{6})\\
        -1/(2\sqrt{2})
      \end{bmatrix}, 
      &\quad
      \v{w}_1
      &=
      \v{v}_1 + \frac{1}{2}\alpha_1 + \frac{1}{4}\alpha_2 + \frac{1}{4}\alpha_3
      =
      \begin{bmatrix}
        2/\sqrt{3}\\
        1/\sqrt{6}\\
        0
      \end{bmatrix}
      \\
      \v{v}_2
      &=
      \frac{1}{4}\alpha_2 - \frac{1}{4}\alpha_1
      =
      \begin{bmatrix}
        -1/\sqrt{3}\\
        1/\sqrt{6}\\
        0
      \end{bmatrix}, 
      &\quad
      \v{w}_2
      &=
      \v{v}_2 + \frac{1}{4}\alpha_1 + \frac{1}{2}\alpha_2 + \frac{1}{4}\alpha_3
      =
      \begin{bmatrix}
        -1/\sqrt{3}\\
        5/(2\sqrt{6})\\
        1/(2\sqrt{2})
      \end{bmatrix}
      \\
      \v{v}_3
      &=
      \frac{1}{4}\alpha_3 - \frac{1}{4}\alpha_2
      =
      \begin{bmatrix}
        0\\
        -(1/2)\sqrt{3/2}\\
        1/(2\sqrt{2})
      \end{bmatrix}, 
      &\quad
      \v{w}_3
      &=
      \v{v}_3 + \frac{1}{4}\alpha_1 + \frac{1}{4}\alpha_2 + \frac{1}{2}\alpha_3
      =
      \begin{bmatrix}
        0\\
        -(1/2)\sqrt{3/2}\\
        3/(2\sqrt{2})
      \end{bmatrix}. 
    \end{aligned}
  \end{displaymath}
  A gyroid lattice is called a $K_4$ lattice since its base graph is $K_4$.
  It is also called a Laves' graph of girth ten, a $(10,3)$-$a$ network, and a diamond twin.
  The minimum length of closed path (without backtracking paths) is called the {\em girth} of the graph.
  The girth of a gyroid lattice is $10$ (see Fig.~\ref{fig:k4ring}), 
  and hence, it is called $(10,3)$-$a$ network.
\end{example}

\begin{remark}
  \label{remark:k4coordinates}
  We can also take coordinates which all vertices have rational numbers.
  Taking
  \begin{displaymath}
    \begin{bmatrix}
      \alpha_1 \\ \alpha_2 \\ \alpha_3
    \end{bmatrix}
    =
    \begin{bmatrix}
      -1 & 1 & -1 \\
      -1 & -1 & 1 \\
      1  & -1 & -1 \\
    \end{bmatrix},
  \end{displaymath}
  then 
  \begin{displaymath}
    \begin{alignedat}{5}
      \v{v}_1 
      &=
      \frac{1}{2}
      \begin{bmatrix}
        0 \\ -1 \\ 1
      \end{bmatrix}, 
      &\quad
      \v{v}_2
      &=
      \frac{1}{2}
      \begin{bmatrix}
        1 \\ 0 \\ -1
      \end{bmatrix}, 
      &\quad
      \v{v}_3
      &=
      \begin{bmatrix}
        -1 \\ 1 \\ 0
      \end{bmatrix},
      &\quad
      \v{w}_1
      &=
      \frac{1}{2}
      \begin{bmatrix}
        -1 \\ -2 \\ 1
      \end{bmatrix}, 
      \quad
      \v{w}_2
      &=
      \begin{bmatrix}
        1 \\ -1 \\ -2
      \end{bmatrix}, 
      \quad
      \v{w}_3
      &=
      \begin{bmatrix}
        -2 \\ 1 \\ -1
      \end{bmatrix}.
    \end{alignedat}
  \end{displaymath}
\end{remark}

\begin{figure}[htbp]
  \centering
  \begin{tabular}{ccc}
    \multicolumn{1}{l}{(a)}
    &\mbox{}
    &\multicolumn{1}{l}{(b)}
    \\
    \includegraphics[bb=0 0 177 153,scale=0.75]{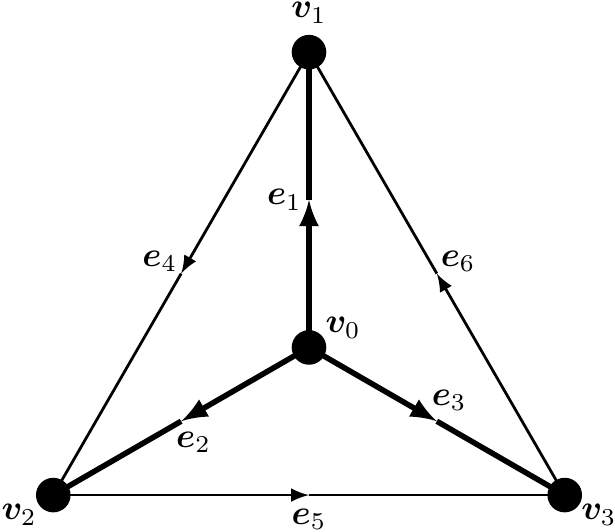}
    &\mbox{}
    &\includegraphics[bb=0 0 177 197,scale=0.75]{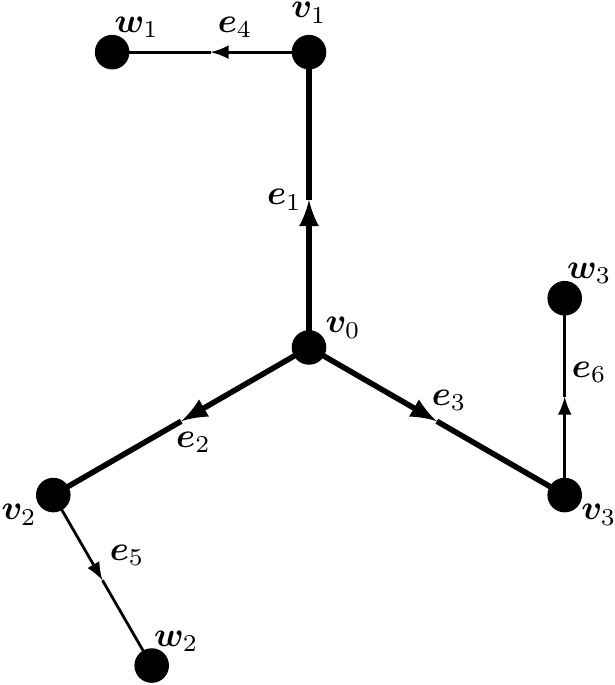}
  \end{tabular}
  \caption{
    (a) The base graph of gyroid ($K_4$) lattices, 
    thick edges consist a spanning tree of them.
    (b) building block (in the abstract graph) of gyroid lattices.
  }
  \label{fig:basegraph:k4}
\end{figure}

\begin{figure}[htp]
  \centering
  \begin{tabular}{ccc}
    \multicolumn{1}{l}{(a)}
    &\mbox{}
    &\multicolumn{1}{l}{(b)}\\
    \includegraphics[bb=0 0 109 107,height=100pt]{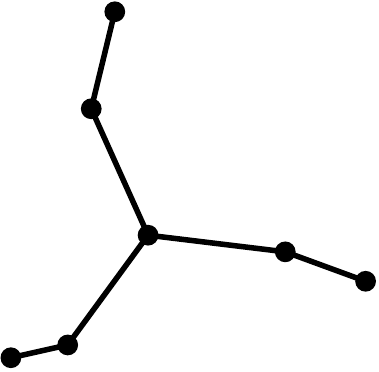}
    &\mbox{}
    &\includegraphics[bb=0 0 92 87,height=100pt]{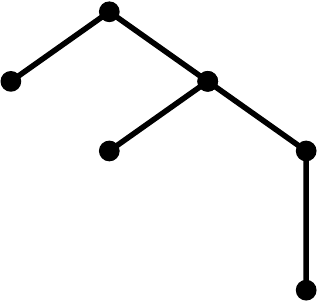}
  \end{tabular}
  \caption{
    (a) A building block of a gyroid lattice ($K_4$ lattice) viewed from a perpendicular direction of 
    the plane consisted by $\v{v}_1$, $\v{v}_2$, and $\v{v}_3$, 
    (b) one viewed from a parallel direction of it.
  }
  \label{fig:block}
\end{figure}

\begin{figure}[htbp]
  \centering
  \begin{tabular}{ccc}
    \multicolumn{1}{l}{(a)}
    &\mbox{}
    &\multicolumn{1}{l}{(b)}\\
    \includegraphics[bb=0 0 256 256,height=100pt]{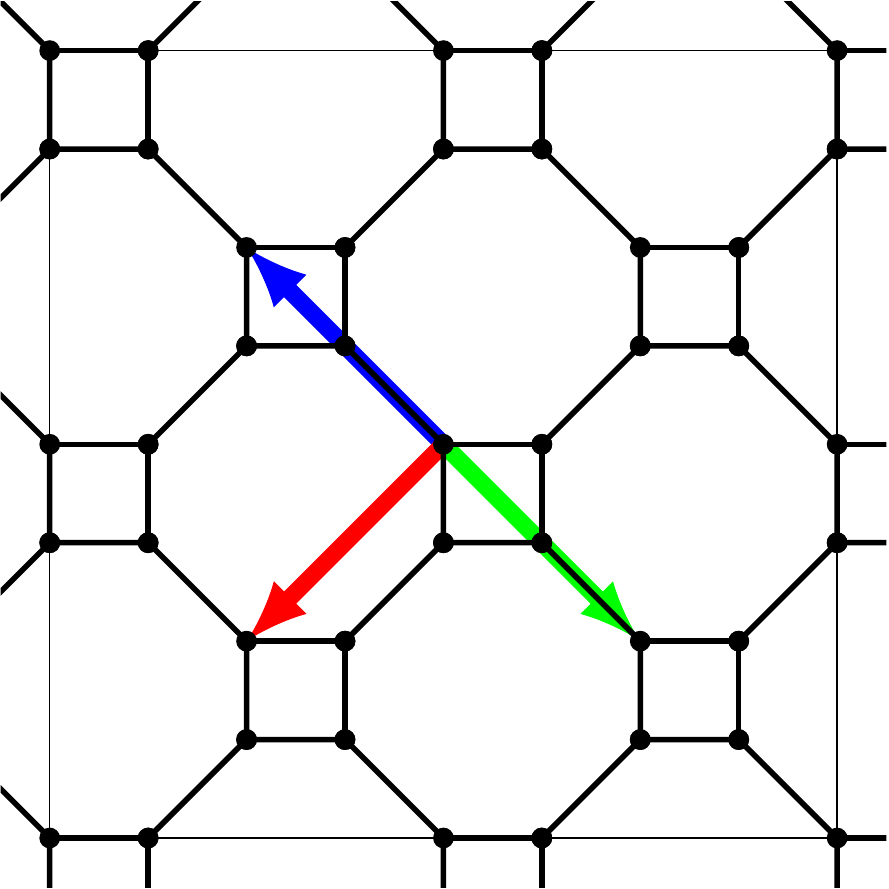}
    &\mbox{}
    &\includegraphics[bb=0 0 361 412,height=141.4pt]{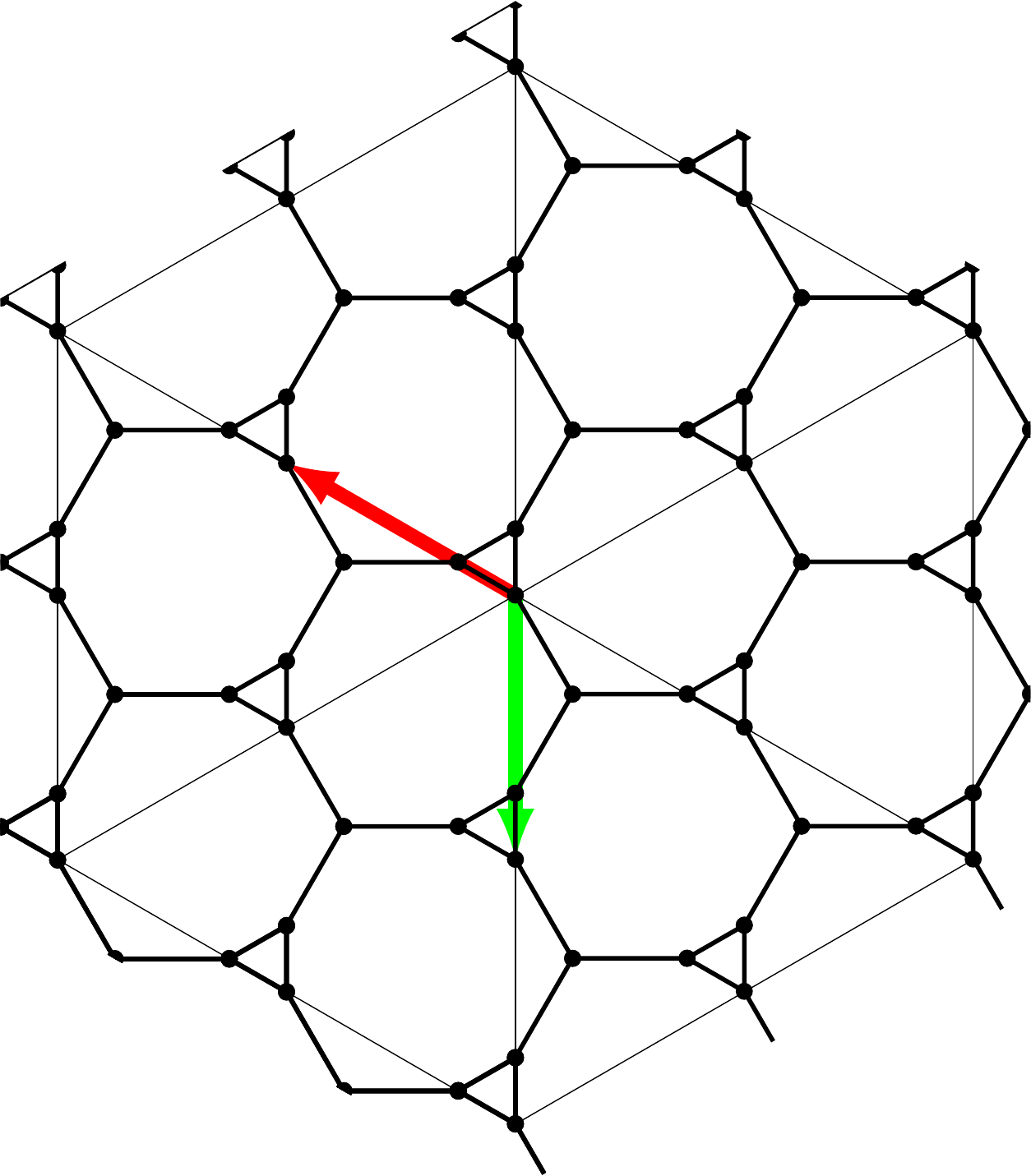}
  \end{tabular}
  \caption{
    A gyroid ($K_4$) lattice from (a) $(0, 0, 1)$-direction and (b) $(1, 1, 1)$-direction 
    by using coordinates in Remark \ref{remark:k4coordinates}.
    The blue, red, and green vectors are $\v{\alpha}_1$, $\v{\alpha}_2$, and $\v{\alpha}_3$, respectively.
    In (b), 
    $\v{\alpha}_1$ is 
    the vector perpendicular to the paper from the back to the front.
  }
  \label{fig:k4}
\end{figure}

\begin{remark}
  Let $\Phi(X)$ be a standard realization of diamond or cubic lattices, 
  and $C \in O(3)\setminus SO(3)$.
  Then, $C(\Phi(X))$ and $\Phi(X)$ are mutually congruent, 
  namely, $\Phi(X)$ and its mirror image are mutually congruent in $\R^3$.
  This property is called {\em chiral symmetry}.
  On the other hand, 
  a standard realization of $K_4$ lattices is not chiral symmetric.
  Taking 
  a $C \in O(3)\setminus SO(3)$, 
  $X' = XC$ and constructing the realization
  as in Example \ref{example:standard:k4}, 
  then we obtain a chiral image of $\Phi(X)$.
\end{remark}

\begin{figure}[htbp]
  \centering
  \begin{tabular}{ccccc}
    \includegraphics[bb=0 0 551 544,width=75pt]{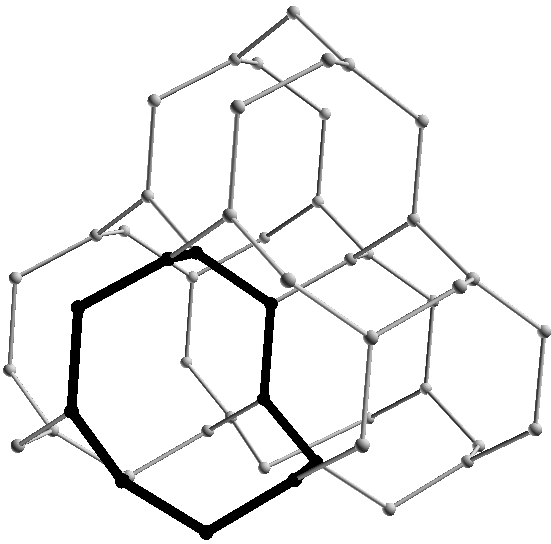}
    &\includegraphics[bb=0 0 551 544,width=75pt]{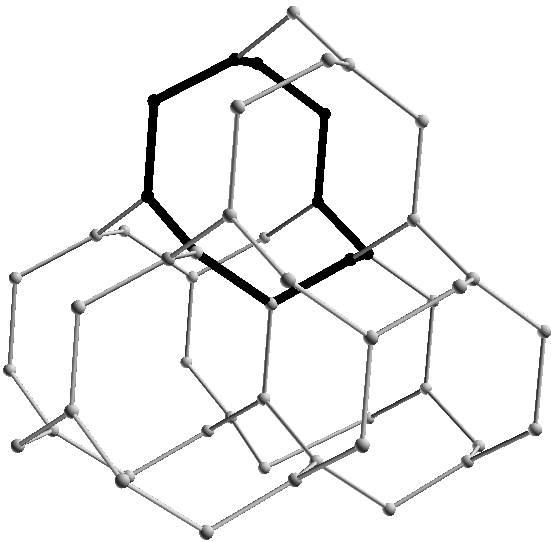}
    &\includegraphics[bb=0 0 551 544,width=75pt]{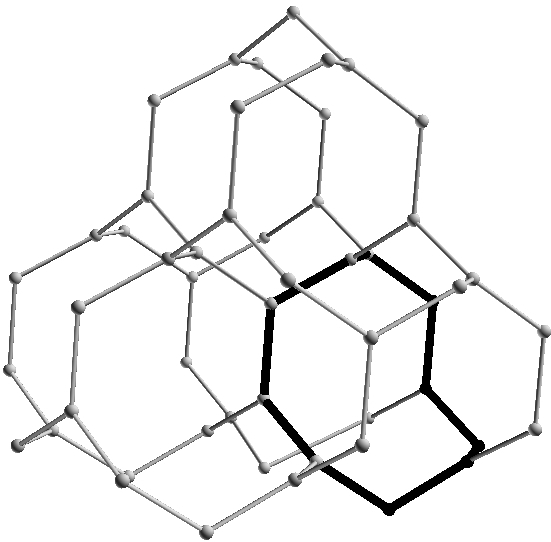}
    &\includegraphics[bb=0 0 551 544,width=75pt]{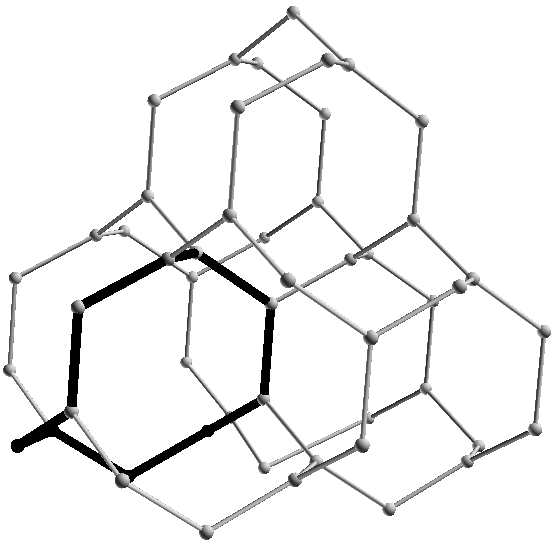}
    &\includegraphics[bb=0 0 551 544,width=75pt]{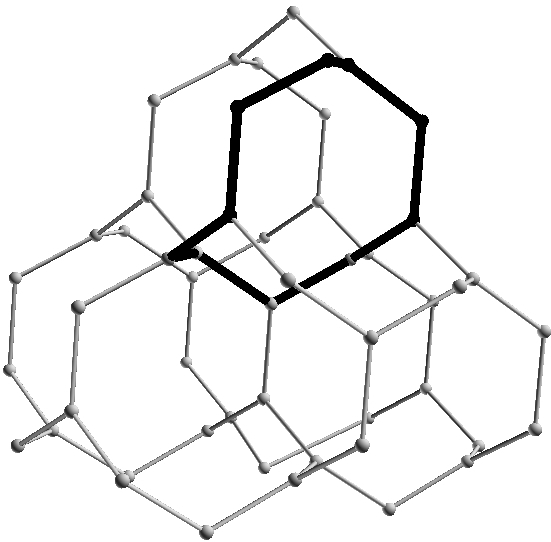}
    \\
    \includegraphics[bb=0 0 551 544,width=75pt]{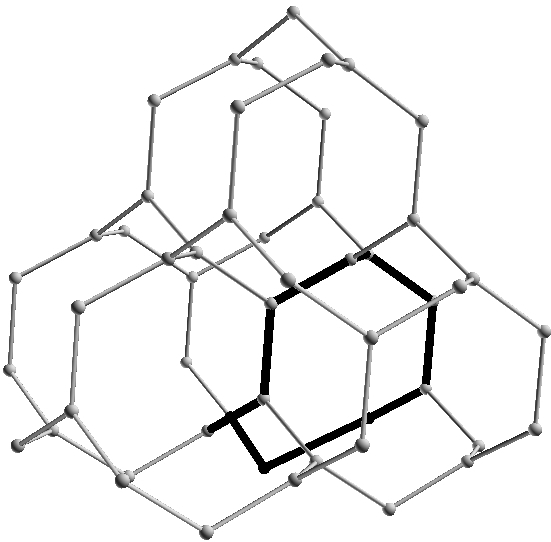}
    &\includegraphics[bb=0 0 551 544,width=75pt]{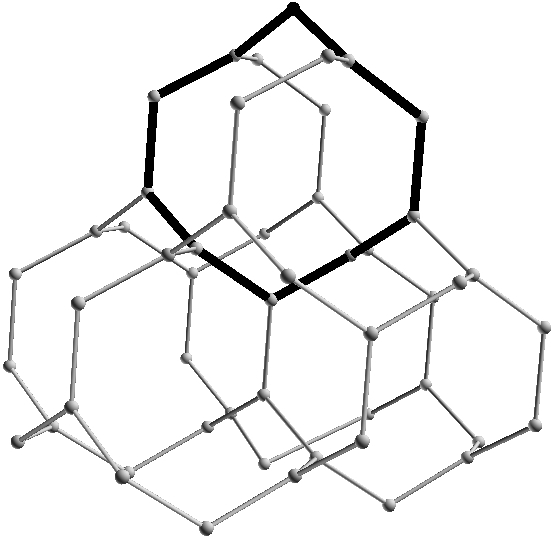}
    &\includegraphics[bb=0 0 551 544,width=75pt]{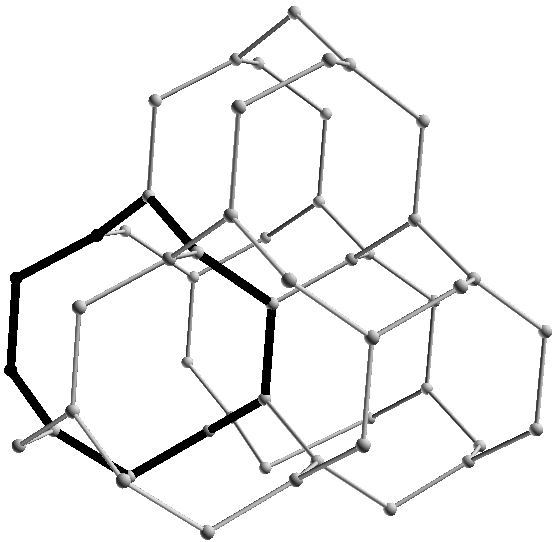}
    &\includegraphics[bb=0 0 551 544,width=75pt]{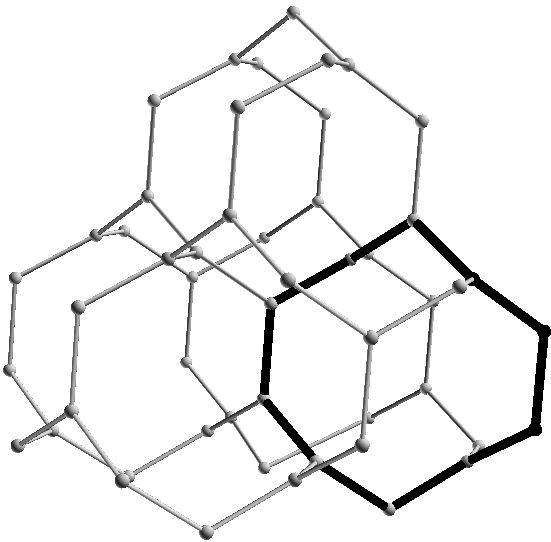}
    &\includegraphics[bb=0 0 551 544,width=75pt]{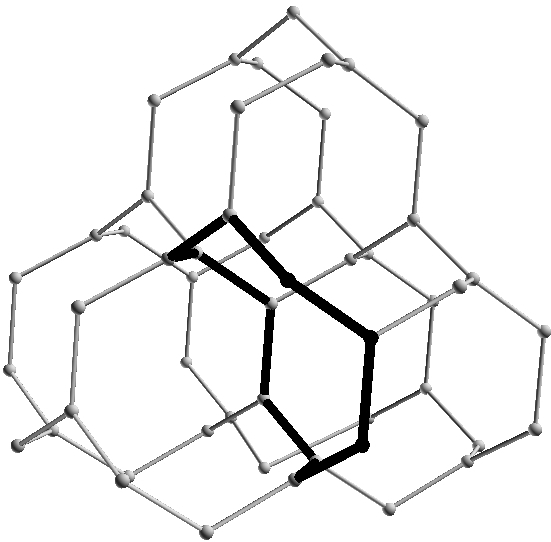}
    \\
    \includegraphics[bb=0 0 551 544,width=75pt]{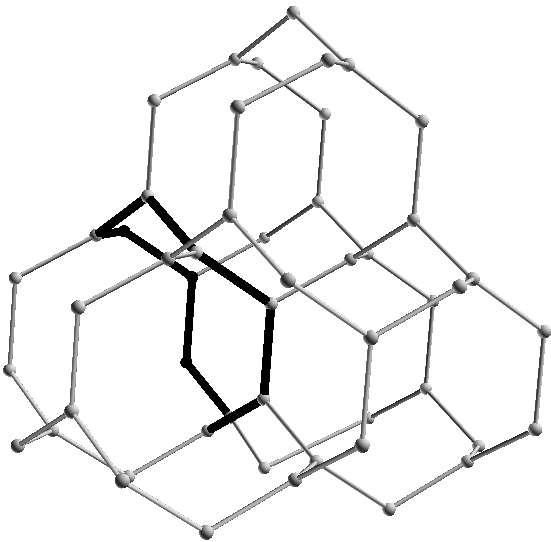}
    &\includegraphics[bb=0 0 551 544,width=75pt]{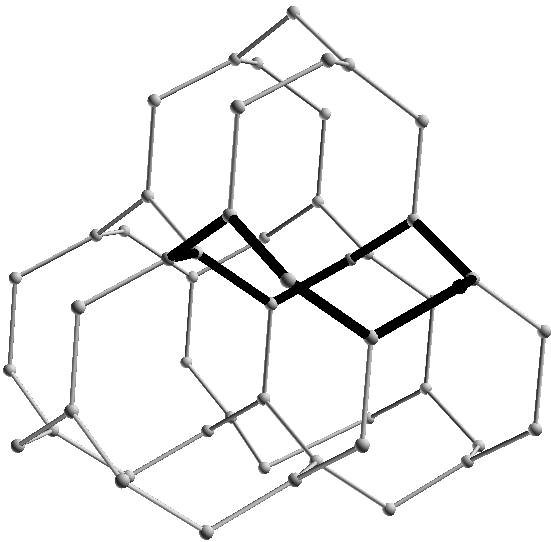}
    &\includegraphics[bb=0 0 551 544,width=75pt]{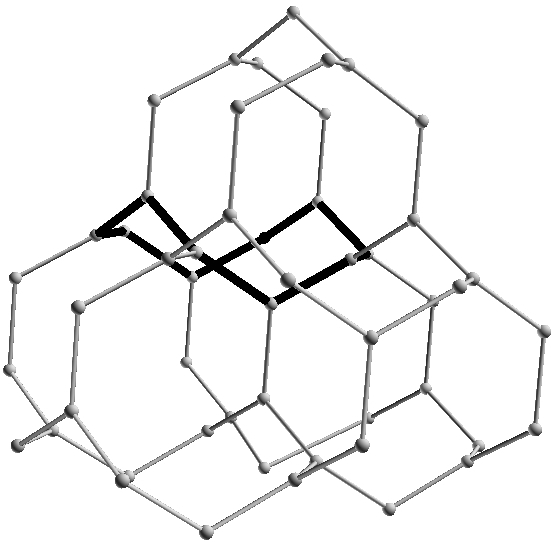}
    &\includegraphics[bb=0 0 551 544,width=75pt]{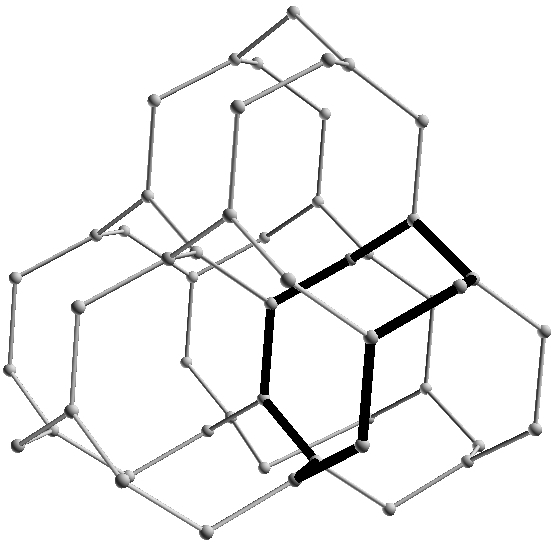}
    &\includegraphics[bb=0 0 551 544,width=75pt]{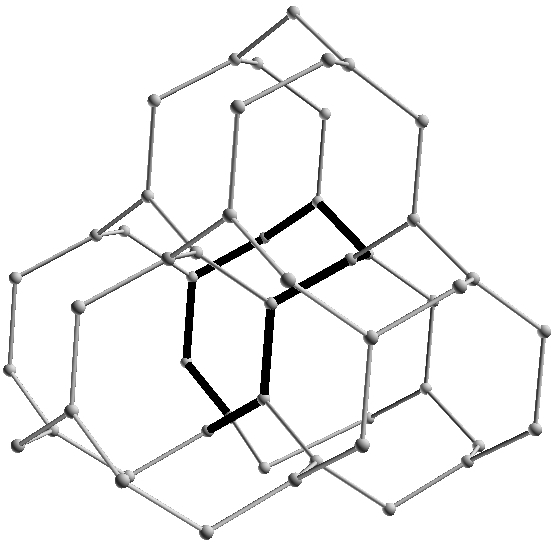}
    \\
  \end{tabular}
  \caption{
    Fifteen $10$-members rings pass through a vertex in a gyroid lattice.
    Each ring is mutually congruent.
  }
  \label{fig:k4ring}
\end{figure}
%%%%% END of document

%%%%% \input{generic_cases.tex}
% -*- coding: utf-8 -*-
\subsubsection{Explicit algorithm for generic cases}
In general, a standard realization of $d$-dimensional topological crystal $X$ is 
not maximal abelian covering of a base graph $X_0$.
In this section, we assume that $d < b = \rank H_1(X_0, \Z)$, 
and explain explicit algorithm to obtain a standard realization of $X$.
This method is followed by Sunada \cite{MR3014418}.
\par
Since $d < b$, the standard realization is constructed in a $d$-dimensional subspace $V$ of $H_1(X_0, \R)$, 
whose orthogonal subspace $H$ is called the {\em vanishing subspace}, 
namely, $H_1(X_0, \R) = V \oplus H$, and $V = H^\perp$ with $\dim H = b-d$.

\paragraph*{Step 1}
By using the method of Step 1 in Section \ref{sec:explicit:max}, 
find a $\Z$-basis $\{{\alpha}_i\}_{i=1}^b$ of $H_1(X_0, \Z)$, 
such that $\{\alpha_i\}_{i=d+1}^b$ is a basis of the vanishing subspace $H$, 
using a linear transformation if necessary.

\paragraph*{Step 2}
Compute $A$, $\v{b}(e)$, and $\v{a}(e)$ as in Step 2 of Section \ref{sec:explicit:max}, 
then we obtain a standard realization of a topological crystal $\widetilde{X}$, 
which is a maximal abelian covering of the base graph $X_0$.
This realization $\Phi^{\text{max}}$ is in $H_1(X_0, \R) \cong \R^b$.

\paragraph*{Step 3}
Let $p \colon H_1(X_0, \R) \longrightarrow H$ be the orthogonal projection, 
then $\{\beta_i\}_{i=1}^d$
is a $\Z$-basis of the period lattice, where $\beta_i = p(\alpha_i)$.
We should obtain $B = (\inner{\beta_i}{\beta_j}) \in GL(d, \R)$ to calculate 
standard realizations of $X$.
Since $\gamma_i - \alpha_i = p(\alpha_i) - \alpha_i \in H$, 
we may write 
\begin{displaymath}
  p(\alpha_i) = \alpha_i + \sum_{j=b+1}^d d_{ij}\alpha_j
\end{displaymath}
and $\inner{p(\alpha_i)}{\alpha_k} = 0$ for $k = b+1, \ldots, d$, 
and hence we obtain
\begin{equation}
  \label{eq:explicit:2:1}
  \inner{\alpha_i}{\alpha_k}
  =
  -\sum_{j=b+1}^d d_{ij} \inner{\alpha_j}{\alpha_k}, 
  \quad
  k = b+1, \ldots, d, 
  \quad
  i = 1, \ldots, b.
\end{equation}
Write
\begin{math}
  A 
  =
  \begin{bmatrix}
    A_{11} & A_{12} \\
    A_{21} & A_{22}
  \end{bmatrix}, 
\end{math}
where $A_{11}$ is $d \times d$ matrix, 
$A_{22}$ is $(b-d)\times(b-d)$ matrix, 
$A_{12}^T = A_{21}$, 
and $D = (d_{ij})$, 
then (\ref{eq:explicit:2:1}) implies
\begin{equation}
  \label{eq:explicit:2:2}
  A_{12} = -D A_{22}.
\end{equation}
Therefore, we obtain 
\begin{equation}
  \label{eq:explicit:2:3}
  \begin{aligned}
    \inner{\beta_i}{\beta_j}
    &=
    \inner{p(\alpha_i)}{p(\alpha_j)}
    =
    \inner{p^T p(\alpha_i)}{\alpha_j}
    =
    \inner{p(\alpha_i)}{\alpha_j}
    \\
    &=
    \inner{\alpha_i + \sum_{k=b+1}^d d_{ik} \alpha_k}{\alpha_j}
    =
    \inner{\alpha_i}{\alpha_j} + \sum_{k=b+1}^d d_{ik}\inner{\alpha_k}{\alpha_j}, 
  \end{aligned}
\end{equation}
and thus, by (\ref{eq:explicit:2:3}), 
we obtain
\begin{equation}
  \label{eq:explicit:2:4}
  B = A_{11} + D A_{21} = A_{11} - A_{12} A_{22}^{-1} A_{21}.
\end{equation}
\par
Since realizations of an edge $e \in E_0$ of the maximal abelian covering of $X_0$
is written as $\v{e}^{\text{max}} = \sum_{i=1}^b a(e) \alpha_i$, 
combining $P \colon C_1(X_0, \R) \longrightarrow H_1(X_0, \R)$ and 
$p \colon H_1(X_0, \R) \longrightarrow H$, we obtain 
\begin{equation}
  \label{eq:explicit:2:5}
  p(P(e)) = p(\v{e}^{\text{max}})
  = \v{e}
  = \sum_{i=1}^d a(e) \beta_i.
\end{equation}

\begin{example}[{\bfseries Triangular lattice}, Fig.~\ref{fig:standard} (b), Sunada {\cite[Section 8.3]{MR3014418}}]
  \label{example:standard:triangular}
  A triangular lattice is the projection of a cubic lattice in $\R^3$ onto 
  a suitable $2$-dimensional plane. 
  Hence, $d = 2$ and $b = \rank H_1(X_0, \R) = 3$, 
  and the base graph $X_0 = (V_0, E_0)$ of triangular lattices is the one of cubic lattices, 
  i.\,e., $X_0$ is the $3$-bouquet graph (\ref{fig:basegraph:cubic}).
  Using notation in Example \ref{example:standard:cubic}, 
  take $\alpha_1 = e_1$, $\alpha_2 = e_2$, $\alpha_3 = e_1 + e_2 + e_3$, 
  and $H = \Span\{\alpha_1, \alpha_2\}$, 
  then we obtain 
  \begin{equation}
    \label{eq:explicit:2:3.5}
    A
    =
    \begin{bmatrix}
      1 & 0 & 1 \\
      0 & 1 & 1 \\
      1 & 1 & 3
    \end{bmatrix}, 
    \quad
    \v{b}(e)
    =
    \begin{bmatrix}
      1 & 0 & 0 \\
      0 & 1 & 0 \\
      1 & 1 & 1
    \end{bmatrix}, 
    \quad
    \v{a}(e)
    =
    \begin{bmatrix}
      1 & 0 & -1 \\
      0 & 1 & -1 \\
      0 & 0 & 1
    \end{bmatrix}, 
  \end{equation}
  and 
  \begin{equation}
    \label{eq:explicit:2:5.5}
    \v{e}^{\text{max}}_1 = \alpha_1, 
    \quad
    \v{e}^{\text{max}}_2 = \alpha_2, 
    \quad
    \v{e}^{\text{max}}_3 = -\alpha_1 - \alpha_2 + \alpha_3.
  \end{equation}
  By (\ref{eq:explicit:2:5}) and (\ref{eq:explicit:2:5.5}), we obtain 
  \begin{equation}
    \label{eq:explicit:2:6}
    \v{e}_1 = \beta_1, 
    \quad
    \v{e}_2 = \beta_2, 
    \quad
    \v{e}_3 = -\beta_1 - \beta_2, 
  \end{equation}
  and by (\ref{eq:explicit:2:4}) and (\ref{eq:explicit:2:3.5}), we also obtain 
  \begin{displaymath}
    B =
    \begin{bmatrix}
      \inner{\beta_i}{\beta_j}
    \end{bmatrix}
    =
    \begin{bmatrix}
      1 & 0 \\ 0 & 1
    \end{bmatrix}
    - 
    \frac{1}{3}
    \begin{bmatrix}
      1 \\ 1
    \end{bmatrix}
    \begin{bmatrix}
      1 & 1 
    \end{bmatrix}
    =
    \begin{bmatrix}
      1 & 0 \\ 0 & 1
    \end{bmatrix}
    - 
    \frac{1}{3}
    \begin{bmatrix}
      1 & 1 \\ 1 & 1
    \end{bmatrix}
    =
    \frac{1}{3}
    \begin{bmatrix}
      2 & -1 \\ -1 & 2
    \end{bmatrix}.
  \end{displaymath}
  On the other hand, 
  the shortest paths from $\v{v}_0$ to other vertices are 
  \begin{displaymath}
    \shortestpath(\v{v}_0, \v{v}_i) = (\v{v}_0 \v{v}_i) = \v{e}_i
    \quad
    i = 1, \, 2, \, 3.
  \end{displaymath}
  By using the Cholesky decomposition, we may write 
  \begin{displaymath}
    \begin{bmatrix}
      \beta_1 \\ \beta_2 
    \end{bmatrix}
    =
    \begin{bmatrix}
      \sqrt{2/3} &
      -1/\sqrt{6} \\
      0 &
      1/\sqrt{2}
    \end{bmatrix}, 
  \end{displaymath}
  and hence by (\ref{eq:explicit:2:6}), we obtain 
  \begin{displaymath}
    \v{e}_1 
    =
    \begin{bmatrix}
      \sqrt{2/3} \\ 0
    \end{bmatrix}, 
    \quad
    \v{e}_2
    =
    \begin{bmatrix}
      -1/\sqrt{6} \\ 1/\sqrt{2}
    \end{bmatrix}.
    \quad
    \v{e}_3
    =
    \begin{bmatrix}
      -1/\sqrt{6} \\ -1/\sqrt{2}
    \end{bmatrix}, 
  \end{displaymath}
  and 
  \begin{displaymath}
    \v{v}_0 
    =
    \begin{bmatrix}
      0 \\ 0
    \end{bmatrix}, 
    \quad
    \v{v}_1 = \v{v}_0 + \v{e}_1 
    =
    \begin{bmatrix}
      \sqrt{2/3} \\ 0
    \end{bmatrix}, 
    \quad
    \v{v}_2 = \v{v}_0 + \v{e}_2 
    =
    \begin{bmatrix}
      -1/\sqrt{6} \\ 1/\sqrt{2}
    \end{bmatrix}.
    \quad
    \v{v}_3 = \v{v}_0 + \v{e}_3
    =
    \begin{bmatrix}
      -1/\sqrt{6} \\ -1/\sqrt{2}
    \end{bmatrix}.
  \end{displaymath}
  The above datas allow us to write figure in Fig.~\ref{fig:standard} (b) (see also Fig.~\ref{fig:triangle_blocks}).
\end{example}

\begin{figure}[htbp]
  \centering
  \begin{tabular}{ccc}
    \multicolumn{1}{l}{(a)}
    &\multicolumn{1}{l}{(b)}
    \\
    \includegraphics[bb=0 0 115 131,scale=0.50]{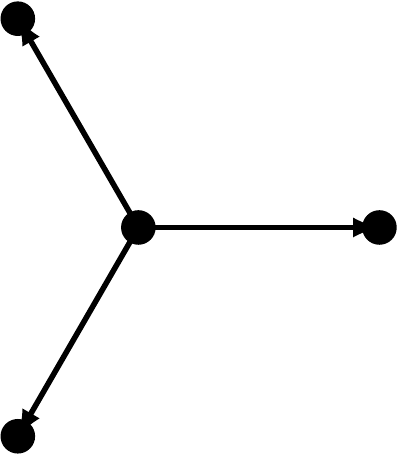}
    &\includegraphics[bb=0 0 210 187,scale=0.50]{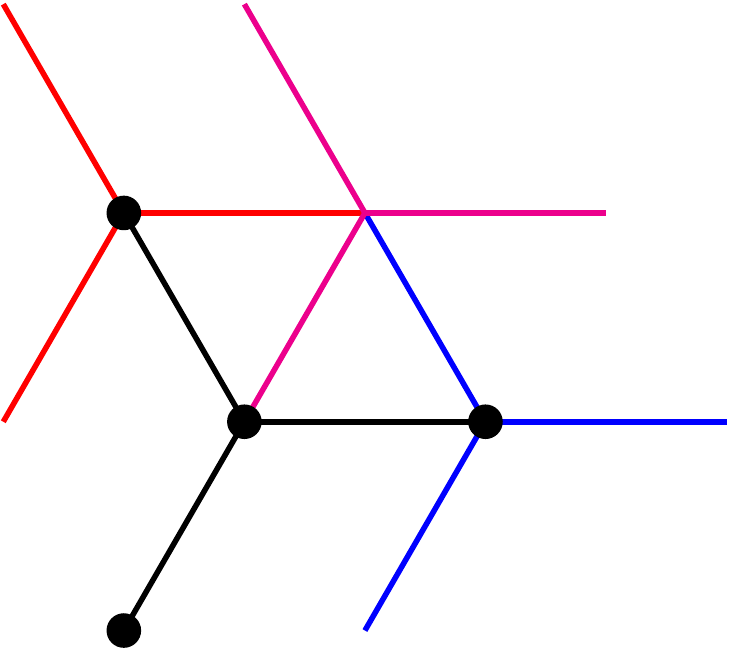}
  \end{tabular}
  \caption{Building block of the triangular lattice, 
    and its translations by $B = \{\v{\beta}_1, \v{\beta}_2\}$.
    (a) building block $\{\v{b}\}$ of the triangular lattice, 
    (b) 
    the blue, red, and magenta blocks are the block translated by $\v{\beta}_1$, 
    $\v{\beta}_2$, 
    and 
    $\v{\beta}_1 + \v{\beta}_2$.
  }
  \label{fig:triangle_blocks}
\end{figure}

\begin{example}[{\bfseries Kagome lattice}, Fig.~\ref{fig:standard} (d), Sunada {\cite[Section 8.3]{MR3014418}}]
  \label{example:standard:kagome}
  A kagome lattice is a standard realization in $\R^2$ 
  whose base graph $X_0$ shown in Fig.~\ref{fig:basegraph:kagome} (a).
  The graph $X_0$ satisfies $b = \rank H_1(X_0, \R) = 4$, 
  and we may select 
  \begin{displaymath}
    \alpha_1 = e_1 - e_4, 
    \quad
    \alpha_2 = e_2 - e_5, 
    \quad
    \alpha_3 = e_1 + e_2 + e_3, 
    \quad
    \alpha_4 = e_4 + e_5 + e_6, 
  \end{displaymath}
  and $H = \Span\{e_1 + e_2 + e_3, e_4 + e_5 + e_6\}$.
  Then, we obtain 
  \begin{equation}
    \label{eq:explicit:2:10}
    A
    =
    \scalebox{0.75}[1.0]{
      $
      \displaystyle
      \begin{bmatrix}
        2 & 0 & 1 & -1 \\
        0 & 2 & 1 & -1 \\
        1 & 1 & 3 & 0 \\
        -1 & -1 & 0 & 3 \\
      \end{bmatrix}
      $}, 
    \quad
    \v{b}(e)
    =
    \scalebox{0.75}[1.0]{
      $
      \displaystyle
      \begin{bmatrix}
        1 & 0 & 0 & -1 & 0 & 0 \\
        0 & 1 & 0 & 0 & -1 & 0 \\
        1 & 1 & 1 & 0 & 0 & 0 \\
        0 & 0 & 0 & 1 & 1 & 1 \\
      \end{bmatrix}
      $}
    , 
    \quad
    \v{a}(e)
    =
    \scalebox{0.75}[1.0]{
      $
      \displaystyle
      \frac{1}{6}
      \begin{bmatrix}
        3 & 0 & -3 & -3 & 0 & 3 \\
        0 & 3 & -3 & 0 & -3 & 3 \\
        1 & 1 & 4 & 1 & 1 & -2 \\
        1 & 1 & -2 & 1 & 1 & 4 \\
      \end{bmatrix}
      $}, 
  \end{equation}
  \begin{equation}
    \label{eq:explicit:2:11}
    \begin{alignedat}{3}
      \v{e}^{\text{max}}_1 &= (1/6)(3\alpha_1+\alpha_3+\alpha_4), 
      \\
      \v{e}^{\text{max}}_2 &= (1/6)(3\alpha_2+\alpha_3+\alpha_4), 
      \\
      \v{e}^{\text{max}}_3 &= (1/6)(-3\alpha_1-3\alpha_2+4\alpha_3-2\alpha_4), 
      \\
      \v{e}^{\text{max}}_4 &= (1/6)(-3\alpha_1+\alpha_3+\alpha_4), 
      \\
      \v{e}^{\text{max}}_5 &= (1/6)(-3\alpha_2+\alpha_3+\alpha_4), 
      \\
      \v{e}^{\text{max}}_6 &= (1/6)(3\alpha_1+3\alpha_2-2\alpha_3+4\alpha_4), 
      \\
    \end{alignedat}
  \end{equation}
  By (\ref{eq:explicit:2:5}) and (\ref{eq:explicit:2:11}), we obtain 
  \begin{equation}
    \label{eq:explicit:2:12}
    \begin{alignedat}{3}
      \v{e}_1 &= (1/2) \beta_1, 
      &\quad
      \v{e}_2 &= (1/2) \beta_2, 
      \quad
      \v{e}_3 &= -(1/2)(\beta_1+\beta_2), 
      \\
      \v{e}_4 &= -(1/2)\beta_1, 
      &\quad
      \v{e}_5 &= -(1/2)\beta_2, 
      \quad
      \v{e}_6 &= (1/2)(\beta_1+\beta_2), 
      \\
    \end{alignedat}
  \end{equation}
  and by (\ref{eq:explicit:2:4}) and (\ref{eq:explicit:2:12}), we also obtain 
  \begin{displaymath}
    B =
    \begin{bmatrix}
      \inner{\beta_i}{\beta_j}
    \end{bmatrix}
    =
    \begin{bmatrix}
      2 & 0 \\ 0 & 2
    \end{bmatrix}
    - 
    \begin{bmatrix}
      1 & -1 \\ 1 & -1 
    \end{bmatrix}
    \begin{bmatrix}
      1/3 & 0 \\ 0 & 1/3 \\
    \end{bmatrix}
    \begin{bmatrix}
      1 & 1 \\ -1 & -1 
    \end{bmatrix}
    =
    \frac{2}{3}
    \begin{bmatrix}
      2 & -1 \\ -1 & 2
    \end{bmatrix}
  \end{displaymath}
  On the other hand, 
  the shortest paths from $\v{v}_0$ to other vertices are 
  \begin{displaymath}
    \shortestpath(\v{v}_0, \v{v}_i) = (\v{v}_0 \v{v}_i) = \v{e}_i
    \quad
    \shortestpath(\v{v}_0, \v{w}_i) = (\v{v}_0 \v{w}_i) = \v{e}_{i+3}
    \quad
    i = 1, \, 2, \, 3.
  \end{displaymath}
  By using the Cholesky decomposition, we may write 
  \begin{displaymath}
    \begin{bmatrix}
      \beta_1 \\ \beta_2
    \end{bmatrix}
    =
    \begin{bmatrix}
      2/\sqrt{3} &
      -1/\sqrt{3} \\
      0 & 1
    \end{bmatrix}, 
  \end{displaymath}
  and hence by (\ref{eq:explicit:2:12}), we obtain 
  \begin{displaymath}
    \v{e}_1 = -\v{e}_4
    =
    \begin{bmatrix}
      1/\sqrt{3} \\ 0
    \end{bmatrix}, 
    \quad
    \v{e}_2 = -\v{e}_5
    =
    \begin{bmatrix}
      -1/(2\sqrt{3}) \\ 1/2
    \end{bmatrix}.
    \quad
    \v{e}_3 = -\v{e}_6
    =
    \begin{bmatrix}
      1/(2\sqrt{3}) \\ 1/2
    \end{bmatrix}, 
  \end{displaymath}
  and 
  \begin{displaymath}
    \begin{aligned}
      \v{v}_0 
      =
      \begin{bmatrix}
        0 \\ 0
      \end{bmatrix}, 
      \quad
      \v{v}_1 = \v{v}_0 + \v{e}_1 
      &=
      \begin{bmatrix}
        1/\sqrt{3} \\ 0
      \end{bmatrix}, 
      &\quad
      \v{v}_2 = \v{v}_0 + \v{e}_2 
      &=
      \begin{bmatrix}
        -1/(2\sqrt{3}) \\ 1/2
      \end{bmatrix}, 
      \\
      \v{w}_1 = \v{v}_0 + \v{e}_4
      &=
      \begin{bmatrix}
        -1/\sqrt{3} \\ 0
      \end{bmatrix}, 
      &\quad
      \v{w}_2 = \v{v}_0 + \v{e}_5 
      &=
      \begin{bmatrix}
        1/(2\sqrt{3}) \\ -1/2
      \end{bmatrix}.
    \end{aligned}
  \end{displaymath}
  The above datas allow us to write figure in Fig.~\ref{fig:standard} (d) (see also Fig.~\ref{fig:kagome_blocks}).
  Since $\triangle \v{v}_0 \v{v}_1 \v{v}_2$ and $\triangle \v{v}_0 \v{w}_1 \v{w}_2$ consist regular triangles, 
  a standard realization of kagome lattices is consisted by regular triangles sharing vertices each other.
\end{example}

\begin{figure}[htbp]
  \centering
  \begin{tabular}{ccc}
    \multicolumn{1}{l}{(a)}
    &\multicolumn{1}{l}{(b)}
    \\
    \includegraphics[bb=0 0 109 96,scale=0.75]{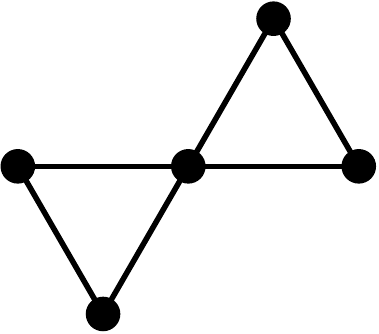}
    &\includegraphics[bb=0 0 248 177,scale=0.50]{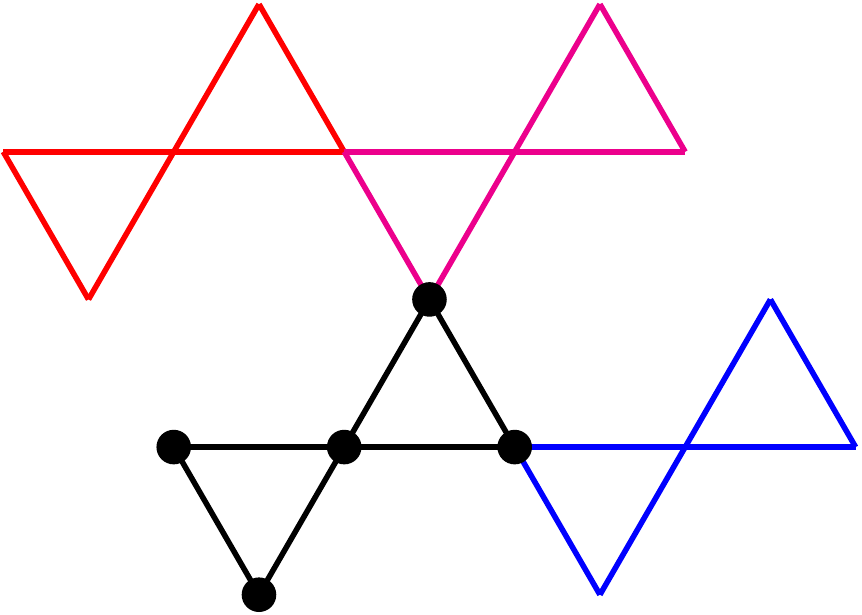}
  \end{tabular}
  \caption{Building block of the kagome lattice, 
    and its translations by $B = \{\v{\beta}_1, \v{\beta}_2\}$.
    (a) building block $\{\v{b}\}$ of the triangular lattice, 
    (b) 
    the blue, red, and magenta blocks are blocks translated by $\v{\beta}_1$, 
    $\v{\beta}_2$, 
    and 
    $\v{\beta}_1 + \v{\beta}_2$.
  }
  \label{fig:kagome_blocks}
\end{figure}

Next, we consider higher dimensional analogues of kagome lattices.
As mentioned in Example \ref{example:standard:kagome}, 
a standard realization of a kagome lattice consists by 
regular triangles sharing vertices each other.
One of the $3$-dimensional analogues of kagome lattices is 
a hyper-kagome lattice of type II,
whose standard realization consists quadrilaterals sharing vertices each other.
Since a triangles in $\R^2$ is a $1$-simplex, 
the other is a hyper-kagome lattice of type I, 
whose standard realization consists $1$-skeleton of $2$-simplex sharing vertices each other.

\begin{example}[{\bfseries 3D kagome lattice of type I}, Sunada {\cite[Section 8.3]{MR3014418}}]
  \label{example:standard:hyperkagome1}
  One of the 3-dimensional analogues of kagome lattices is defined as follows.
  Let $X_0$ be a graph in Fig.~\ref{fig:basegraph:kagome} (b), and 
  and $\widetilde{X}$ be its maximal abelian covering.
  Since $b = \rank H_1(X_0, \R) = 9$, $\widetilde{X}$ is $9$-dimensional a topological crystal.
  Take a $\Z$-basis of $H_1(X_0, \R)$ as 
  \begin{displaymath}
    \begin{alignedat}{3}
      \alpha_1 &= e_1 - e_4, 
      &\quad
      \alpha_2 &= e_2 - e_5, 
      &\quad
      \alpha_3 &= e_3 - e_6, 
      \\
      \alpha_4 &= e_7 - e_2 + e_1, 
      &\quad
      \alpha_5 &= e_8 - e_3 + e_2,
      &\quad
      \alpha_6 &= e_9 - e_1 + e_3, 
      \\
      \alpha_7 &= e_{10} + e_5 - e_4, 
      &\quad
      \alpha_8 &= e_{11} + e_6 - e_5, 
      &\quad
      \alpha_9 &= e_{12} + e_4 - e_6, 
    \end{alignedat}
  \end{displaymath}
  and 
  \begin{displaymath}
    H = \Span\{\alpha_4, \alpha_5, \alpha_6, \alpha_7, \alpha_8, \alpha_9\}.
  \end{displaymath}
  The number of vertices in a building block in $H_1(X_0, \R)$ is $7$, and 
  the shortest paths from $v_0$ are
  \begin{displaymath}
    \begin{alignedat}{5}
      \shortestpath(\v{v}_0, \v{v}_1) &= \v{e}_1, 
      &\quad
      \shortestpath(\v{v}_0, \v{v}_2) &= \v{e}_2, 
      &\quad
      \shortestpath(\v{v}_0, \v{v}_3) &= \v{e}_3, 
      \\
      \shortestpath(\v{v}_0, \v{w}_1) &= -\v{e}_4, 
      &\quad
      \shortestpath(\v{v}_0, \v{w}_2) &= -\v{e}_5, 
      &\quad
      \shortestpath(\v{v}_0, \v{w}_3) &= -\v{e}_6.
      \\
    \end{alignedat}
  \end{displaymath}
  A building block are 
  \begin{displaymath}
    \begin{aligned}
      \v{e}_1 &=
      \begin{bmatrix}
        1/2 \\ 0 \\ 0 \\
      \end{bmatrix}, 
      &\quad
      \v{e}_2 &=
      \begin{bmatrix}
        1/4 \\ \sqrt{3}/4 \\ 0 \\
      \end{bmatrix}
      &\quad
      \v{e}_3 
      &=
      \begin{bmatrix}
        1/4 \\ 1/(4 \sqrt{3}) \\ 1/\sqrt{6} \\
      \end{bmatrix}
      &\quad
      \v{e}_4 
      &= 
      \begin{bmatrix}
        -1/2 \\ 0 \\ 0 \\
      \end{bmatrix}
      \\
      \v{e}_5
      &=
      \begin{bmatrix}
        -1/4 \\ -\sqrt{3}/4 \\ 0 \\
      \end{bmatrix}
      &\quad
      \v{e}_6
      &=
      \begin{bmatrix}
        -1/4 \\ -1/(4 \sqrt{3}) \\ -1/\sqrt{6} \\        
      \end{bmatrix}
      &\quad
      \v{e}_7
      &=
      \begin{bmatrix}
        -1/4 \\ \sqrt{3}/4 \\ 0 \\
      \end{bmatrix}
      &\quad
      \v{e}_8
      &=
      \begin{bmatrix}
        0 \\ -1/(2 \sqrt{3}) \\ 1/\sqrt{6} \\
      \end{bmatrix}
      \\
      \v{e}_9
      &=
      \begin{bmatrix}
        1/4 \\ -1/(4 \sqrt{3}) \\ -1/\sqrt{6} \\
      \end{bmatrix}
      &\quad
      \v{e}_{10}
      &=
      \begin{bmatrix}
        -1/4 \\ \sqrt{3}/4 \\ 0 \\
      \end{bmatrix}
      &\quad
      \v{e}_{11}
      &=
      \begin{bmatrix}
        0 \\ -1/(2 \sqrt{3}) \\ 1/\sqrt{6} \\
      \end{bmatrix}
      &\quad
      \v{e}_{12}
      &=
      \begin{bmatrix}
        1/4 \\ -1/(4 \sqrt{3} \\ -1/\sqrt{6} \\        
      \end{bmatrix}
      \\
    \end{aligned}
  \end{displaymath}
  and
  \begin{displaymath}
    \begin{bmatrix}
      \beta_1 \\ \beta_2 \\ \beta_3
    \end{bmatrix}
    =
    \begin{bmatrix}
      1 & 0 & 0 \\
      1/2 & \sqrt{3}/2 & 0 \\
      1/2 & 1/(2 \sqrt{3}) & \sqrt{2/3} \\
    \end{bmatrix}, 
  \end{displaymath}
  then we obtain 3D kagome lattice of type I (Fig.~\ref{fig:kagome_1_blocks}).
  This lattice is sometimes called as {\em Pyrochlore lattice} or simply {\em hyper-kagome lattice}.
\end{example}

\begin{example}[{\bfseries 3D kagome lattice of type II}, Sunada {\cite[Section 8.3]{MR3014418}}]
  \label{example:standard:hyperkagome2}
  The other 3-dimensional analogue of kagome lattices is defined as follows.
  Let $X_0$ be a graph in Fig.~\ref{fig:basegraph:kagome} (c), and 
  and $\widetilde{X}$ be its maximal abelian covering.
  Since $b = \rank H_1(X_0, \R) = 5$, $\widetilde{X}$ is $5$-dimensional a topological crystal.
  Take a $\Z$-basis of $H_1(X_0, \R)$ as
  \begin{displaymath}
    \begin{aligned}
      \alpha_1 &= e_1 - e_4, 
      \quad
      \alpha_2 = e_2 - e_5, 
      \quad
      \alpha_3 = e_3 - e_6, 
      \\
      \alpha_4 &= e_1 + e_2 + e_3 + e_4, 
      \quad
      \alpha_5 = e_5 + e_6 + e_7 + e_8, 
    \end{aligned}
  \end{displaymath}
  and 
  \begin{displaymath}
    H = \Span\{e_1 + e_2 + e_3 + e_4, e_5 + e_6 + e_7 + e_8\}.
  \end{displaymath}
  The number of vertices in a building block in $H_1(X_0, \R)$ is $7$, and 
  the shortest paths from $v_0$ are
  \begin{displaymath}
    \begin{alignedat}{5}
      \scalebox{1.0}[1.0]{$\shortestpath(\v{v}_0, \v{v}_1)$} &= \scalebox{1.0}[1.0]{$\v{e}_1$}, 
      &\quad
      \scalebox{1.0}[1.0]{$\shortestpath(\v{v}_0, \v{v}_2)$} &= \scalebox{1.0}[1.0]{$\v{e}_1 + \v{e}_2$}, 
      &\quad
      \scalebox{1.0}[1.0]{$\shortestpath(\v{v}_0, \v{v}_3)$} &= \scalebox{1.0}[1.0]{$\v{e}_1 + \v{e}_2 + \v{e}_3$}, 
      \\
      \scalebox{1.0}[1.0]{$\shortestpath(\v{v}_0, \v{w}_1)$} &= \scalebox{1.0}[1.0]{$\v{e}_5$}, 
      &\quad
      \scalebox{1.0}[1.0]{$\shortestpath(\v{v}_0, \v{w}_2)$} &= \scalebox{1.0}[1.0]{$\v{e}_5 + \v{e}_6$}, 
      &\quad
      \scalebox{1.0}[1.0]{$\shortestpath(\v{v}_0, \v{w}_3)$} &= \scalebox{1.0}[1.0]{$\v{e}_5 + \v{e}_6 + \v{e}_7$}.
    \end{alignedat}
  \end{displaymath}
  A building block are 
  \begin{displaymath}
    \begin{aligned}
      \v{e}_1 &=
      \begin{bmatrix}
        (1/2)\sqrt{3/2} \\ 0 \\ 0 
      \end{bmatrix}, 
      &\quad
      \v{e}_2 &=
      \begin{bmatrix}
        -1/(2 \sqrt{6}) \\ 1/\sqrt{3} \\ 0
      \end{bmatrix}
      &\quad
      \v{e}_3 
      &=
      \begin{bmatrix}
        -1/(2 \sqrt{6}) \\ -1/(2 \sqrt{3}) \\ 1/2
      \end{bmatrix}
      &\quad
      \v{e}_4 
      &= 
      \begin{bmatrix}
        -1/(2 \sqrt{6}) \\ -1/(2 \sqrt{3} \\ -1/2
      \end{bmatrix}
      \\
      \v{e}_5
      &=
      \begin{bmatrix}
        -(1/2)\sqrt{3/2} \\ 0 \\ 0 
      \end{bmatrix}
      &\quad
      \v{e}_6
      &=
      \begin{bmatrix}
        1/(2 \sqrt{6}) \\ -1/\sqrt{3} \\ 0 
      \end{bmatrix}
      &\quad
      \v{e}_7
      &=
      \begin{bmatrix}
        1/(2 \sqrt{6}) \\ 1/(2 \sqrt{3}) \\ -1/2
      \end{bmatrix}
      &\quad
      \v{e}_8
      &=
      \begin{bmatrix}
        1/(2 \sqrt{6}) \\ 1/(2 \sqrt{3}) \\ 1/2
      \end{bmatrix}, 
    \end{aligned}
  \end{displaymath}
  and
  \begin{displaymath}
    \begin{bmatrix}
      \beta_1 \\ \beta_2 \\ \beta_3
    \end{bmatrix}
    =
    \begin{bmatrix}
      \sqrt{{3}/{2}} & -1/\sqrt{6} & -{1}/{\sqrt{6}} \\
      0 & {2}/{\sqrt{3}} & -{1}/{\sqrt{3}} \\
      0 & 0 & 1 \\
    \end{bmatrix}, 
  \end{displaymath}
  then we obtain 3D kagome lattice of type II (Fig.~\ref{fig:kagome_2_blocks}).
\end{example}

\begin{figure}[htbp]
  \centering
  \begin{tabular}{ccc}
    \multicolumn{1}{l}{(a)}
    &\multicolumn{1}{l}{(b)}
    &\multicolumn{1}{l}{(c)}
    \\
    \includegraphics[bb=0 0 128 139,scale=1.00]{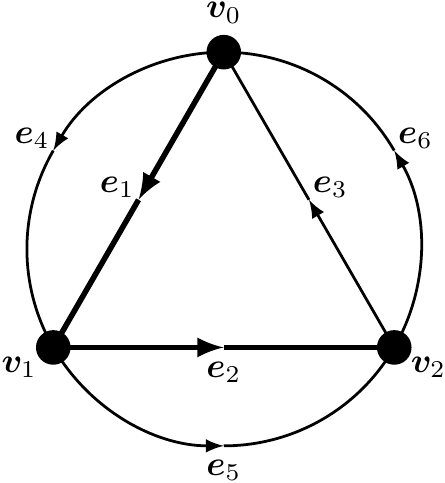}
    &\includegraphics[bb=0 0 128 139,scale=1.00]{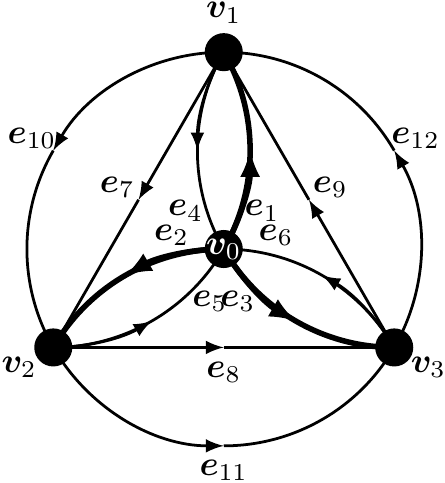}
    &\includegraphics[bb=0 0 138 135,scale=1.00]{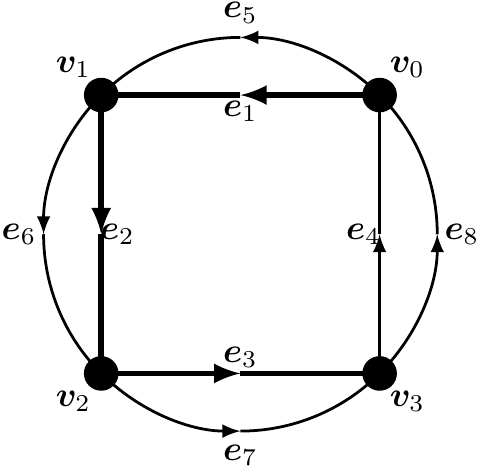}
  \end{tabular}
  \caption{
    (a) The base graph of kagome lattices, 
    (b) the base graph of 3D kagome lattices of type I, 
    (c) the base graph of 3D kagome lattices of type II.
  }
  \label{fig:basegraph:kagome}
\end{figure}

\begin{figure}[htbp]
  \centering
  \begin{tabular}{ccc}
    \multicolumn{1}{l}{(a)}
    &\multicolumn{1}{l}{(b)}
    &\multicolumn{1}{l}{(c)}
    \\
    \includegraphics[bb=0 0 124 122,height=120pt]{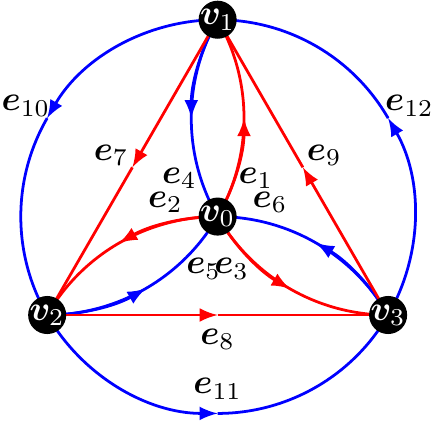}
    &\includegraphics[bb=0 0 202 128,height=80pt]{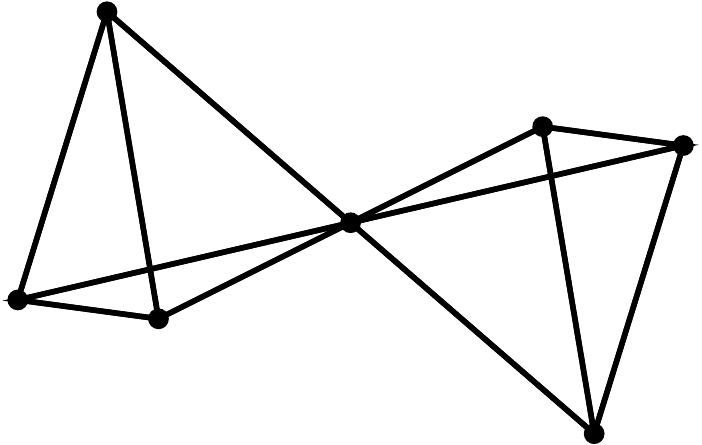}
    &\includegraphics[bb=0 0 279 150,height=90pt]{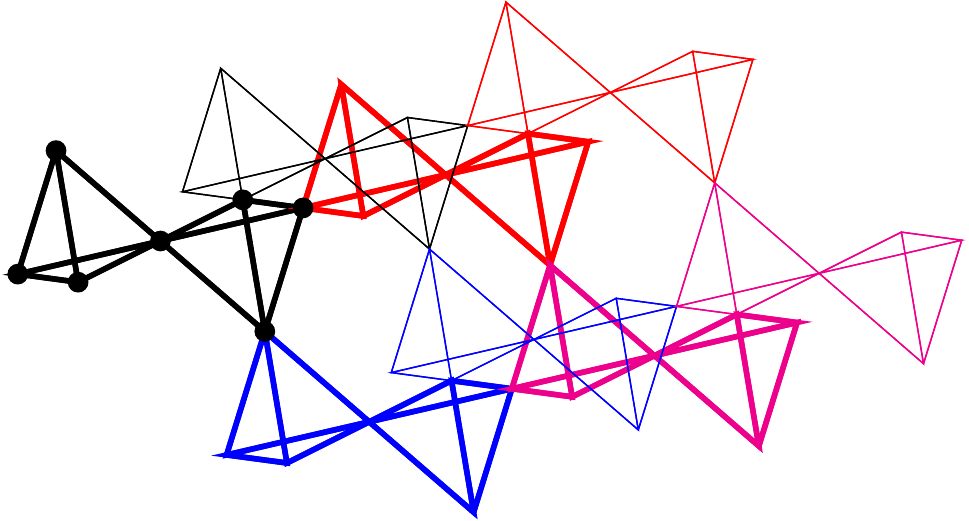}
  \end{tabular}
  \caption{
    Building block of the 3D kagome lattice of type I, 
    and its translations by $B = \{\v{\beta}_1, \v{\beta}_2, \v{\beta}_3\}$.
    (a) Each graph consisted by blue edges with vertices and
    by red edges with vertices 
    is a tetrahedral graph.
    (b) building block $\{\v{b}\}$, 
    (c)
    the blue, red, and magenta blocks are blocks translated by $\v{\beta}_1$, 
    $\v{\beta}_2$, 
    and 
    $\v{\beta}_1 + \v{\beta}_2$.
    The thin layer are translated by $\v{\beta}_3$ of above them. 
  }
  \label{fig:kagome_1_blocks}
\end{figure}

\begin{figure}[htbp]
  \centering
  \begin{tabular}{ccc}
    \multicolumn{1}{l}{(a)}
    &\multicolumn{1}{l}{(b)}
    \\
    \includegraphics[bb=0 0 125 71,scale=0.75]{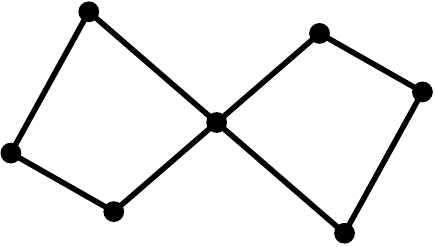}
    &\includegraphics[bb=0 0 298 245,scale=0.50]{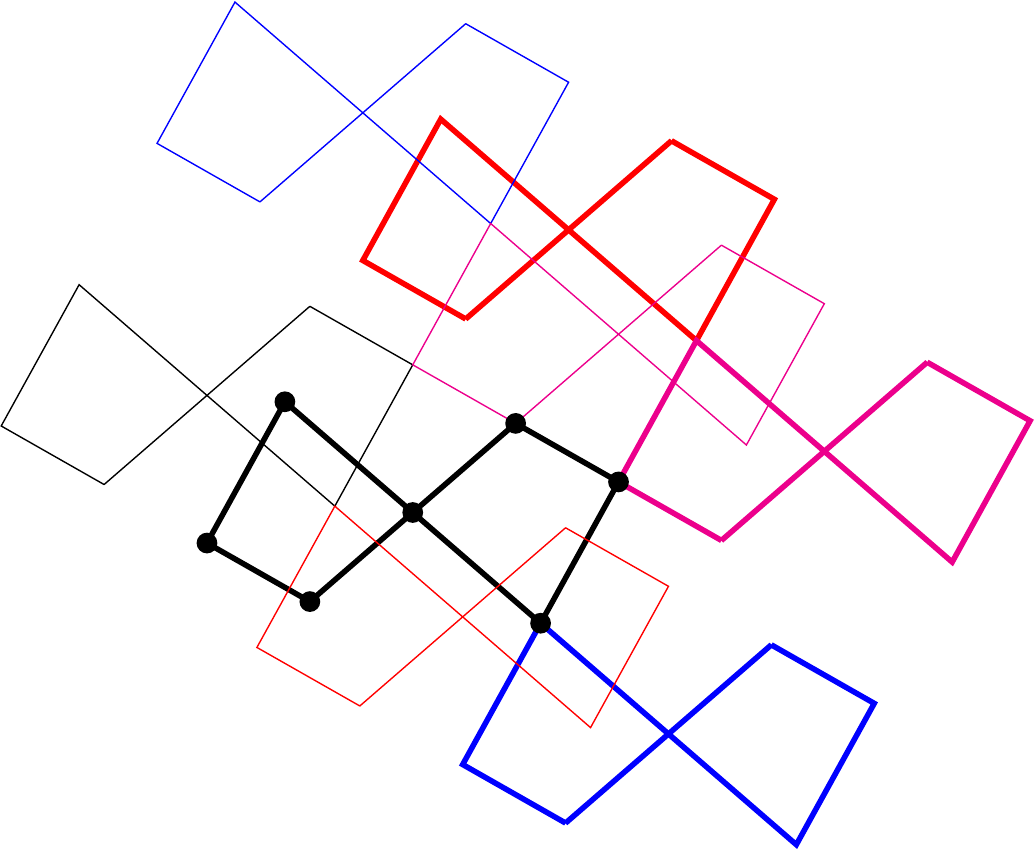}
  \end{tabular}
  \caption{Building block of the 3D kagome lattice of type II, 
    and its translations by $B = \{\v{\beta}_1, \v{\beta}_2, \v{\beta}_3\}$.
    (a) building block $\{\v{b}\}$, 
    (b) 
    the blue, red, and magenta blocks are blocks translated by $\v{\beta}_1$, 
    $\v{\beta}_2$, 
    and 
    $\v{\beta}_1 + \v{\beta}_2$.
    The thin layer are translated by $\v{\beta}_3$ of above them. 
  }
  \label{fig:kagome_2_blocks}
\end{figure}

%%%%%%%%%%%%%%%% 
\begin{example}[{\bfseries Cairo pentagonal tiling}, Sunada {\cite[Section 8.3]{MR3014418}}]
  \label{example:standard:cairo}
  A periodic tessellations is also considered as a topological crystal.
  A Cairo pentagonal tiling (Fig.~\ref{fig:cairo0} (b)) is a tessellation by congruent pentagons, 
  and it is topologically equivalent to the basketweave tiling (Fig.~\ref{fig:cairo0} (a)).
  It is also called MacMahon's net \cite{PlaneNets}.
  Here, we compute a standard realization of the cairo pentagonal tiling.
  \par
  We number vertices and edges of the graph of a fundamental region of the basketweave tiling 
  as Fig.~\ref{fig:cairo0} (a).
  Taking a basis $\{\v{p}_1, \v{p}_2\}$ of the period lattice, 
  then we obtain the equation of harmonic realizations (the equation of the balancing condition) as
  \begin{equation}
    \label{eq:cairo:1}
    \begin{alignedat}{3}
      \scalebox{0.8}[1.0]{$4 \v{v}_0$} &= \scalebox{0.8}[1.0]{$
        \v{v}_2 + (\v{v}_9 - \v{p}_2) + (\v{v}_4 - \v{p}_1) + (\v{v}_{11} - \v{p}_1 - \v{p}_2),
        $}
      &\quad
      \scalebox{0.8}[1.0]{$4 \v{v}_1$} &= \scalebox{0.8}[1.0]{$
        \v{v}_3 + \v{v}_4 + (\v{v}_9 - \v{p}_2) + (\v{v}_{10} - \v{p}_2), 
        $}
      \\
      \scalebox{0.8}[1.0]{$3 \v{v}_2$} &= \scalebox{0.8}[1.0]{$\v{v}_0 + \v{v}_3 + \v{v}_6, 
        $}
      &\quad
      \scalebox{0.8}[1.0]{$3 \v{v}_3$} &= \scalebox{0.8}[1.0]{$\v{v}_1 + \v{v}_2 + \v{v}_7, $}
      \\
      \scalebox{0.8}[1.0]{$3 \v{v}_4$} &= 
      \scalebox{0.8}[1.0]{$
        \v{v}_1 + \v{v}_5 + (\v{v}_0 + \v{p}_1), 
        $}
      &\quad
      \scalebox{0.8}[1.0]{$3 \v{v}_5$} &= \scalebox{0.8}[1.0]{$
        \v{v}_7 + \v{v}_4 + (\v{v}_6 + \v{p}_1), 
        $}
      \\
      \scalebox{0.8}[1.0]{$4 \v{v}_6$} &= 
      \scalebox{0.8}[1.0]{$
        \v{v}_2 + \v{v}_8 + (\v{v}_5 - \v{p}_1) + (\v{v}_{11} - \v{p}_1), 
        $}
      &\quad
      \scalebox{0.8}[1.0]{$4 \v{v}_7$} &= \scalebox{0.8}[1.0]{$
        \v{v}_3 + \v{v}_5 + \v{v}_8 + \v{v}_{10}, 
        $}
      \\
      \scalebox{0.8}[1.0]{$3 \v{v}_8$} &= \scalebox{0.8}[1.0]{$
        \v{v}_6 + \v{v}_7 + \v{v}_9, 
        $}
      &\quad
      \scalebox{0.8}[1.0]{$3 \v{v}_9$} &= \scalebox{0.8}[1.0]{$
        \v{v}_8 + \v{v}_0 + \v{p}_2 + \v{v}_1 + \v{p}_2, 
        $}
      \\
      \scalebox{0.8}[1.0]{$3 \v{v}_{10}$} &= \scalebox{0.8}[1.0]{$
        \v{v}_7 + \v{v}_1 + \v{p}_2 + \v{v}_{11}, 
        $}
      &\quad
      \scalebox{0.8}[1.0]{$3 \v{v}_{11}$} &= \scalebox{0.8}[1.0]{$
        \v{v}_{10} + \v{v}_6 + \v{p}_1 + \v{v}_0 + \v{p}_1 + \v{p}_2.
        $}
    \end{alignedat}
  \end{equation}
  For a given basis $\{\v{p}_1, \v{p}_2\}$, 
  we obtain a solution of (\ref{eq:cairo:1}) (a harmonic realization of the cairo pentagonal tiling) as
  \begin{equation}
    \label{eq:cairo:2}
    \begin{alignedat}{7}
      \v{v}_0 &= \v{0}, 
      &\quad
      \v{v}_1 &= (1/2)\v{p}_1, 
      &\quad
      \v{v}_2 &= (1/8)(\v{p}_1 + 2 \v{p}_2), 
      \\
      \v{v}_3 &= (1/8)(3\v{p}_1 + 2\v{p}_2), 
      &\quad
      \v{v}_4 &= (1/8)(6\v{p}_1 + \v{p}_2), 
      &\quad
      \v{v}_5 &= (1/8)(6\v{p}_1 + 3\v{p}_2), 
      \\
      \v{v}_6 &= (1/2\v{p}_2, 
      &\quad
      \v{v}_7 &= (1/8)(4\v{p}_1 + 4\v{p}_2), 
      &\quad
      \v{v}_8 &= (1/8)(2\v{p}_1 + 5\v{p}_2), 
      \\
      \v{v}_9 &= (1/8)(2\v{p}_1 + 7\v{p}_2), 
      &\quad
      \v{v}_{10} &= (1/8)(5\v{p}_1 + 6\v{p}_2), 
      &\quad
      \v{v}_{11} &= (1/8)(7\v{p}_1 + 6\v{p}_2).
    \end{alignedat}
  \end{equation}
  A harmonic realization (\ref{eq:cairo:2}) is standard 
  if and only if $\{\v{e}_i\}_{i=1}^{20}$ satisfies (\ref{eq:standard}).
  Taking $\v{f}_1 = (1, 0)^T$ and $\v{f}_2 = (0, 1)^T$, 
  and solving 
  \begin{equation}
    \label{eq:cairo:3}
    \sum_{i=1}^{20} \inner{\v{e}_i}{\v{f}_j}\v{e}_i
    = c \v{f}_j, 
    \quad 
    j = 1,\, 2, 
  \end{equation}
  then we obtain 
  \begin{equation}
    \label{eq:cairo:4}
    |\v{p}_1| = |\v{p}_2|,
    \quad
    \inner{\v{p}_1}{\v{p}_2} = 0.
  \end{equation}
  Substituting (\ref{eq:cairo:4}) into (\ref{eq:cairo:2}), 
  we obtain a standard realization of a Cairo pentagonal tiling (Fig.~\ref{fig:cairo0} (b) and Fig.~\ref{fig:cairo} (c)).
\end{example}

\begin{remark}
  The carbon structure with regular hexagonal shaped is called a graphene (see Section \ref{sec:carbon}), 
  and the carbon structure with a Cairo pentagonal shaped is called a penta-graphene \cite{Kawazoe}.
\end{remark}

\begin{figure}[htpb]
  \centering
  \begin{tabular}{ccccc}
    \multicolumn{1}{l}{(a)}
    &\multicolumn{1}{l}{(b)}
    \\
    \includegraphics[bb=0 0 172 173,height=100pt]{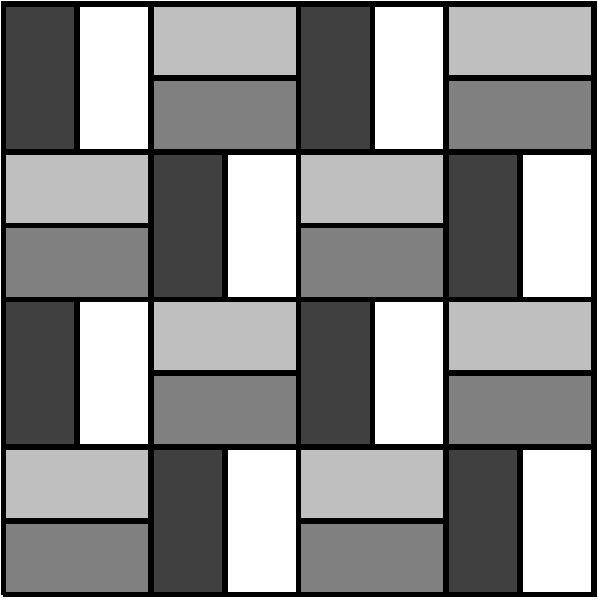}
    &\includegraphics[bb=0 0 171 171,height=100pt]{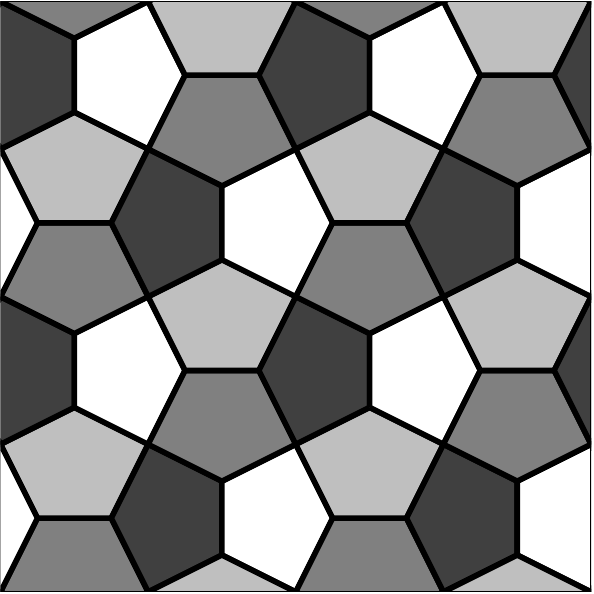}
  \end{tabular}
  \caption{
    (a) A basketweave tiling, 
    (b) a Cairo tiling.
    Both $1$-skeletons of the tiling are topologically equivalent.
  }
  \label{fig:cairo0}
\end{figure}
\begin{figure}[htpb]
  \centering
  \begin{tabular}{ccccc}
    \multicolumn{1}{l}{(a)}
    &\multicolumn{1}{l}{(b)}
    &\multicolumn{1}{l}{(c)}
    \\
    \includegraphics[bb=0 0 233 233,height=150pt]{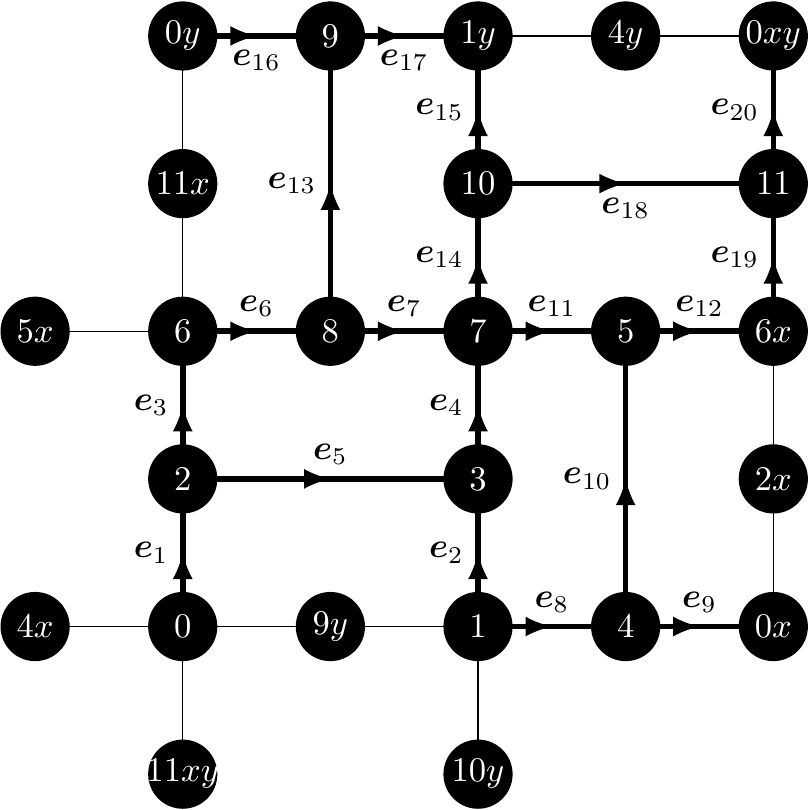}
    &\includegraphics[bb=0 0 233 233,height=150pt]{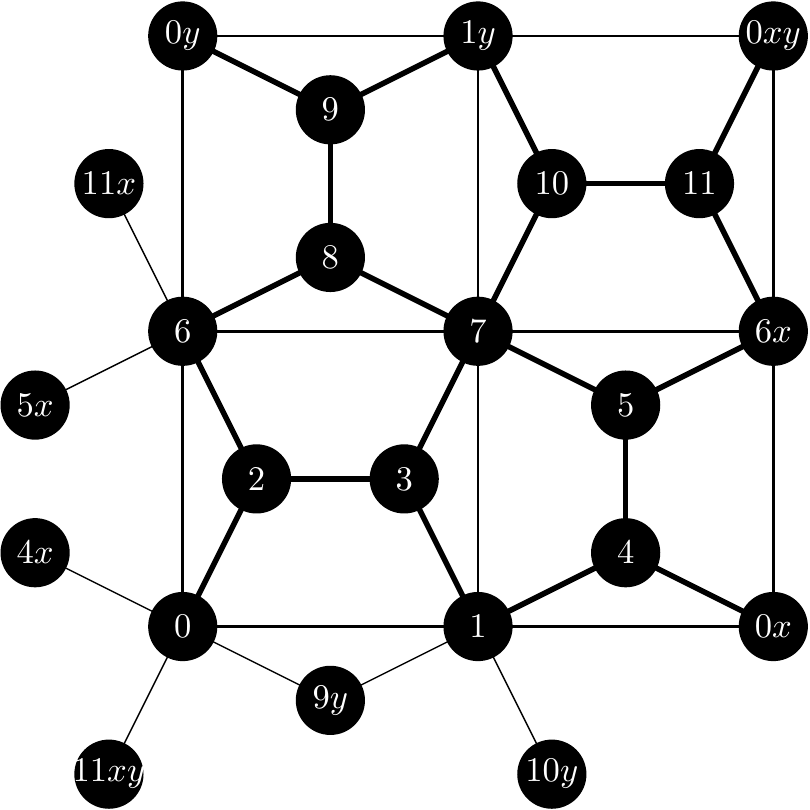}
    &\includegraphics[bb=0 0 111 147,height=150pt]{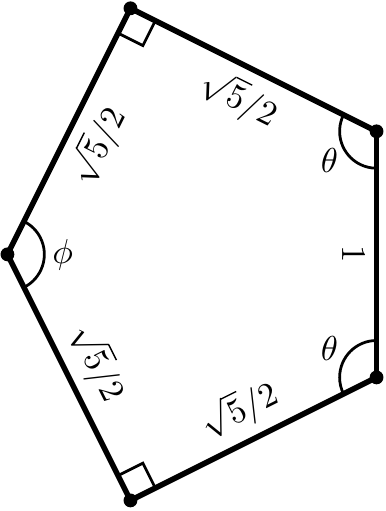}
  \end{tabular}
  \caption{
    (a) Numbering of vertices and edges of a fundamental region of the basketweave tiling, 
    (b) a fundamental region of a standard realization of the Cairo tiling, 
    (c) the congruent-pentagon of the standard realization.
    The ratio of length of edges is $1:(\sqrt{5}/2)$, 
    and angles are $\cos(\theta) = -1/\sqrt{5}$ and $\cos(\phi) = -3/5$ 
    ($\theta \sim 116.57^\circ$ and $\phi \sim 126.87^\circ$).
  }
  \label{fig:cairo}
\end{figure}
%%%%%%%%%%%%%%%% 
\begin{remark}
  A Cairo pentagonal tiling is constructed by line segments joining four vertices
  of a square as in Fig.~\ref{fig:cairo3} (a) and (b).
  Define three kind of energies $L$, $E$ and $C$ by 
  \begin{displaymath}
    \begin{aligned}
      L(t) &= |\v{a} - \v{\alpha}| + |\v{b} - \v{\alpha}| + |\v{c} - \v{\beta}| + |\v{d} - \v{\beta}| + |\v{\alpha} - \v{\beta}|
      = 
      1-2t + 2 \sqrt{1+4t^2}, 
      \\
      E(t) &= |\v{a} - \v{\alpha}|^2 + |\v{b} - \v{\alpha}|^2 + |\v{c} - \v{\beta}|^2 + |\v{d} - \v{\beta}|^2 + |\v{\alpha} - \v{\beta}|^2
      = 
      8 t^2 - 4 t + 2, 
      \\
      C(t) &= |\v{a} - \v{\alpha}|^{-1} + |\v{b} - \v{\alpha}|^{-1} + |\v{c} - \v{\beta}|^{-1} + |\v{d} - \v{\beta}|^{-1} + |\v{\alpha} - \v{\beta}|^{-1}.
    \end{aligned}
  \end{displaymath}
  For each $t \in (-1/2, 1/2)$, the configuration in Fig.~\ref{fig:cairo3} (b) yields a monohedral pentagon tiling.
  The energy $L$ attains its minimum at $t = 1/(2\sqrt{3})$, and then
  the angle $\theta = \theta(t)$ in Fig.~\ref{fig:cairo3} (a) satisfies $\cos(\theta) = -1/2$, ($\theta = 2\pi/3$).
  The minimum of $L$ gives us the configuration of the minimum length of line segments, 
  On the other hand, 
  the energy $E$ attains its minimum at $t = 1/4$, and then
  the angle $\theta$ satisfies $\cos(\theta) = -1/\sqrt{5}$.
  The minimum of $E$ gives us 
  a standard realization of the Cairo pentagonal tiling (see also Fig.~\ref{fig:cairo} (c)).
  The energy $C$ is based on the Coulomb repulsive force, 
  and attains its local minimum at $t \sim 0.17264$, 
  the angle $\theta$ satisfies $\cos(\theta) \sim -0.326374$.
\end{remark}
\begin{figure}[htpb]
  \centering
  \begin{tabular}{ccccc}
    \multicolumn{1}{l}{(a)}
    &\mbox{}
    &\multicolumn{1}{l}{(b)}
    &\mbox{}
    &\multicolumn{1}{l}{(c)}
    \\
    \includegraphics[bb=0 0 114 108,height=105pt]{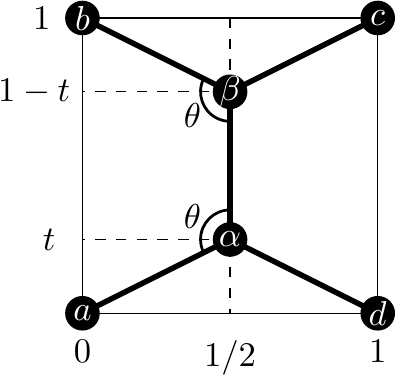}
    &\mbox{}
    &\includegraphics[bb=0 0 172 87,height=85pt]{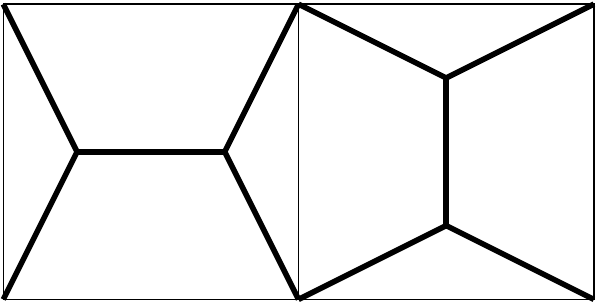}
    &\mbox{}
    &\includegraphics[bb=0 0 107 86,height=85pt]{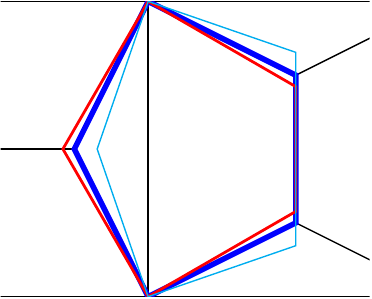}
  \end{tabular}
  \caption{
    (a) Move $\alpha$ and $\beta$ on the line $x = 1/2$, and calculate minimum of $L(t)$ and $E(t)$, 
    (b) a building block of the pentagonal tiling, 
    (c) the red, blue, and cyan pentagons are minimizers of $L$, $E$ and $C$, respectively.
  }
  \label{fig:cairo3}
\end{figure}
%%%%%%%%%%%%%%%% 
\begin{remark}
  This algorithm is easily programmable 
  by using Kruskal's and Dijkstra's algorithms, and the Cholesky decomposition.
  To calculate the matrix $A$ and vectors $a(e)$, $b(e)$, 
  it is easy to set $e_i = (0, \ldots, 1, \ldots, 0) \in \R^{\numberOf{E}}$.
\end{remark}
%%%%%%%%%%%%%%%% 

\begin{example}
  By using Mathematica, coordinates of vertices of standard realizations of topological crystals 
  are easy to compute, 
  if we obtain a $\Z$-basis of $H_1(X_0, \Z)$ and a $\Z$-basis of vanishing subspace $H$
  (and a spanning tree of $X_0$).
  The following is a sample code of Mathematica to compute vertices of a kagome lattice
  (lines 1 and 3 are specific datas of a kagome lattice).
  \begin{center}
    \shadowbox{
      \begin{minipage}{0.95\linewidth}
        {\ttfamily numberOfEdges=6;b=4;d=2;}\\
        {\ttfamily e = IdentityMatrix[numberOfEdges];}\\
        {\ttfamily alpha = \{e[[1]]-e[[4]], e[[2]]-e[[5]], e[[1]]+e[[2]]+e[[3]], e[[4]]+e[[5]]+e[[6]]\};}\\
        {\ttfamily matrixA = alpha.Transpose[alpha];}\\
        {\ttfamily matrixb = Table[e[[j]].alpha[[i]], \{i, 1, b\}, \{j, 1, numberOfEdges\}];}\\
        {\ttfamily matrixa = Inverse[matrixA].matrixb;}\\
        {\ttfamily matrixA11 = matrixA[[1;;d,1;;d]]; matrixA22 = matrixA[[d+1;;b,d+1;;b]]; matrixA12 = matrixA[[1;;d,d+1;;b]]; matrixA21 = matrixA[[d+1;;b,1;;d]]; matrixB = matrixA11 - matrixA12.Inverse[matrixA22].matrixA21;}\\
        {\ttfamily matrixprojecta = matrixa[[1;;d]];}\\
        {\ttfamily beta = CholeskyDecomposition[matrixB];}\\
        {\ttfamily beta.matrixprojecta}
      \end{minipage}
    }
  \end{center}
  The output of this code is 
  \begin{displaymath}
    \begin{bmatrix}
      {1}/{\sqrt{3}} &
      -{1}/(2 \sqrt{3}) &
      -{1}/(2 \sqrt{3}) &
      -{1}/{\sqrt{3}} &
      {1}/(2 \sqrt{3}) &
      {1}/(2 \sqrt{3}) \\
      0 & {1}{/2} &
      -{1}/{2} & 0 &
      -{1}/{2} & {1}/{2}
      \\
    \end{bmatrix}, 
  \end{displaymath}
  which expresses coordinates of vertices of a kagome lattice.
  To obtain complete datas of standard realizations, 
  we should obtain datas of building blocks by using datas of the shortest paths from an origin.
\end{example}
%%%%%%%%%%%%% 5
%%%%%%%%%%%%%%%% 
\begin{remark}
  Crystallographers often call periodic realizations in $\R^2$ and $\R^3$ of graphs {\em $2$-net} and {\em $3$-net}, 
  respectively.
  Names of $2$-net of each lattice are
  \begin{flushleft}
    \begin{tabular}{ccl}
      {\bfseries sq1} & \mbox{} & the regular square lattice, \\
      {\bfseries hcb} & \mbox{} & the regular hexagonal lattice (honeycomb lattice) \\
      {\bfseries hx1} & \mbox{} & the regular triangular lattice \\
      {\bfseries kgm} & \mbox{} & the regular kagome lattice, \\
      {\bfseries mcm} & \mbox{} & the $1$-skeleton of Cairo pentagonal tiling, \\
    \end{tabular}
  \end{flushleft}
  and names of $3$-net of each lattice are
  \begin{flushleft}
    \begin{tabular}{ccl}
      {\bfseries pcu} & \mbox{} & the regular cubic lattice, \\
      {\bfseries dia} & \mbox{} & the diamond lattice, \\
      {\bfseries src} & \mbox{} & the gyroid lattice, \\
      {\bfseries crs} & \mbox{} & the 3D kagome lattice of type I, \\
      {\bfseries lvt} & \mbox{} & the 3D kagome lattice of type II. \\
    \end{tabular}
  \end{flushleft}
  Lists of $2$-net and $3$-net are available in EPINET \cite{EPINET}.
\end{remark}
%%%%% END of document

%%%%% END of document

%%%%% \input{carbon.tex}
% -*- coding: utf-8 -*-
\subsection{Carbon structures and standard realizations}
\label{sec:carbon}
In this section, we consider carbon crystal structures via standard realizations.
\par
Graphene is an allotrope of carbons, and is 2-dimensional crystal structure.
Each carbon atom binds chemically other three carbon atoms by $sp^2$-orbitals
(see Fig.~\ref{fig:carbonasstandard}).
In mathematical view points, 
a graphene is a standard realization of a regular hexagonal lattice.
A fundamental piece (Fig.~\ref{fig:carbonasstandard}) is a graph with four points, 
where each point locates at vertices of regular triangle and its barycenter.
Translating the fundamental piece by $\v{\alpha}_1$ and $\v{\alpha}_2$ 
with $|\v{\alpha}_i| = 1$ and $\inner{\v{\alpha}_1}{\v{\alpha}_2} = (1/2)|\v{\alpha}_1|\,|\v{\alpha}_2|$, we obtain a structure of graphenes.
\begin{figure}[htp]
  \centering
  \begin{tabular}{ccccccc}
    \multicolumn{1}{l}{(a)}
    &{}
    &\multicolumn{1}{l}{(b)}
    &{}
    &\multicolumn{1}{l}{(c)}
    &{}
    &\multicolumn{1}{l}{(d)}\\
    \includegraphics[bb=0 0 120 67,scale=0.5]{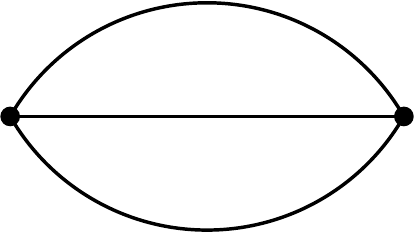}
    &\mbox{}
    &\includegraphics[bb=0 0 105 91,scale=0.5]{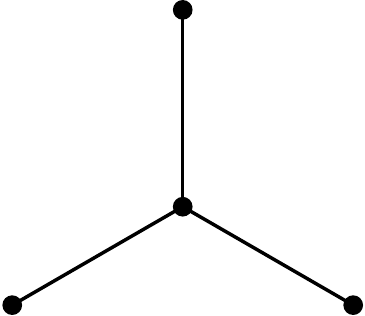}
    &\mbox{}
    &\includegraphics[bb=0 0 118 106,scale=0.75]{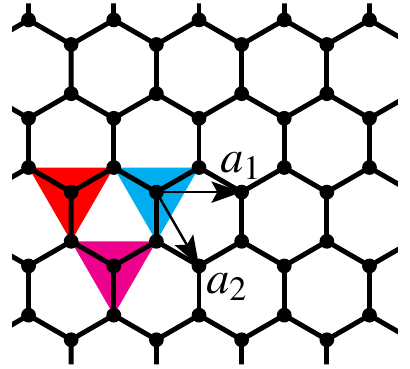}%
    &\mbox{}
    &\includegraphics[bb=0 0 104 104,scale=0.75]{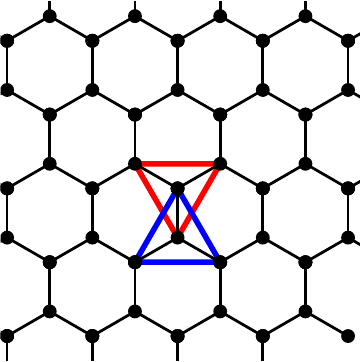}
  \end{tabular}
  \caption{
    (a) The base graph of a regular hexagonal lattice, 
    (b) a fundamental piece of a regular hexagonal lattice, 
    (c) a graphene structure, 
    which is constructed by (b) and its translations, 
    (d) the barycenter of the blue regular triangle is a vertex of the red regular triangle.
    The blue triangle is consisted by $\v{w}_1$, $\v{w}_2$, and $\v{w}_3$
    (by using notations in Example 
    \ref{example:standard:hexagonal}), 
    then the red is consisted by $\v{v}_0$, $\v{v}_0 + \v{\alpha}_1$, and $\v{v}_0 + \v{\alpha}_1 - \v{\alpha}_2$.
    The barycenters of blue and red are $\v{v}_0$ and $\v{w}_1$, respectively.
  }
  \label{fig:carbonasstandard}
\end{figure}

\par
Diamond is also an allotrope of carbons, and is 3-dimensional crystal structure.
Each carbon atom binds chemically other four carbon atoms by $sp^3$-orbitals
(see Fig.~\ref{fig:diamondasstandard}).
In mathematical view points, 
a diamond is a standard realization of a graph, which can be called regular tetrahedral graph.
A fundamental piece (Fig.~\ref{fig:diamondasstandard}) is a graph with five points, 
each points located at vertices of regular tetrahedron and its barycenter.
Translating the fundamental piece by $\v{\alpha}_1$, $\v{\alpha}_2$, and $\v{\alpha}_3$
with $|\v{\alpha}_i| = 1$ and $\inner{\v{\alpha}_i}{\v{\alpha}_j} = (1/3)|\v{\alpha}_i|\,|\v{\alpha}_j|$ if $i\not=j$, 
we obtain a structure of diamonds.
\begin{figure}[htp]
  \centering
  \begin{tabular}{ccccccc}
    \multicolumn{1}{l}{(a)}
    &{}
    &\multicolumn{1}{l}{(b)}
    &{}
    &\multicolumn{1}{l}{(c)}
    &{}
    &\multicolumn{1}{l}{(c)}\\
    \includegraphics[bb=0 0 120 67,scale=0.5]{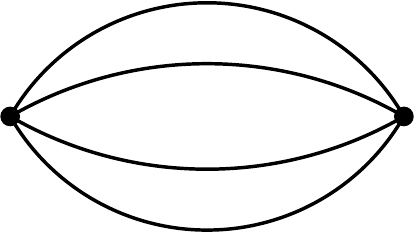}
    &\mbox{}
    &\includegraphics[bb=0 0 116 106,scale=0.5]{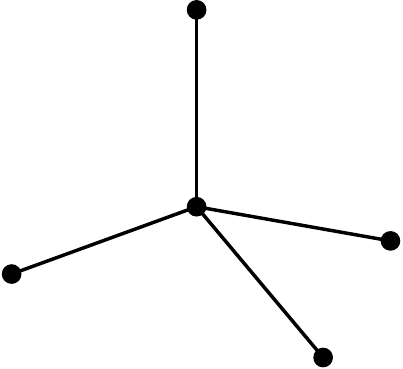}
    &\mbox{}
    &\includegraphics[bb=0 0 640 640,height=100pt]{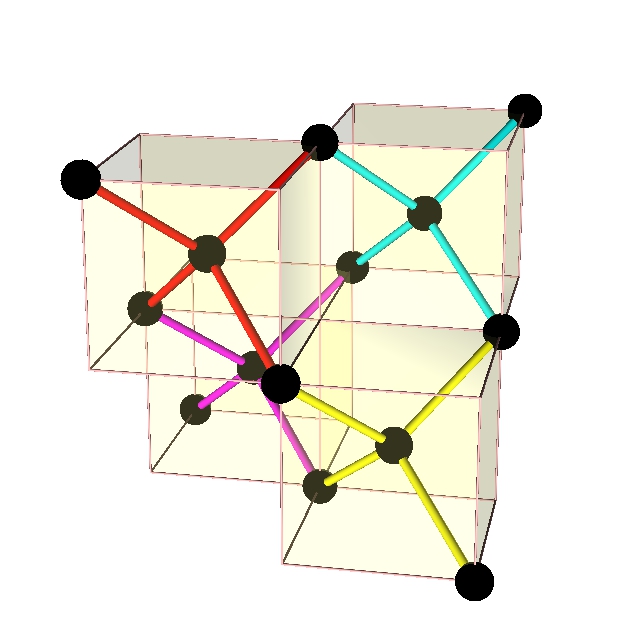}%
    &\mbox{}
    &\includegraphics[bb=0 0 264 292,height=100pt]{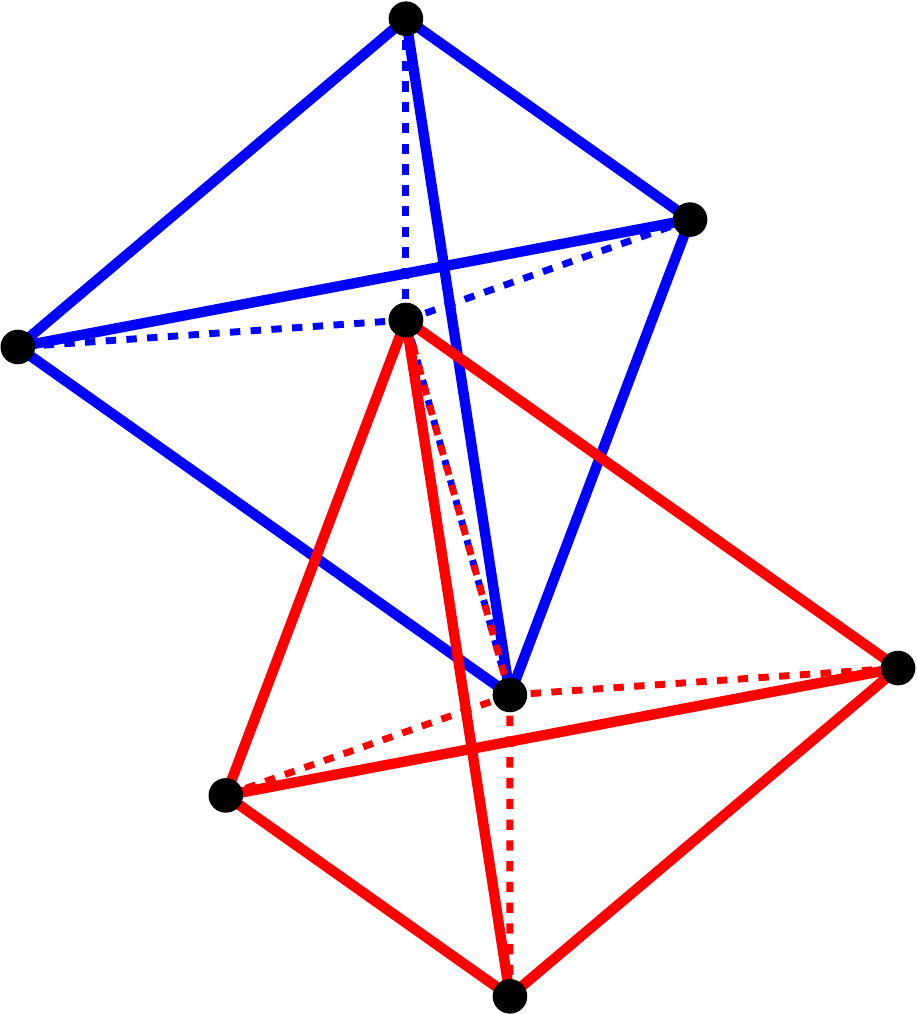}
  \end{tabular}
  \caption{
    (a) The base graph of a diamond structure, 
    (b) a fundamental piece of a diamond structure, 
    whose vertices are located at vertices of a regular tetrahedron and its barycenter, 
    (c) a diamond structure.
    A diamond structure is constructed by (b) and its translations.
    (d) the barycenter of the blue regular tetrahedron is a vertex of the red regular tetrahedron.
    The blue tetrahedron is consisted by $\v{v}_1$, $\v{v}_2$, $\v{v}_3$, and $\v{v}_4$
    (by using notations in Example 
    \ref{example:standard:diamond}), 
    then the red is consisted by $\v{v}_0$, $\v{w}_1 - \v{w}_2 = \v{v}_0 + \v{\alpha}_1 - \v{\alpha}_2$,  
    $\v{w}_1 - \v{w}_3 = \v{v}_0 + \v{\alpha}_1 - \v{\alpha}_3$, 
    and 
    $\v{v}_0 + \v{w}_1 - \v{w}_4 = \v{v}_0 + \v{\alpha}_1$.
    The barycenters of blue and red are $\v{v}_0$ and $\v{w}_1$, respectively.
  }
  \label{fig:diamondasstandard}
\end{figure}

By a textbook of physical chemistry, the space group of diamond structure is $F\!d\overline{3}m$, 
which expresses face-centric structure with a glide reflection, three improper rotations, and certain reflections.
Diamond structures are constructed shown in Fig.~\ref{fig:fd3m};
however it is difficult to realize for mathematicians.
On the other hand, the construction diamond structures by standard realizations is easy for mathematicians.
\begin{figure}[htbp]
  \centering
  \begin{tabular}{cccc}
    \multicolumn{1}{l}{(a)}
    &\multicolumn{1}{l}{(b)}
    &\multicolumn{1}{l}{(c)}
    \\
    \includegraphics[bb=0 0 640 640,height=130pt]{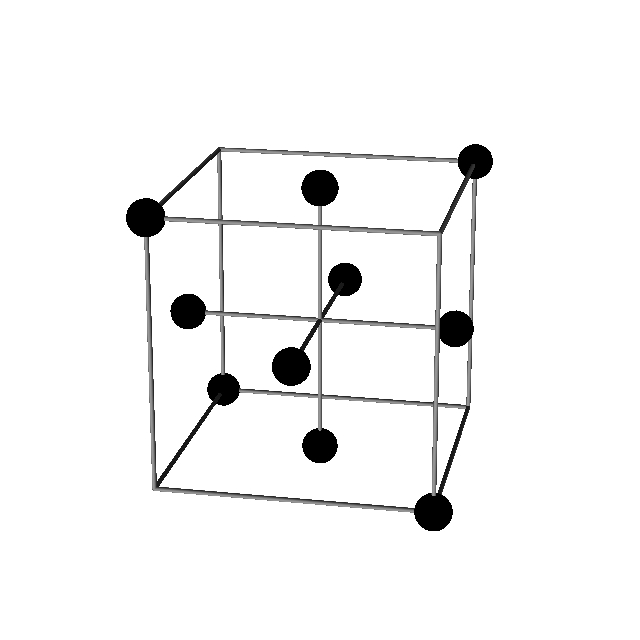}%
    &\includegraphics[bb=0 0 640 640,height=130pt]{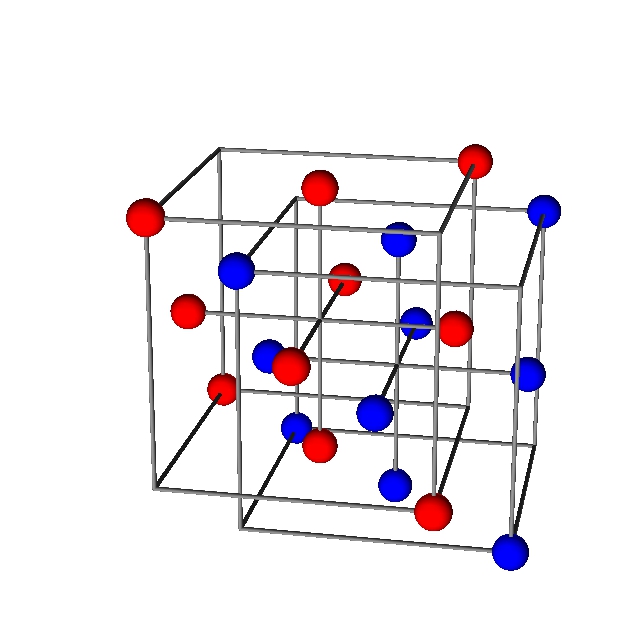}%
    &\includegraphics[bb=0 0 640 640,height=130pt]{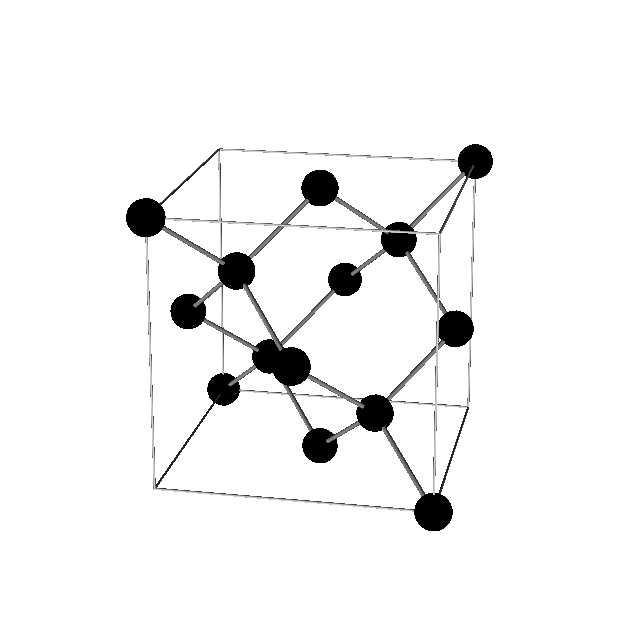}%
  \end{tabular}
  \caption{
    How to construct a diamond structure:
    (a) prepare a face-centric structure, 
    (b) duplicate it and translate to a diagonal direction, 
    and then (c) we obtain a diamond structure.
  }
  \label{fig:fd3m}
\end{figure}
\par
$K_4$ structures are obtained by standard realizations of the $K_4$ graph.
The $K_4$ graph is the complete graph with four vertices (each vertex connects to all other vertices).
Each carbon atom of $K_4$ structures binds chemically other three atoms by $sp^2$-orbitals (see Fig.~\ref{fig:k4asstandrd}).
A fundamental piece is a graph with seven points.
Translation vectors $\{\v{\alpha}_i\}_{i=1}^3$ satisfy 
$|\v{\alpha}_i| = 1$ and $\inner{\v{\alpha}_i}{\v{\alpha}_j} = -(1/3)|\v{\alpha}_i|\,|\v{\alpha}_j|$ ($i\not=j$).
Moreover, $K_4$ structures have chirality, that is, 
A $K_4$ structure and its mirror image are not same.
\par
Physical property of the $K_4$ carbon is computed by Itoh--Kotani--Naito--Sunada--Kawazoe--Adschiri \cite{K4}, 
it is physically meta-stable and metallic; however it has not composed yet.
Recently, 
Mizuno--Shuku--Matsushita--Tsuchiizu--Hara--Wada--Shimizu--Awaga
compose a $K_4$ structure other than carbons \cite{Awaga}.
Their structure is a molecular-$K_4$, 
a radical molecule NDI-$\triangle (-)$ consists a $K_4$ crystal.
\par
\begin{figure}[htbp]
  \centering
  \begin{tabular}{ccc}
    \includegraphics[bb=0 0 1583 1106,height=150pt]{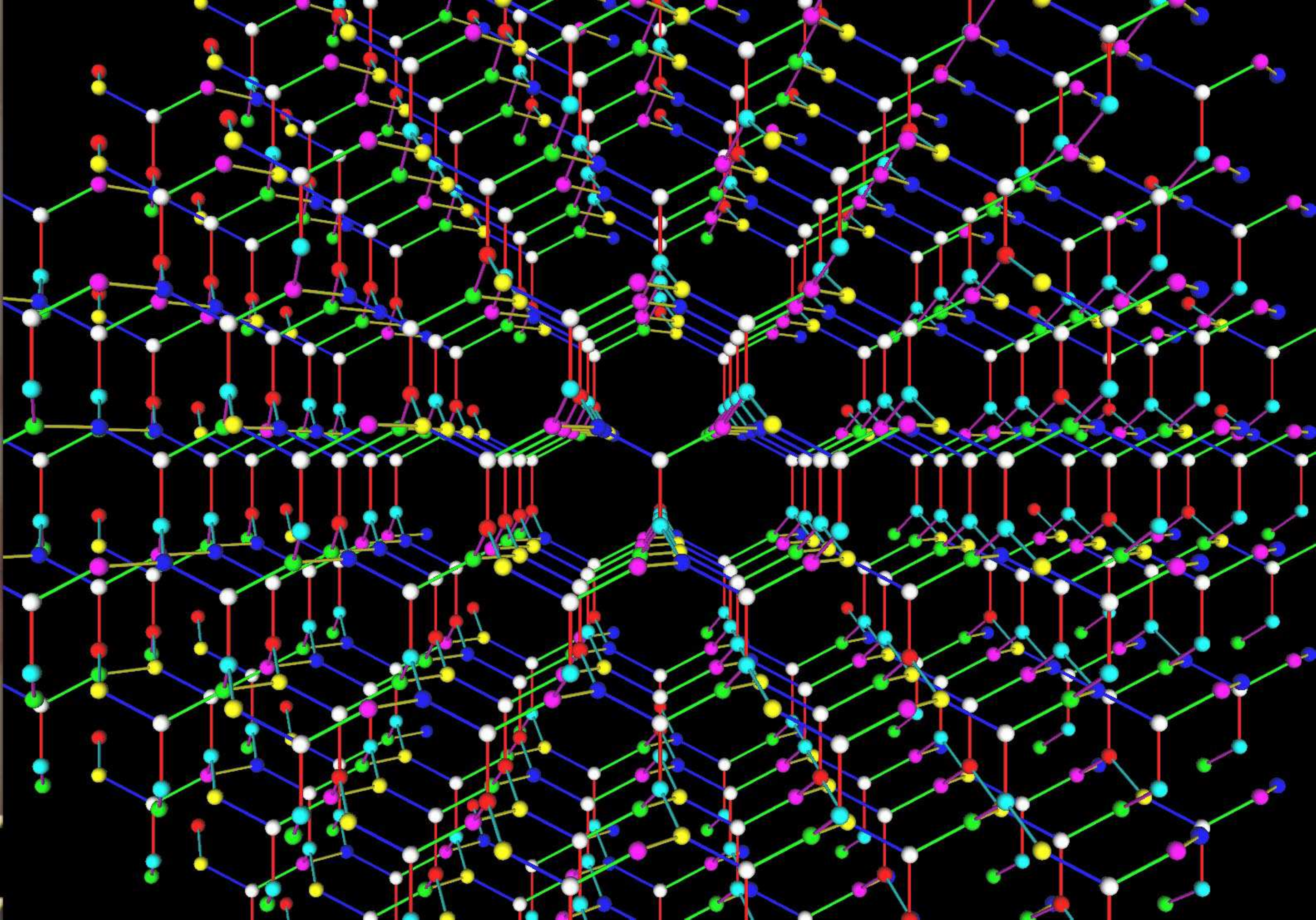}
  \end{tabular}
  \caption{
    A $K_4$ structure, 
    this is the image of the cover page of Notices Amer.~Math.~Soc., {\bfseries 55},
    drawn by the author.
  }
  \label{fig:k4asstandrd}
\end{figure}
%%%%% END of document

%%%%% END of document

%%%%% \input{negativelycarbon.tex}
% -*- coding: utf-8 -*-
\section{Negatively curved carbon structures}
%%%%%%%%%%%%%%%% 
%%%%% \input{negative_1.tex}
% -*- coding: utf-8 -*-
\subsection{Carbon structures as discrete surfaces}
%%%%%%%%%%%%%%%% 
In 1990's, several new $sp^2$-carbon structures, fullerenes (including $\mathrm{C}_{60}$), graphene, and carbon nanotubes were found (See Fig.~\ref{fig:carbonstructures}).
\begin{figure}[htb]
  \centering
  \begin{tabular}{ccc}
    \multicolumn{1}{l}{(a)}
    &\multicolumn{1}{l}{(b)}
    &\multicolumn{1}{l}{(c)}\\
    \includegraphics[bb=0 0 256 256,height=90pt]{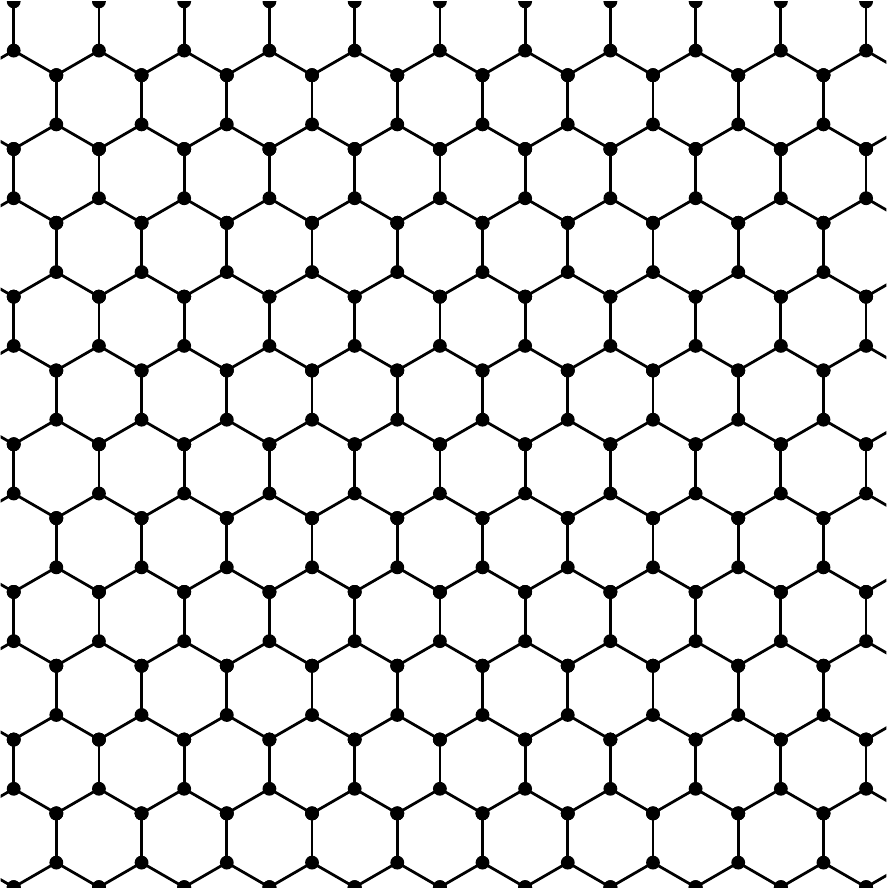}
    & \includegraphics[bb=0 0 640 640,height=90pt]{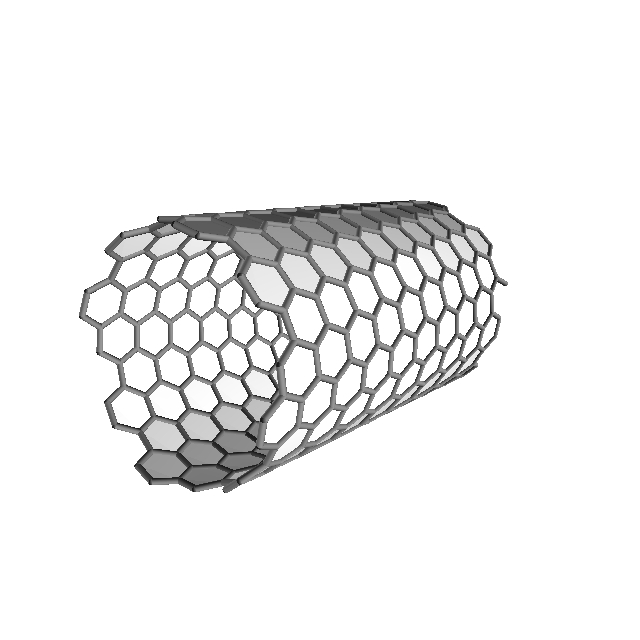} 
    & \includegraphics[bb=0 0 640 640,height=90pt]{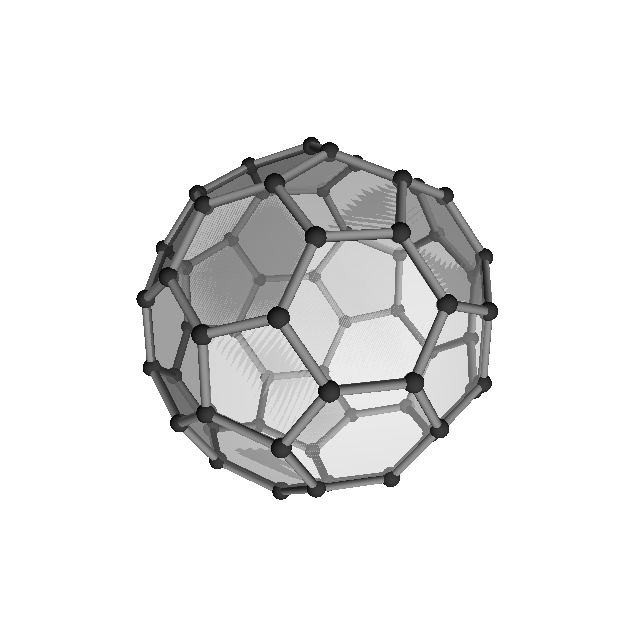}\\
  \end{tabular}
  \caption{(a) Graphene, (b) (single wall) carbon nanotube, (c) $\mathrm{C}_{60}$ (an example of fullerenes).}
  \label{fig:carbonstructures}
\end{figure}
These structures look like surfaces in $\R^3$.
For example, a graphene, $\mathrm{C}_{60}$, and a nanotube are similar to 
a plane, a sphere, and a cylinder, respectively.
Each continuous surface in the above has non-negatively curved, 
i.\,e., the Gauss curvature of a sphere is positive, 
and theses of a cylinder and plane are zero.
Hence, it is a natural question if
an $sp^2$-carbon structure which looks like a negatively curved surface exists or not.
\par
In the followings, we consider $sp^2$-carbon structures as ``trivalent discrete surface''
(realizations of $3$-regular graphs in $\R^3$).
Moreover, we assume that graphs are oriented surface graphs, 
that is, each graphs is realized on a oriented surface without self-intersections.
By the property of surface graphs, we can define the notion of ``faces'' (simple closed path) for trivalent discrete surfaces.
\begin{definition}
  For an oriented surface graph $X = (V, E)$, 
  the {\em Euler number} $\chi(X)$ of $X$ is defined by 
  \begin{equation}
    \label{eq:eulernumberdefinition}
    \chi(X) = \numberOf{F} - \numberOf{E} + \numberOf{V}, 
  \end{equation}
  where $\numberOf{F}$ is the number of faces of $X$.
\end{definition}
Note that the Euler number $\chi(X)$ of an oriented graph $X$ is same as the Euler number of the underlying surface.
\begin{proposition}
  \label{claim:euler}
  Assume an oriented surface graph $X = (V, E)$ is trivalent {\upshape(}$3$-regular{\upshape)} graph, 
  then we obtain 
  $F = \sum N_k$, $E = (1/2) \sum k N_k$, $V = (1/3)\sum k N_k$, 
  where $N_k$ is the number of $k$-gon in $X$.
  Hence, we also obtain 
  \begin{equation}
    \label{eq:eulernumber}
    \chi(X) = \sum (1-k/6)N_k.
  \end{equation}
\end{proposition}
\begin{proof}
  Since $X$ is an oriented surface graph, 
  each edges shared by two faces, and hence, we obtain $\numberOf{E} = (1/2)\sum k N_k$.
  Since $X$ is trivalent, 
  each vertex shared by three faces, and hence, we obtain $\numberOf{V} = (1/3) \sum k N_k$.
  Substituting them into (\ref{eq:eulernumberdefinition}), 
  we obtain (\ref{eq:eulernumber}).
\end{proof}

\begin{remark}
  By Proposition \ref{claim:euler}, 
  the number of hexagons does not affect to the Euler number.
  If $X$ is positively curved ($\chi(X) > 0$), then
  at least one $n$-gon ($n \le 5$) should be contained in $X$.
  If $X$ is negatively curved ($\chi(X) < 0$), then
  at least one $n$-gon ($n \ge 7$) should be contained in $X$.
\end{remark}
\begin{table}[htb]
  \centering
  \begin{tabular}{c|cccc|rrr|r}
    & $N_5$ & $N_6$ & $N_7$ & $N_8$ & \multicolumn{1}{c}{$\numberOf{V}$} & \multicolumn{1}{c}{$\numberOf{E}$} & \multicolumn{1}{c}{$\numberOf{F}$} & \multicolumn{1}{|c}{$\chi$} \\
    \hline
    \hline
    $\mathrm{C}_{60}$
    & $12$ & $20$ & $0$ & $0$ & $60$ & $90$ & $32$ & $2$ \\
    \hline
    SWNT $\v{c}=(6,6)$
    & $0$ & $12$ & $0$ & $0$ & $24$ & $36$ & $12$ & $0$ \\
    \hline
    Mackay--Terrones' 
    & $0$ & $90$ & $0$ & $12$ & $192$ & $288$ & $102$ & $-4$ \\
  \end{tabular}
  \caption{
    Number of polygons of $\mathrm{C}_{60}$, 
    a single-wall nanotube (of fundamental region of it), where
    $\v{c}$ is the chiral index of SWNT (see Section \ref{sec:nanotube}), 
    and Mackay--Terrones' structure (see Fig.~\ref{fig:mackay} (b)).
  }
  \label{tab:C60-SWNT}
\end{table}
\par
By Proposition \ref{claim:euler}, 
if there exists an $sp^2$-carbon structure with $\chi(X) < 0$ (negatively curved), 
then at least one $n$-gon ($n \ge 7$) exists in $X$.
In 1991, Mackay--Terrones \cite{Mackay-Terrones:1991} calculated an $sp^2$-carbon structure which is looked like a minimal surface (Schwarz P surface, $\chi(X) = -4$), which contains $12$ of octagons (see also Lenosky--Gonze--Teter--Elser \cite{Lenosky:1992}).
\begin{figure}[htb]
  \centering
  \begin{tabular}{ccc}
    \multicolumn{1}{l}{(a)}
    &\mbox{}
    &\multicolumn{1}{l}{(b)}\\
    \includegraphics[bb=0 0 464 464,height=120pt]{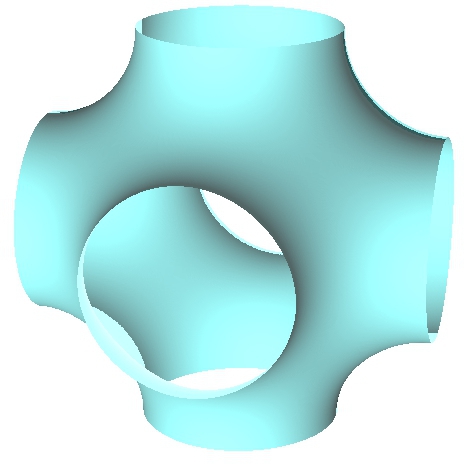}
    &\mbox{}
    &\includegraphics[bb=0 0 640 640,height=120pt]{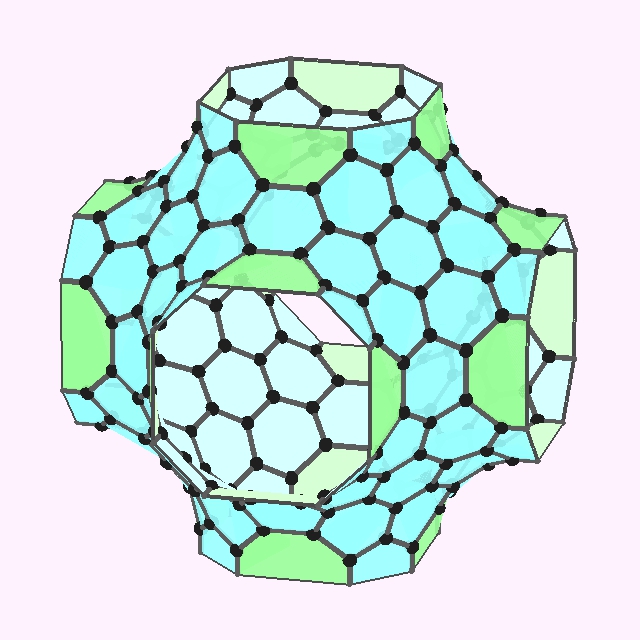}
  \end{tabular}
  \caption{(a) Schwarz P surface, which is a triply periodic minimal surface 
    (the Gauss curvature $K < 0$, the Euler number $\chi = -4$, 
    and the genus $g = -3$), 
    whose period lattice $\{\v{e}_i\}_{i=1}^3$ satisfies $\inner{\v{e}_i}{\v{e}_j} = \delta_{ij}$.
    (b) Mackay--Terrones' structure, 
    the period lattice is the same as (a).
    Note that green faces in (b) are octagons (see also Table \ref{tab:C60-SWNT}).
  }
  \label{fig:mackay}
\end{figure}
%%%%% END of document

%%%%% \input{negative_2.tex}
% -*- coding: utf-8 -*-
\subsection{Construction of negatively curved carbon structures via standard realizations}
%%%%%%%%%%%%%%%% 
Tagami--Liang--Naito--Kawazoe--Kotani \cite{Tagami:2014} constructed negatively curved carbon crystals, 
which are different from the Mackay--Terrones' structure, by using standard realizations of topological crystals. 
\par
The fundamental region of Mackay--Terrones' structure has octahedral symmetry, which is same symmetry of 
truncated octahedrons.
Truncated octahedrons consist eight hexagons, which have $D_6$-symmetry (see \ref{fig:symmetry:Mackay}).
\begin{figure}[htb]
  \centering
  \includegraphics[bb=0 0 477 136,height=120pt]{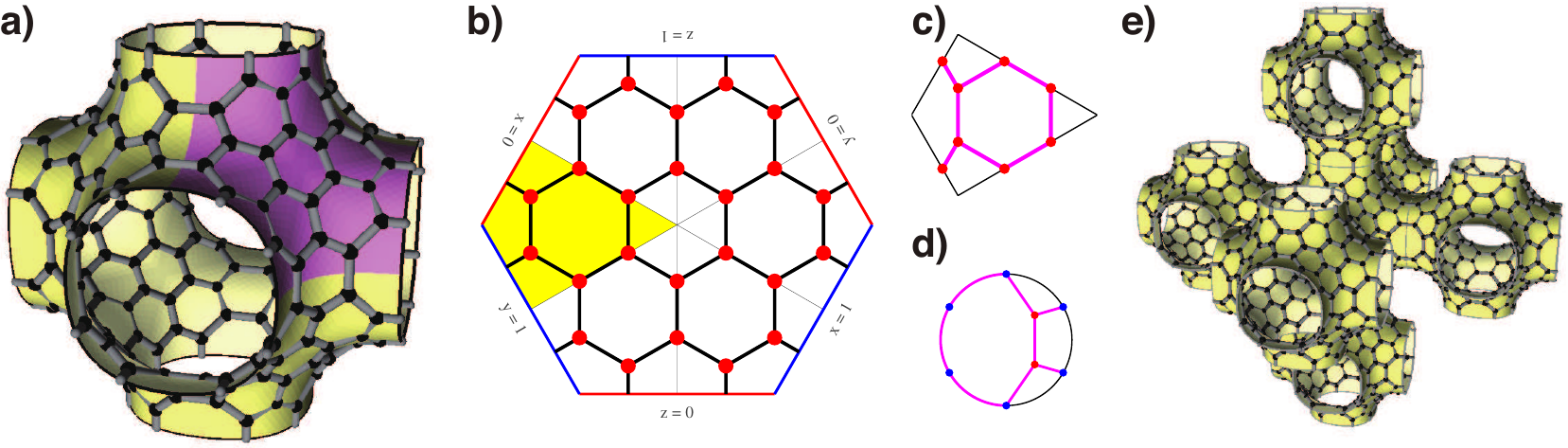}
  \caption{
    Symmetry of the Mackay--Terrones' structure.
    (c) is a fundamental region of $D_6$-action on a hexagon, which is diffeomorphic to (d).
    See also Fig.~\ref{fig:mackaydft}
    (Tagami--Liang--Naito--Kawazoe--Kotani \cite{Tagami:2014}).
  }
  \label{fig:symmetry:Mackay}
\end{figure}
Our method is, 
1) we classify and construct a graph in a fundamental region of $D_6$-action on a hexagon, 
which is a trivalent graph when we extend to hexagon by $D_6$-action (Fig.~\ref{fig:symmetry:Mackay} (c)), 
2) we extends the graph obtained in 1 to the trivalent graph on a hexagon (Fig.~\ref{fig:symmetry:Mackay} (b)), 
3) we extends the graph obtained in 2 to the graph on a truncated octahedron (Fig.~\ref{fig:symmetry:Mackay} (a)), 
and 4) we calculate a standard realization of the graph obtained in 3.
The structure obtained in 4 is a candidate of $sp^2$-carbon structure with $K < 0$ ($\chi = -4$).
In fact, we prove the following result.
\begin{theorem}[Tagami--Liang--Naito--Kawazoe--Kotani \cite{Tagami:2014}]
  The equation to obtain standard realization is linear: 
  $\Delta \v{x} = \v{b}$, 
  where
  $b_i = \pm \v{e}_\alpha$, 
  if a vertex $v_i$ is adjacent to a vertex in neighbouring cell.
  The linear equation is solvable if and only if 
  $\sum b_i = 0$.
  The period lattice $\{\v{e}_i\}$ which gives a standard realization is cubic,  
  i.\,e., $\v{e} = (\v{e}_1, \v{e}_2, \v{e}_3)$ is 
  a period lattice of a standard realization, 
  then 
  \begin{math}
    \v{e}^T \v{e} = \v{E}
  \end{math}.
  If $\Phi$ is a standard realization, then 
  $\Aut(X) \subset \Aut(\Phi(X))$, 
  and hence $X$ has the same symmetries with Mackay--Terrones' structure.
\end{theorem}
In our method, we do not solve the equation (\ref{eq:standard}).
Instead, we only solve the harmonic equation (\ref{eq:harmonic}) using cubic periodicity.
Hence, we should prove that the realization is standard.
To prove it, we use the Lagrange multiplier, and the result of the Lagrange multiplier is 
\begin{math}
  \v{e}^T \v{e} = \v{E}
\end{math}.
This shows that the harmonic realizations with the cubic lattice is standard.
\begin{figure}[htb]
  \centering
  \begin{tabular}{c|c|c|c}
    \multicolumn{1}{l}{(a)}
    &\multicolumn{1}{l}{(b)}
    &\multicolumn{1}{l}{(c)}
    &\multicolumn{1}{l}{(d)}\\
    \includegraphics[bb=0 0 161 142,height=75pt]{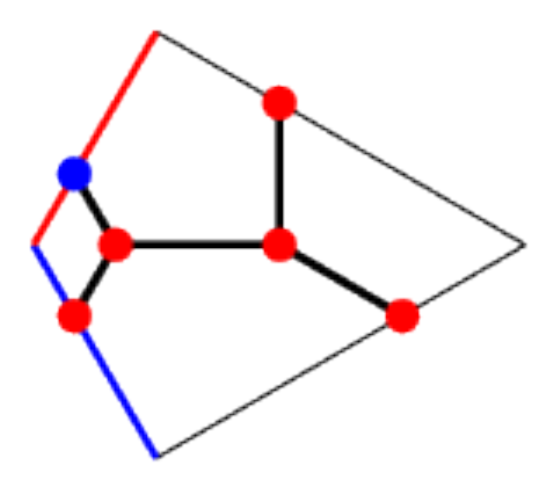} 
    &\includegraphics[bb=0 0 161 142,height=75pt]{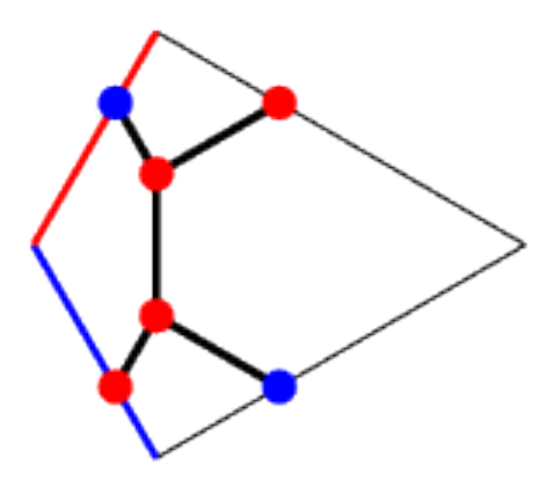}
    &\includegraphics[bb=0 0 161 142,height=75pt]{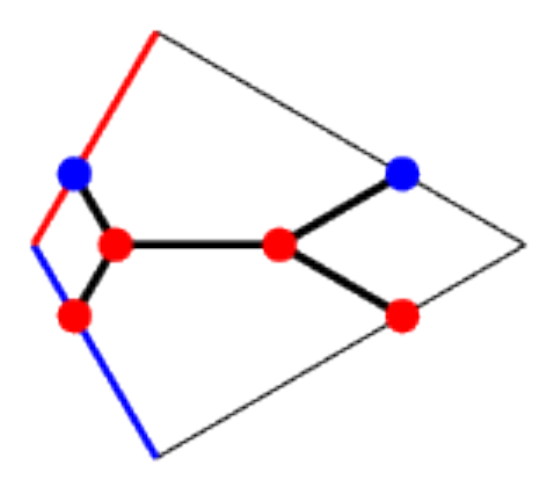} 
    &\includegraphics[bb=0 0 161 142,height=75pt]{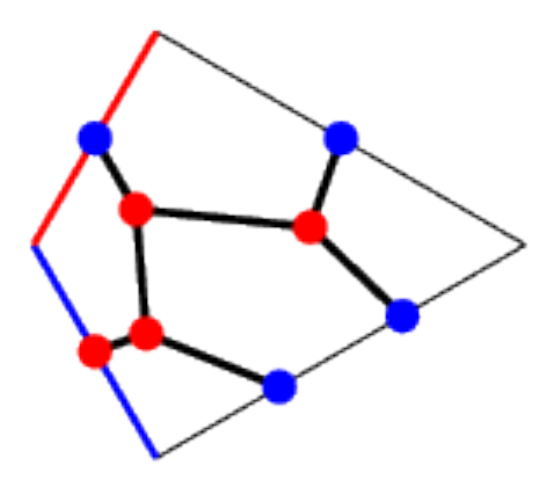} \\
    \includegraphics[bb=0 0 305 322,height=75pt]{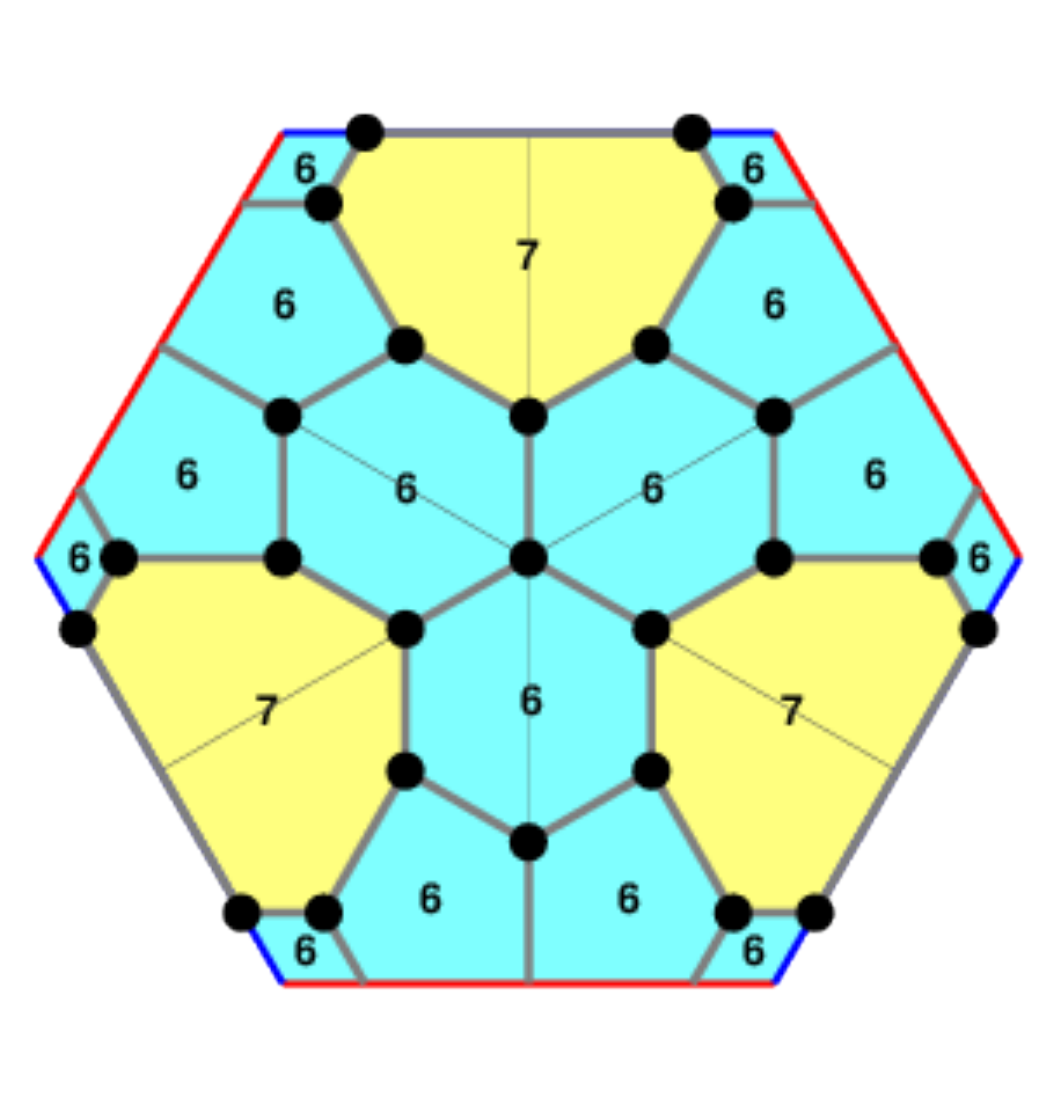} 
    &\includegraphics[bb=0 0 305 322,height=75pt]{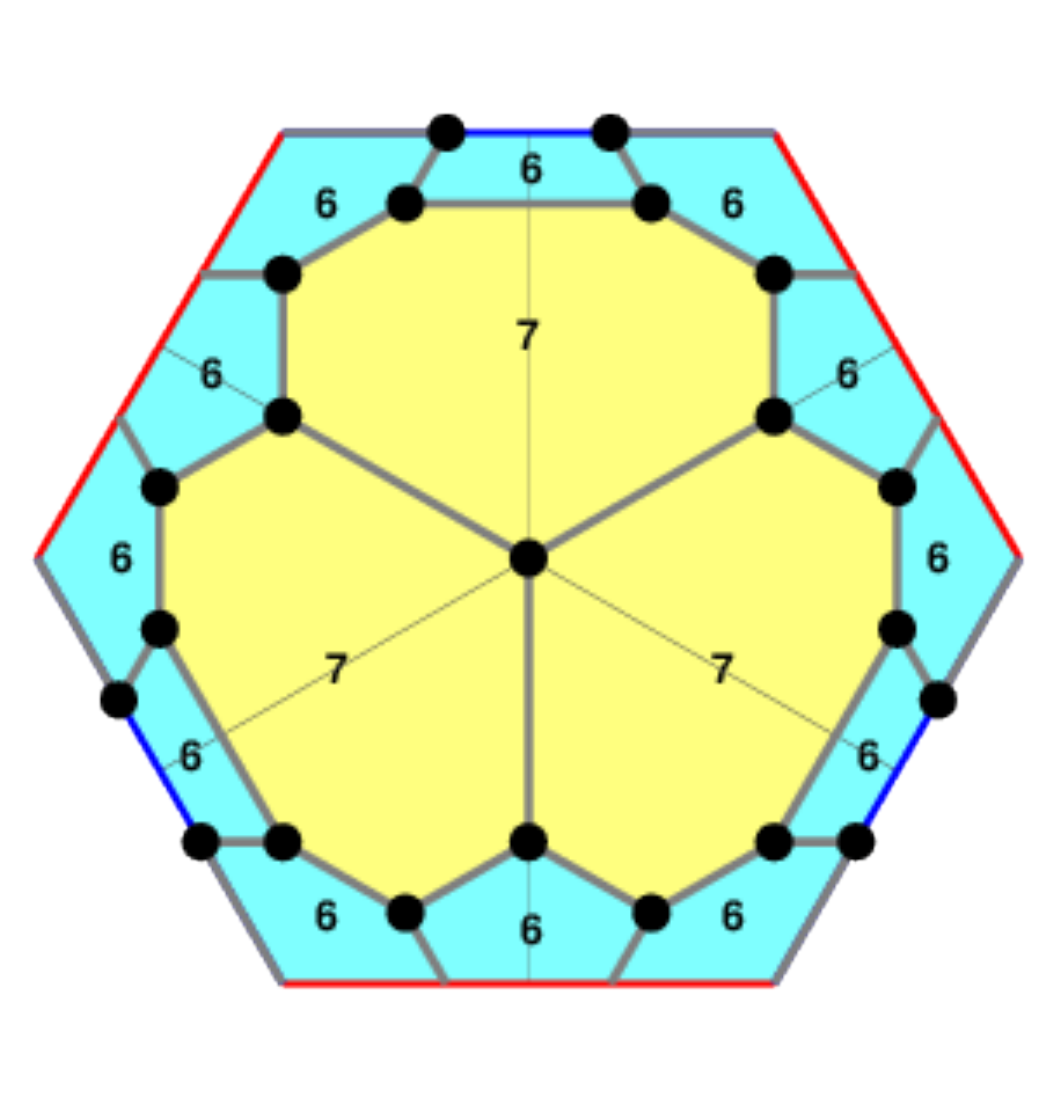} 
    &\includegraphics[bb=0 0 305 322,height=75pt]{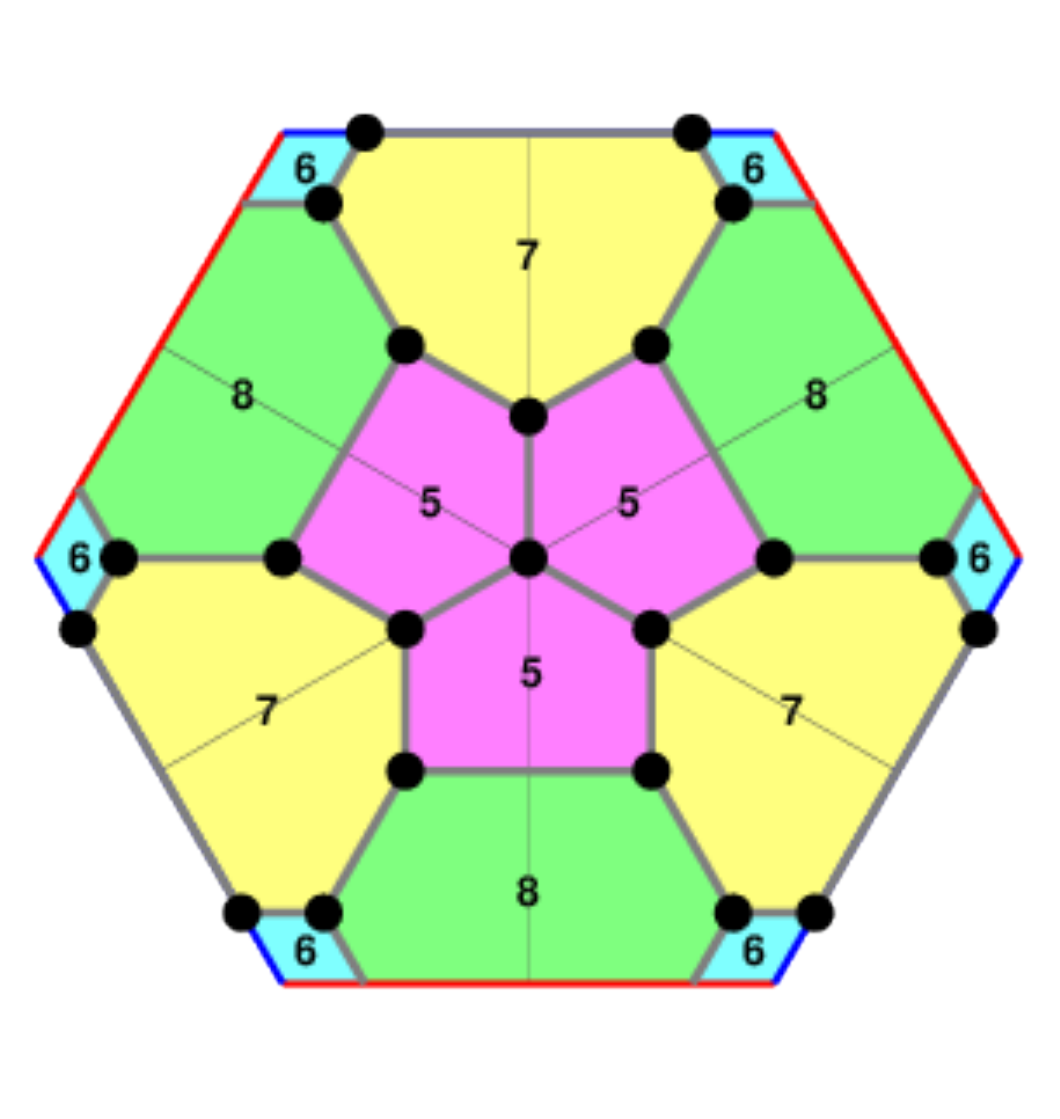} 
    &\includegraphics[bb=0 0 305 322,height=75pt]{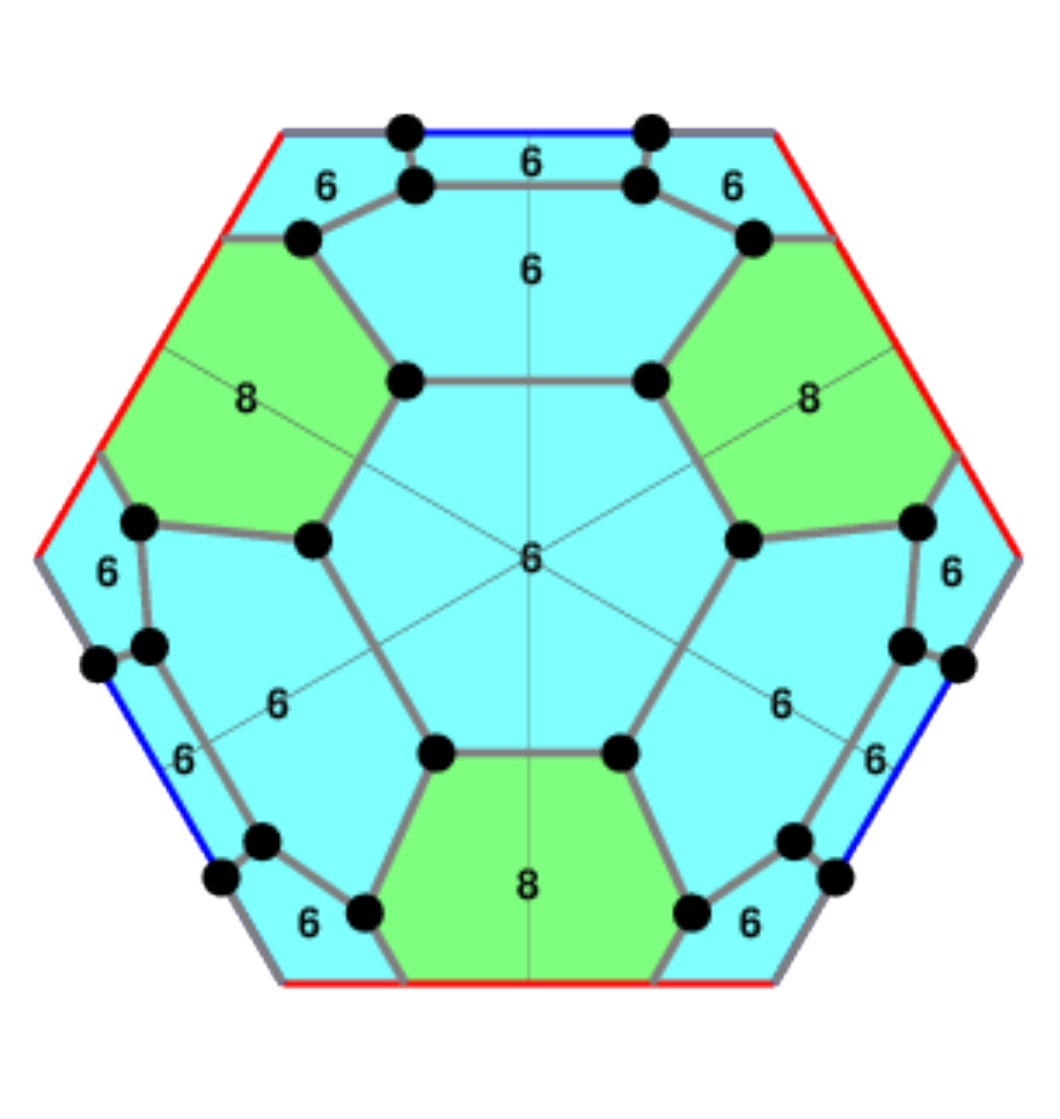} \\
    \includegraphics[bb=0 0 640 640,height=75pt]{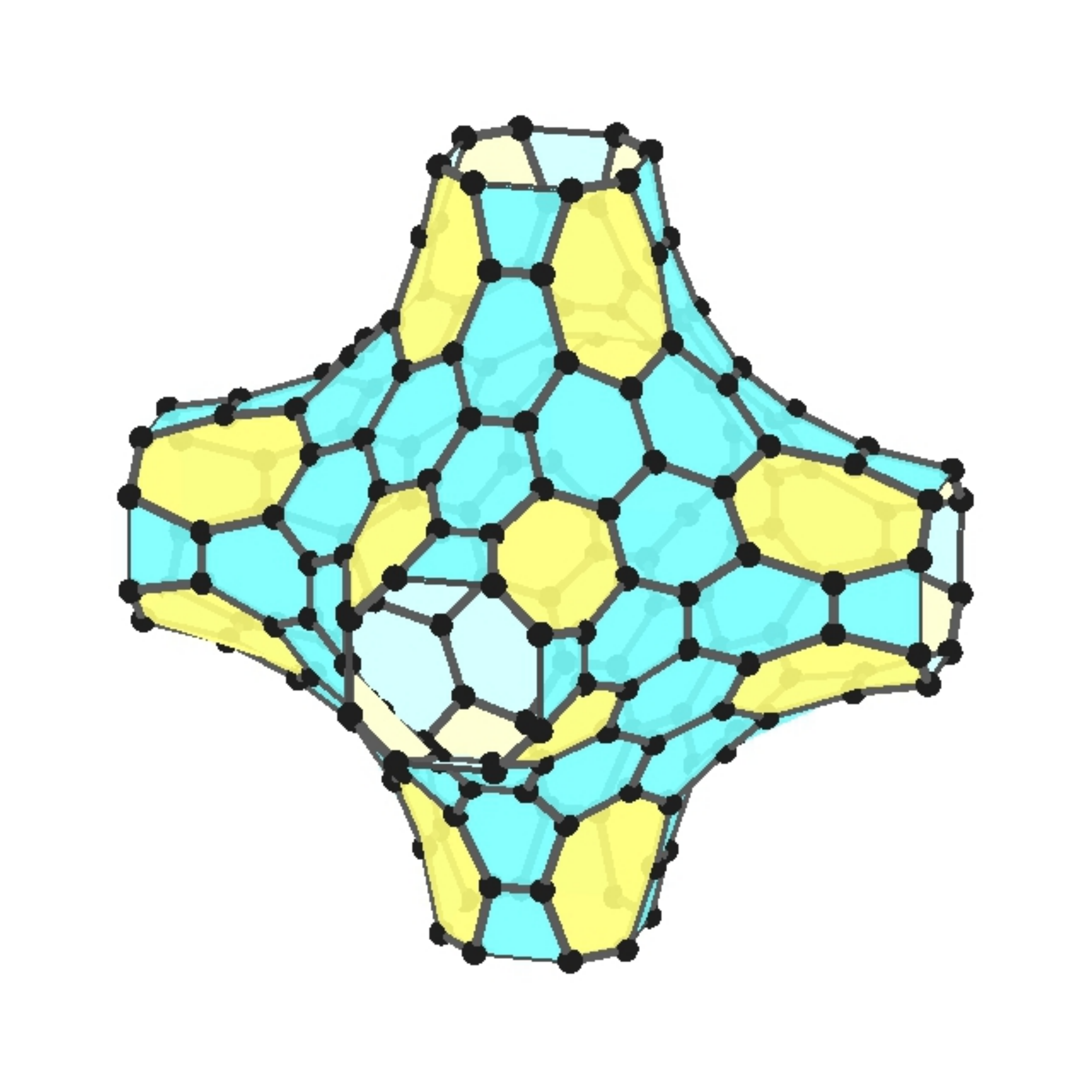} 
    &\includegraphics[bb=0 0 640 640,height=75pt]{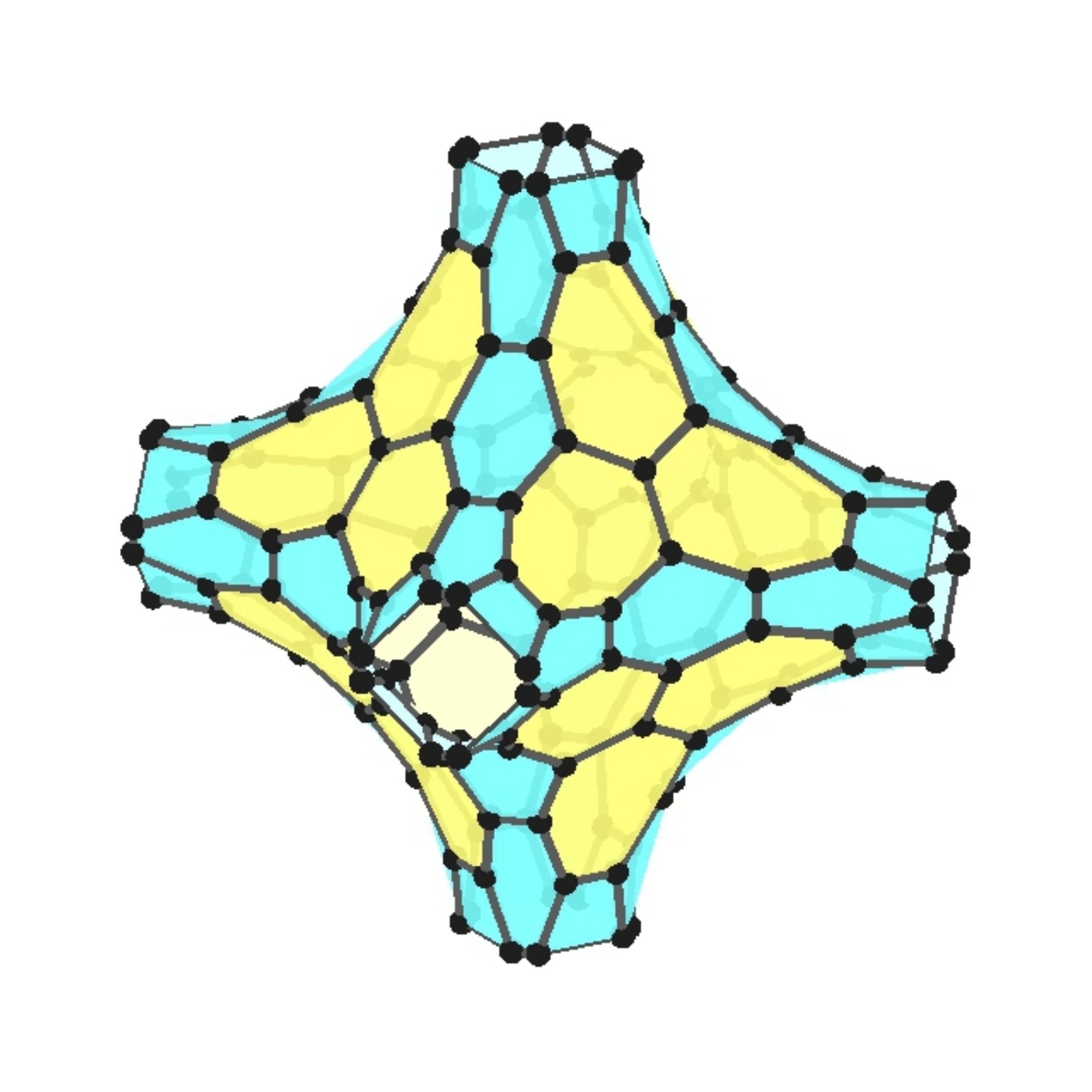} 
    &\includegraphics[bb=0 0 640 640,height=75pt]{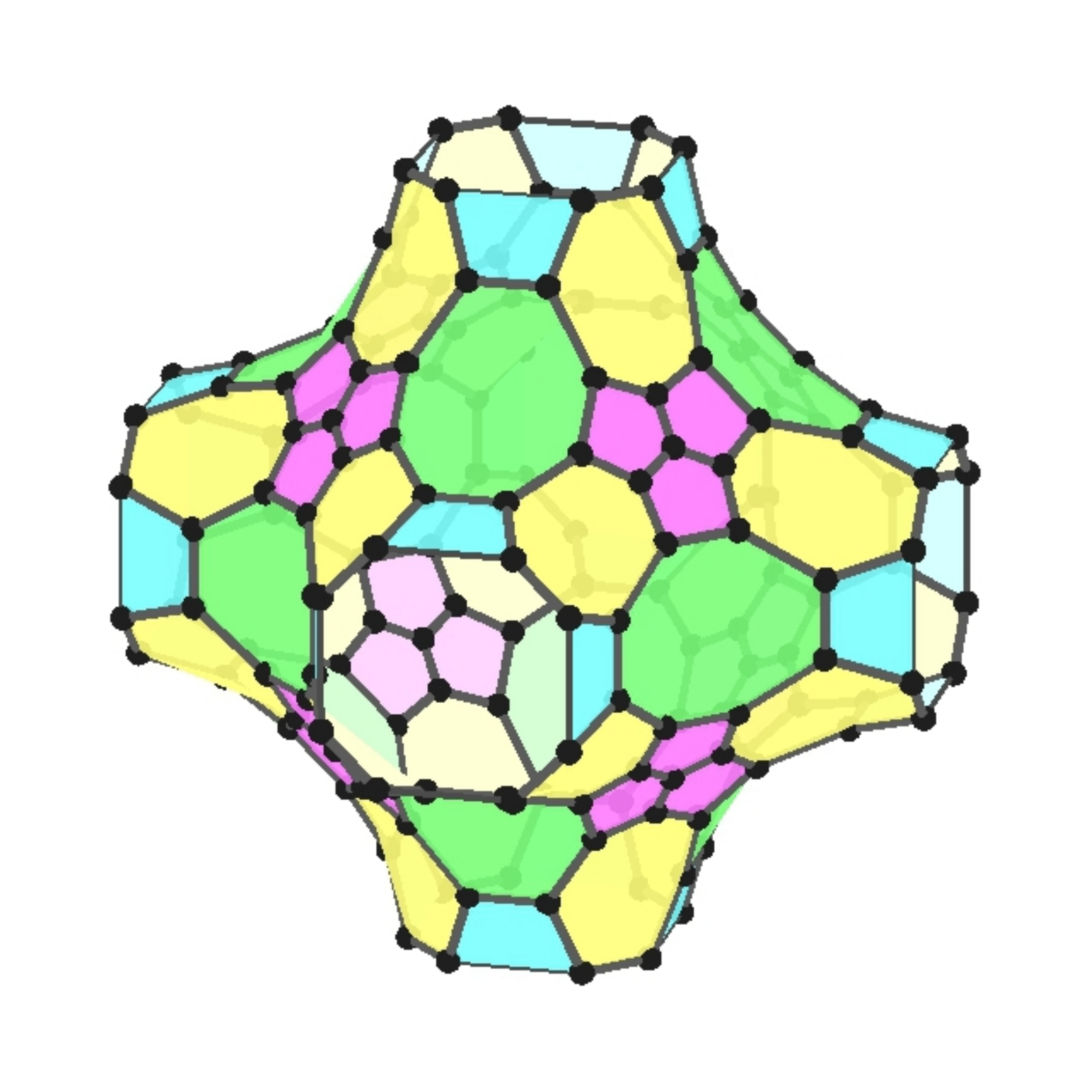} 
    &\includegraphics[bb=0 0 640 640,height=75pt]{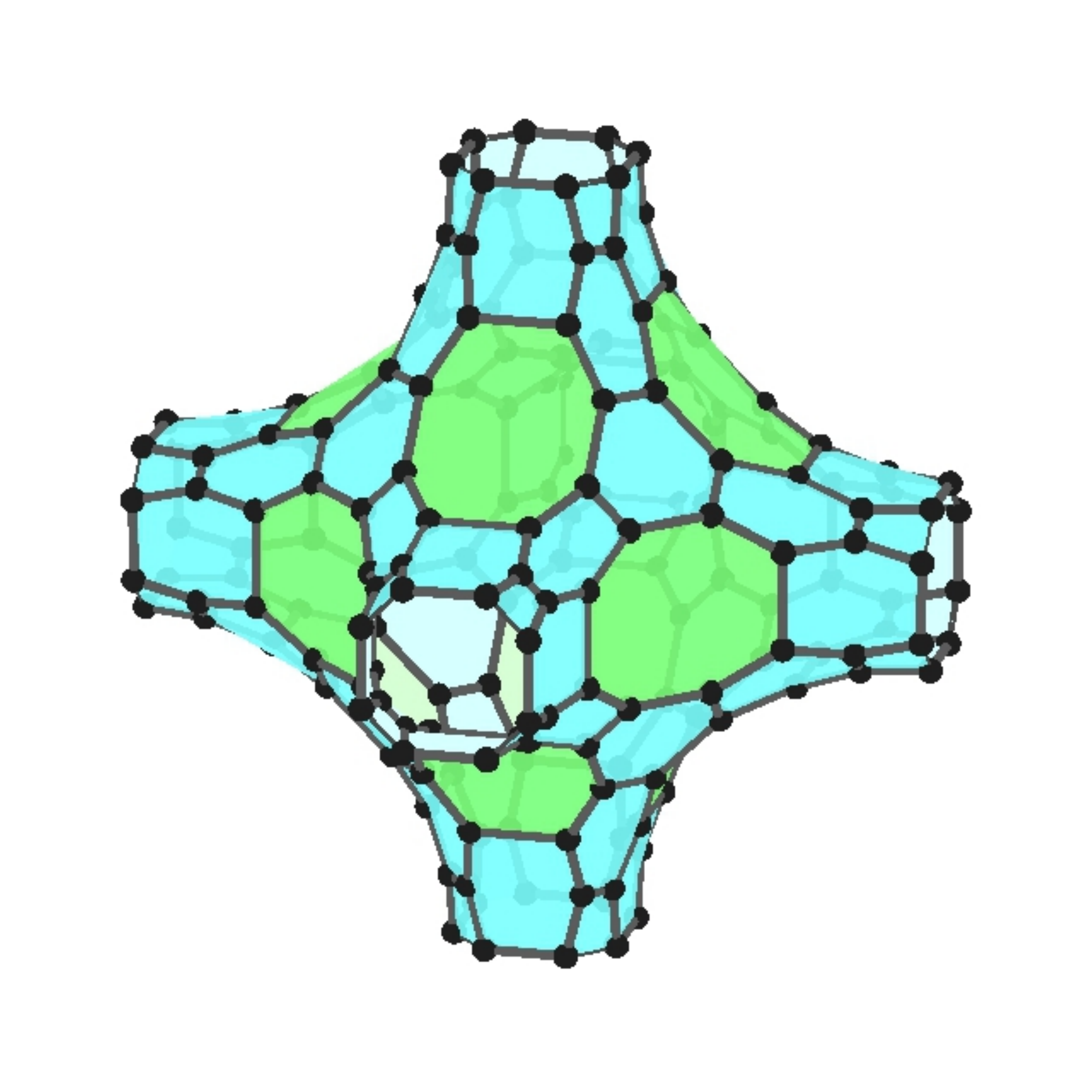} \\
    \includegraphics[bb=0 0 640 640,height=75pt]{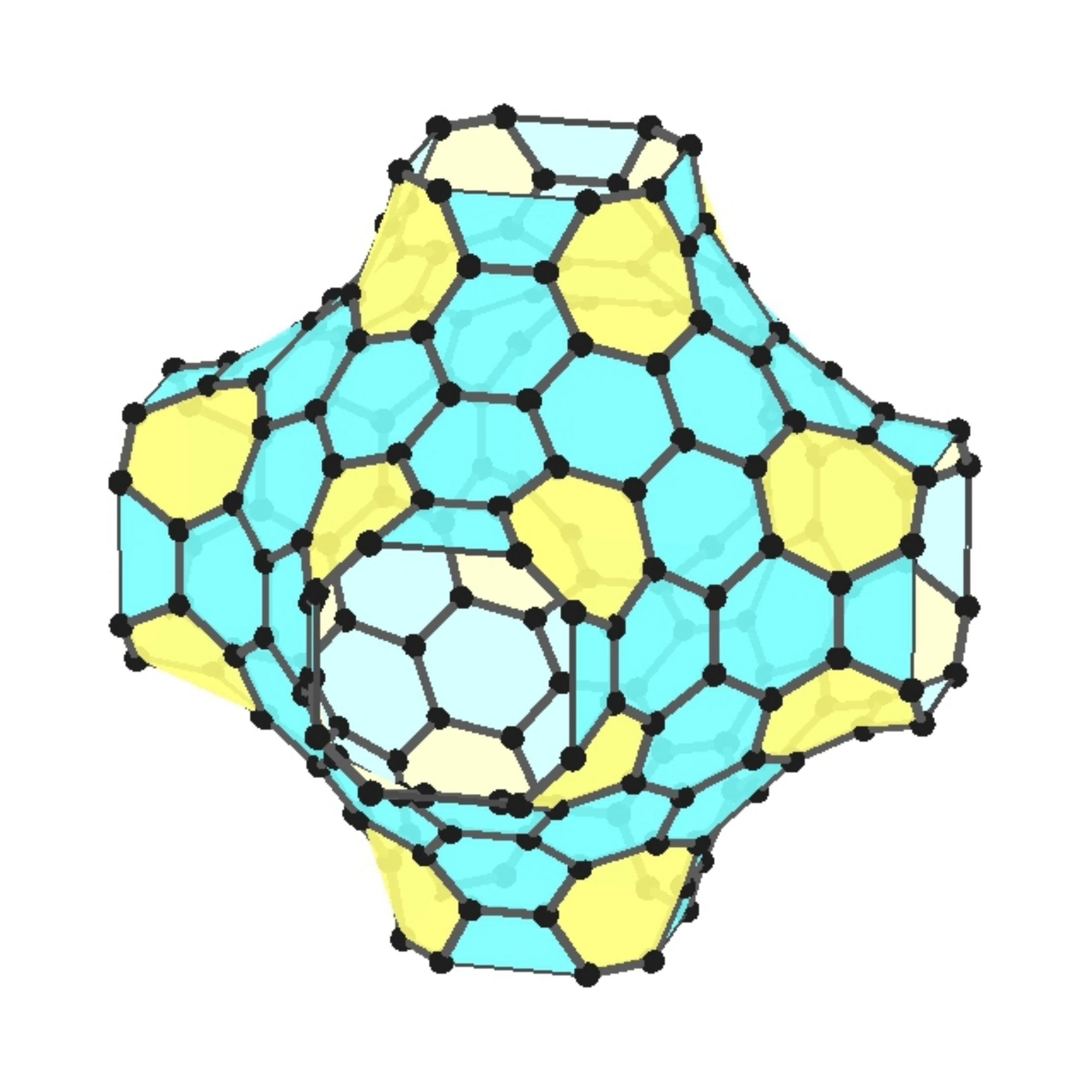} 
    &\includegraphics[bb=0 0 640 640,height=75pt]{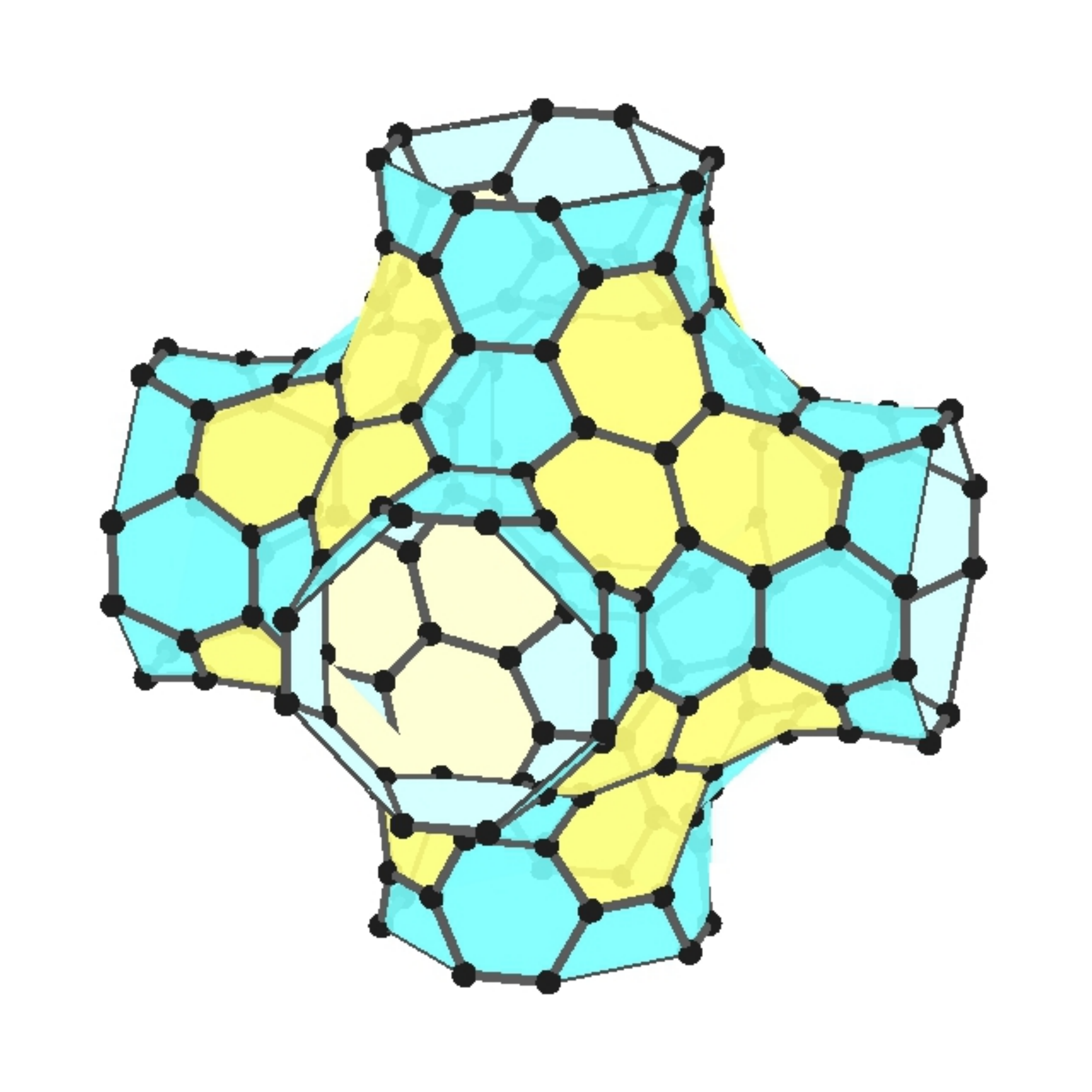} 
    &\includegraphics[bb=0 0 640 640,height=75pt]{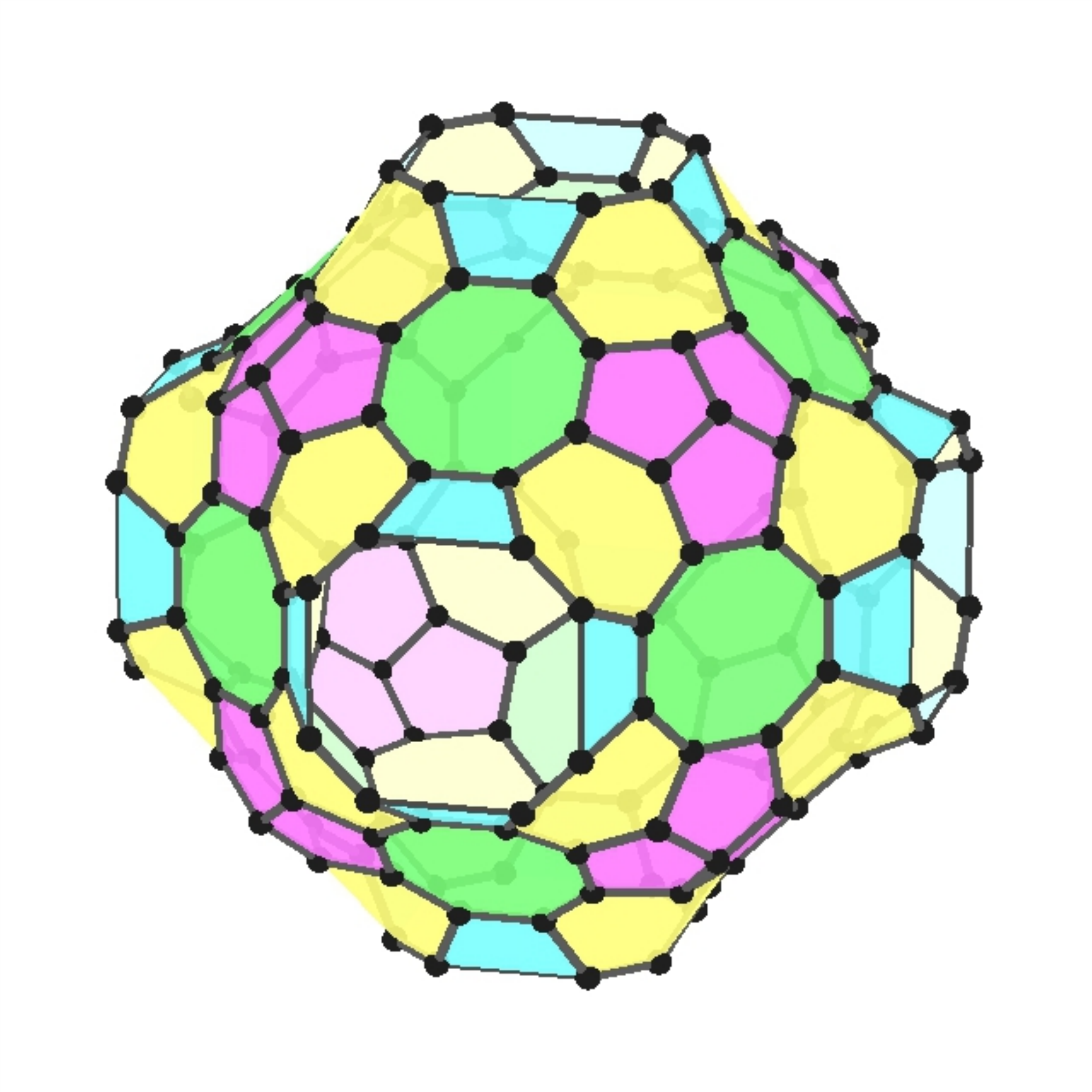} 
    &\includegraphics[bb=0 0 640 640,height=75pt]{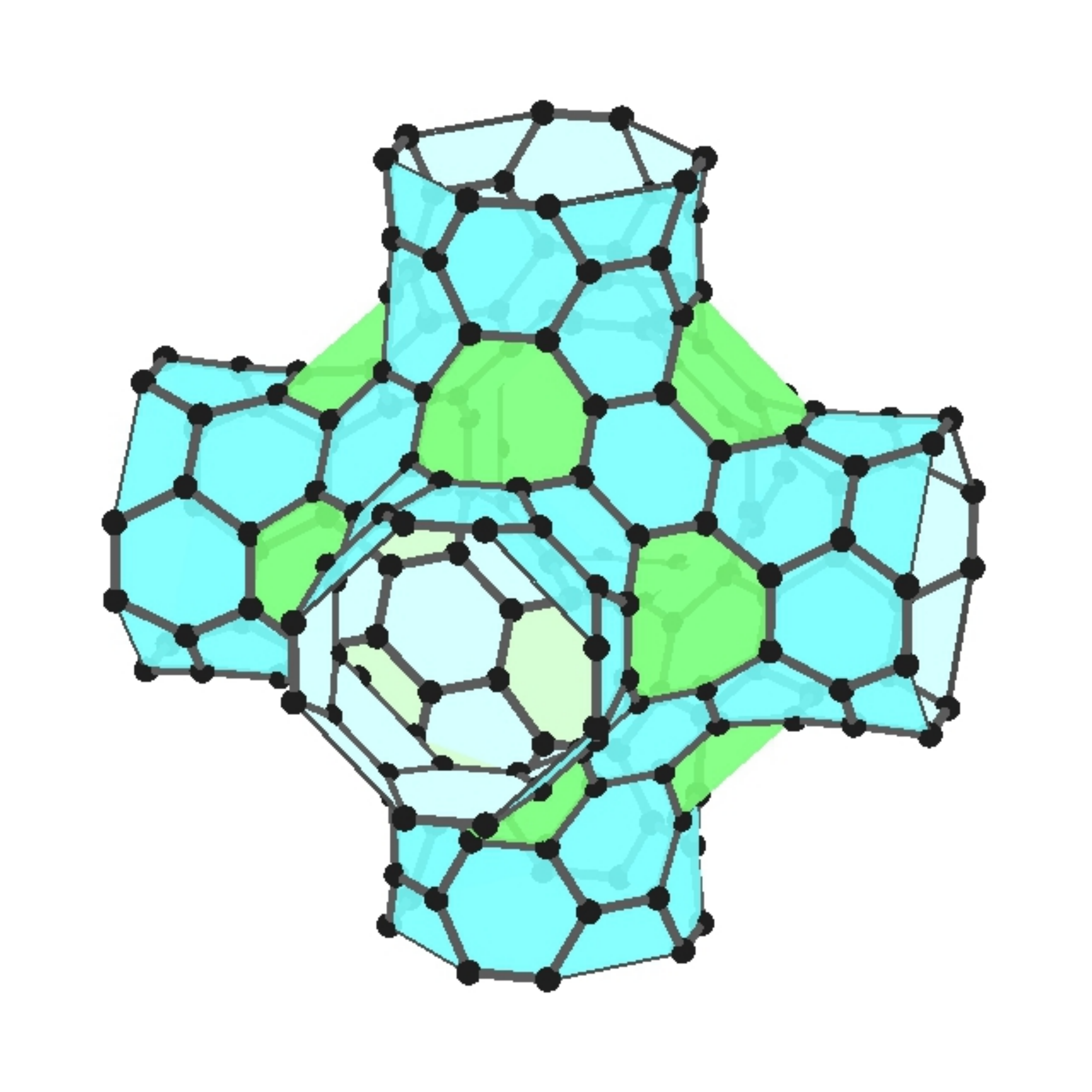} \\
    metal & metal & metal & \scalebox{1.0}[1.0]{semi-conductor}\\
    $\numberOf{V}= 176$
    &$\numberOf{V}= 152$
    &$\numberOf{V}= 152$
    &$\numberOf{V}= 158$\\
    \scalebox{0.8}[1.0]{$N_6 = 60$, $N_7 = 24$}
    &\scalebox{0.8}[1.0]{$N_6 = 48$, $N_7 = 24$}
    &\scalebox{0.8}[1.0]{$N_5 = 24$, $N_6 = 12$}
    &\scalebox{0.8}[1.0]{$N_6 = 68$, $N_8 = 12$}\\
    &&\scalebox{0.8}[1.0]{$N_7 = 24$, $N_8 = 12$}
    &\\
  \end{tabular}
  \caption{
    Mackay-like structures, which are physically stable.
    Physical stabilities are calculated by density functional theory (DFT).
    For the notion of ``metal''/''semi-conductor'', see Section \ref{sec:electronic}.
    (Tagami--Liang--Naito--Kawazoe--Kotani \cite{Tagami:2014}).
  }
  \label{fig:mackay-like}
\end{figure}

\begin{remark}
  No negatively curved $sp^2$-carbon structure has been synthesised so far.
  However, a piece of negatively curved carbon structure is chemically synthesised 
  by 
  Kawasumi--Zhang--Segawa--Scott--Itami \cite{Itami}.
\end{remark}
%%%%% END of document

%%%%%%%%%%%%%%%% 
%%%%% END of document

%%%%% \input{discrete_surface.tex}
% -*- coding: utf-8 -*-
\section{A discrete surface theory}
In the previous section, we explain ``negatively'' curved carbon structure; 
however, the definition of negativity is that the Euler number is negative.
This is a similar situation with that 
``a surface with genus $\ge 2$ cannot be non-negatively curved''.
Since the Gauss curvature of smooth surfaces is a function defined on each point of the surface, 
the negativity in the previous section is not a precise property for discrete surfaces.
\par
On the other hand, there are many definitions of the Gauss curvature for discrete surfaces.
One of the examples is a triangulation of a smooth surface, 
which is a discretization of a smooth surface.
For triangulation of smooth surfaces, 
the Gauss curvature is defined by ``angle defects'', 
that is $K(p) = 2\pi - \sum \theta_i$, where $\theta_i$ is inner angle at $p$ of each triangle gathering at $p$.
\par
Applying the definition of the Gauss curvature by angle defects to the Mackay--Terrones' structure (a standard realization), 
we obtain $K \equiv 0$, 
since a standard realization of a trivalent topological crystal satisfies the ``balance condition'' (\ref{eq:harmonic}),
that is, each point and three neighbouring points are co-planer.
Although each triangle of triangulations of smooth surfaces is planer, however, 
each face of the discrete surface is not planer.
\par
Hence, we should construct a new definitions of the Gauss curvature and the mean curvature for trivalent discrete surfaces.
%%%%% \input{curvature.tex}
% -*- coding: utf-8 -*-
\subsection{Curvatures of trivalent discrete surfaces}
\begin{definition}[Kotani--Naito--Omori \cite{Kotani-Naito-Omori}]
  Let $X = (V, E)$ be a trivalent graph, 
  and $\Phi \colon X \longrightarrow \R^3$ be a realization of $X$.
  The realization $\Phi$ is called a {\em trivalent discrete surface} if and only if 
  for each $x \in V$, 
  at least two vectors of $\{\Phi(e) : e \in E_x\}$ are linearly independent.
\end{definition}

We remark that this definition of trivalent discrete surfaces is not limited to topological crystals, 
and thus we can treat $\mathrm{C}_{60}$ and single-wall nanotubes (SWNTs) for example.
But, the definition also contains $K_4$ structure as trivalent discrete surfaces, 
although, it does not look like a discrete surface.
\par
Before giving a definition of curvatures for trivalent discrete surfaces, 
we recall properties of curvature for smooth surfaces in $\R^2$.

\begin{definition}[{\bfseries Curvatures for smooth surfaces}]
  \label{definition:curvatureforsmooth}
  Let $p \colon \Omega \subset \R^2\longrightarrow \R^3$ be a smooth surface, 
  and $n(x)$ be a unit normal vector at $x \in \Omega$.
  We define the first and the second fundamental forms
  by
  $\FI = \inner{dp}{dp}$, 
  and 
  $\II = -\inner{dn}{dp}$, respectively.
  By using them, 
  the Gauss curvature and the mean curvature are defined by
  $K(x) = \det(\FI^{-1} \II)$
  and
  $H(x) = \frac{1}{2}\trace(\FI^{-1} \II)$, respectively.
\end{definition}

\begin{proposition}
  \label{claim:curvaturesforsmooth}
  For a smooth surface $p \colon \Omega \longrightarrow \R^3$ 
  with a unit normal vector field $n$, 
  the Gauss curvature $K$ satisfies 
  \begin{math}
    \nabla_1 n(x) \times \nabla_2 n(x)
    =
    K(x) (\nabla_1 p(x) \times \nabla_2 p(x))
  \end{math}, 
  and the mean curvature $H$ satisfies 
  \begin{math}
    \left.\frac{d}{dt} A(x,t) \right|_{t=0}
    =
    -H(x) A(x), 
  \end{math}
  where $A(x)$ is the area element of $p$ and
  $A(x, t)$ is the one of $p + t n$.
\end{proposition}

To define curvatures for trivalent discrete surfaces which are an analogue of Definition \ref{definition:curvatureforsmooth}, 
it may be suffice to define the covariant derivative and the unit normal vector at each vertex of trivalent discrete surfaces, 
and we should prove that curvatures for trivalent discrete surfaces satisfy similar properties of Proposition \ref{claim:curvaturesforsmooth}.

\begin{definition}[Kotani--Naito--Omori \cite{Kotani-Naito-Omori}]
  Let $\Phi(X)$ be a trivalent discrete surface, 
  $\v{x} \in \Phi(X)$, and 
  $\v{x}_1$, $\v{x}_2$, $\v{x}_3$ its adjacency vertices.
  We define the unit normal vector at $\v{x}$ as 
  the normal vector of the plane through $\v{e}_1$,  $\v{e}_2$,  $\v{e}_3$ (see Fig.~\ref{fig:normalvector}), 
  that is, 
  \begin{displaymath}
    n(\v{x})
    =
    \frac{(\v{e}_2 - \v{e}_1) \times (\v{e}_3 - \v{e}_1)}{|(\v{e}_2 - \v{e}_1) \times (\v{e}_3 - \v{e}_1)|}
    =
    \frac{\v{e}_1 \times \v{e}_2 + \v{e}_2 \times \v{e}_3 + \v{e}_3 \times \v{e}_1}{|\v{e}_1 \times \v{e}_2 + \v{e}_2 \times \v{e}_3 + \v{e}_3 \times \v{e}_1|}, 
  \end{displaymath}
  and the covariant derivative as
  \begin{displaymath}
    \nabla_\v{e} \v{x}
    =
    \Proj(\v{e})
    =
    \v{e} - \inner{\v{e}}{n(\v{x})}n(\v{x}), 
    \quad
    e \in E_x, 
  \end{displaymath}
  where
  $\v{e}_i = \v{x}_i - \v{x}$.
\end{definition}
\begin{figure}[htp]
  \centering
  \includegraphics[bb=0 0 157 151,scale=1.0]{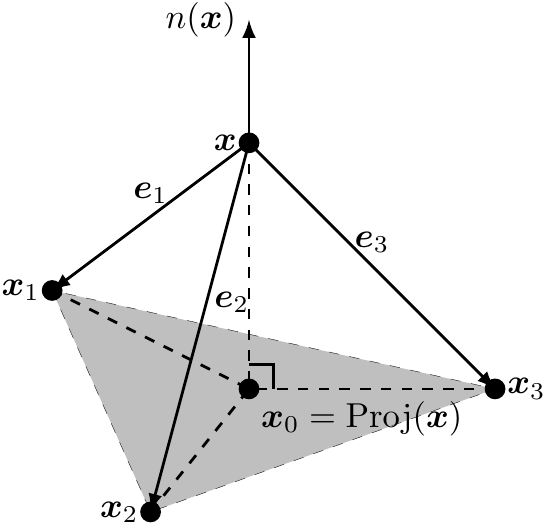}
  \caption{
    The unit normal vector for trivalent discrete surfaces.
    The area of the triangle filled in gray is the local area at $\v{x}$.
  }
  \label{fig:normalvector}
\end{figure}

\begin{definition}[Kotani--Naito--Omori \cite{Kotani-Naito-Omori}]
  \label{definition:curvature}
  Let $X$ be a trivalent discrete surface and $\v{x}$ be a vertex of $X$.
  We define the first and the second fundamental forms by
  \begin{displaymath}
    \begin{aligned}
      \FI(\v{x})
      &=
      \begin{pmatrix}
        \inner{\v{e}_2 - \v{e}_1}{\v{e}_2 - \v{e}_1}
        & \inner{\v{e}_2 - \v{e}_1}{\v{e}_3 - \v{e}_1}
        \\
        \inner{\v{e}_3 - \v{e}_1}{\v{e}_2 - \v{e}_1}
        & \inner{\v{e}_3 - \v{e}_1}{\v{e}_3 - \v{e}_1}
        \\
      \end{pmatrix}, 
      \\
      \II(\v{x})
      &=
      -
      \begin{pmatrix}
        \inner{\v{e}_2 - \v{e}_1}{n_2 - n_1}
        & \inner{\v{e}_2 - \v{e}_1}{n_3 - n_1}
        \\
        \inner{\v{e}_3 - \v{e}_1}{n_2 - n_1}
        & \inner{\v{e}_e - \v{e}_1}{n_3 - n_1}
        \\
      \end{pmatrix}, 
    \end{aligned}
  \end{displaymath}
  and we define the Gauss curvature and the mean curvature at $\v{x}$ by 
  \begin{displaymath}
    \begin{aligned}
      K(\v{x}) &= \det(\FI(\v{x})^{-1} \II(\v{x})), \\
      H(\v{x}) &= \dfrac{1}{2}\trace(\FI(\v{x})^{-1} \II(\v{x})).
    \end{aligned}
  \end{displaymath}
\end{definition}
Then, we obtain the following properties for curvatures of trivalent discrete surfaces.
\begin{theorem}[Kotani--Naito--Omori \cite{Kotani-Naito-Omori}]
  \label{claim:curvature}
  For a trivalent discrete surface $\Phi$, 
  the Gauss curvature $K$ satisfies 
  \begin{math}
    \nabla_i n(\v{x}) \times \nabla_j n(\v{x})
    =
    K(\v{x}) (\v{e}_i \times \v{e}_j)
  \end{math}, 
  where $\nabla_i = \nabla_{e_i}$, 
  and the mean curvature $H$ satisfies 
  \begin{math}
    \left.\frac{d}{dt}\right|_{t=0}{\mathcal A}(\Phi + t n)
    =
    -2 \sum_{x \in V} H(\v{x}) A(\v{x}), 
  \end{math}
  where 
  $A(\v{x}) = \v{e}_1 \times \v{e}_2 + \v{e}_2 \times \v{e}_3 + \v{e}_3 \times \v{e}_1$
  is the local area at $\v{x}$
  and 
  ${\mathcal A}(\Phi) = \sum_{x \in V} A(\v{x})$
  is the total area.
\end{theorem}

\begin{remark}
  Unfortunately, 
  the second fundamental form $\II$ of trivalent discrete surfaces
  is {\em not} symmetric in general.
\end{remark}

\begin{theorem}[Kotani--Naito--Omori \cite{Kotani-Naito-Omori}]
  Trivalent discrete surface $\Phi$ is called {\bfseries minimal}
  if and only if 
  $H(\v{x}) = 0$, which is expressed 
  by the system of linear equation
  \begin{displaymath}
    \nabla_{\v{e}_2 - \v{e}_3} n \times \nabla_{\v{e}_1} \Phi
    +
    \nabla_{\v{e}_3 - \v{e}_1} n \times \nabla_{\v{e}_2} \Phi
    +
    \nabla_{\v{e}_1 - \v{e}_2} n \times \nabla_{\v{e}_3} \Phi
    =
    0.
  \end{displaymath}
\end{theorem}

\begin{figure}[htp]
  \centering
  \begin{tabular}{ccc}
    \multicolumn{1}{l}{(a)}
    &\mbox{}
    &\multicolumn{1}{l}{(b)}\\
    \includegraphics[bb=0 0 350 350,height=120pt]{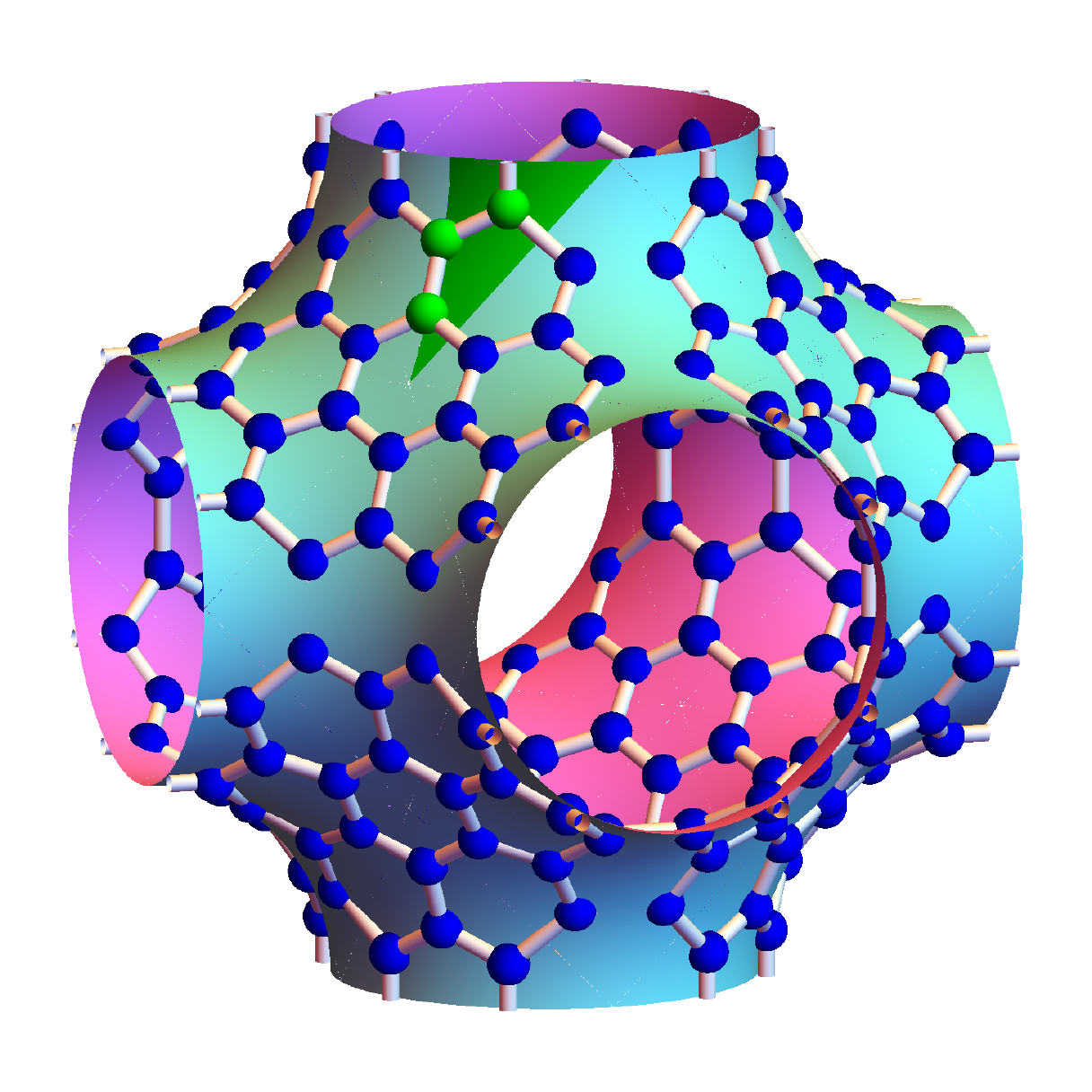}
    &\mbox{}
    &\includegraphics[bb=0 0 350 350,height=120pt]{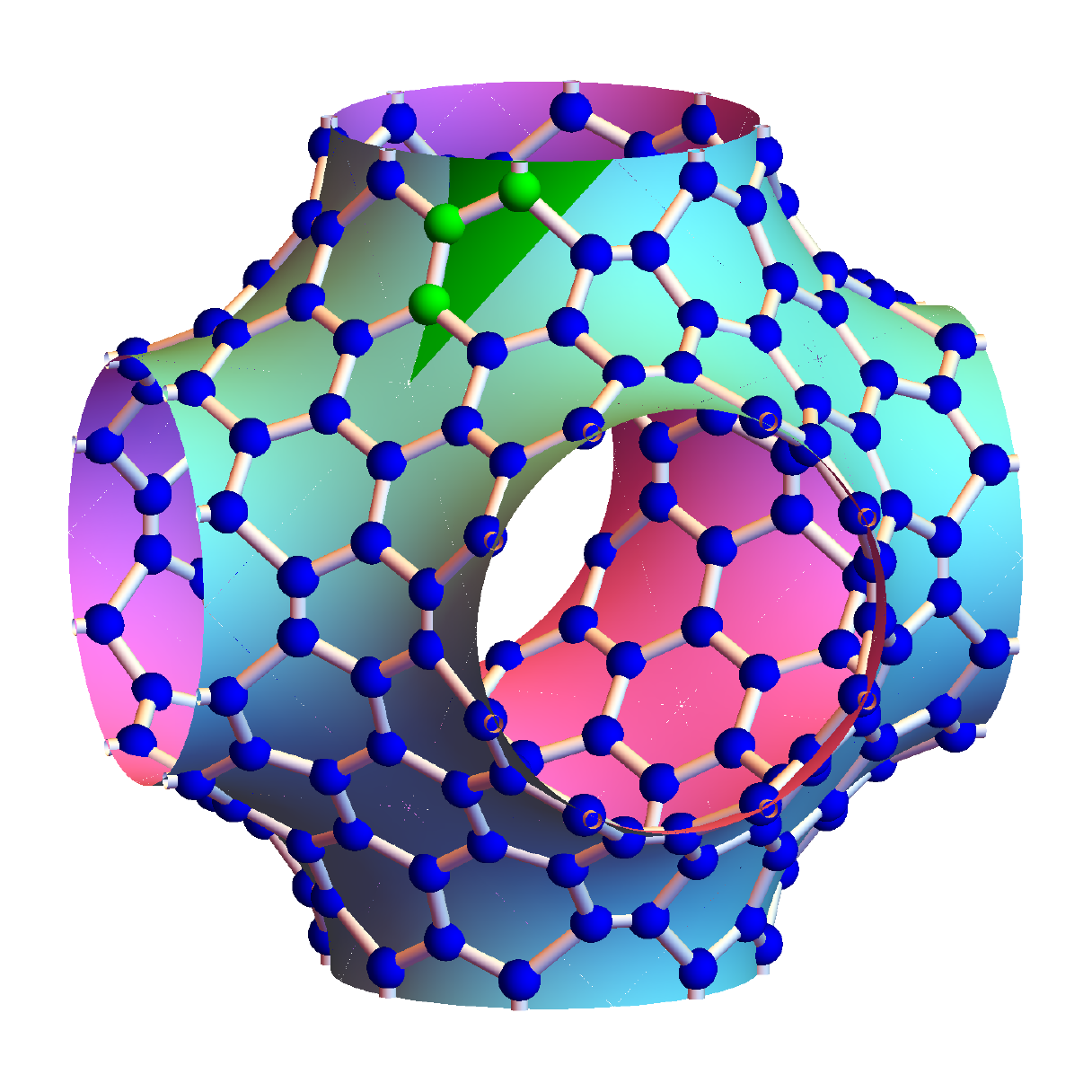}\\
  \end{tabular}
  \caption{
    (a) Mackay--Terrones' structure (standard realization), 
    (b) minimal realization of the same structure
    (Kotani--Naito--Omori \cite{Kotani-Naito-Omori}).
  }
  \label{fig:mackay-minimal}
\end{figure}

\begin{proposition}[Kotani--Naito--Omori \cite{Kotani-Naito-Omori}]
  \label{claim:sphereshaped}
  If a trivalent discrete surface is a plane graph, 
  then $K \equiv 0$, $H \equiv 0$.
  Here, a graph is called plane, if it is drawn in a plane without self-intersection.
  If a trivalent discrete surface satisfies  $\v{x} = r n(\v{x})$ for any $x \in V$, 
  then $K \equiv 1/r^2$, $H \equiv -1/r$.
\end{proposition}
We call a surface satisfying the $\v{x} = r n(\v{x})$ {\em sphere shaped}.
Regular polyhedra and semi-regular polyhedra (including $\mathrm{C}_{60}$) are sphere shaped.

\begin{proposition}[Kotani--Naito--Omori \cite{Kotani-Naito-Omori}]
  Let $\CNT(\lambda, \v{c})$ be a SWNT with the chiral index $\v{c} = (c_1, c_2)$ and the scale factor $\lambda$, 
  that is, 
  \begin{math}
    (x, y) \mapsto
    (r(\lambda, \v{c})\cos(x/r(\lambda, \v{c})), 
    r(\lambda, \v{c})\sin(x/r(\lambda, \v{c})), y).
  \end{math}
  Then the Gauss curvature and the mean curvature of $\CNT(\lambda ,\v{c})$ are 
  \begin{displaymath}
    \begin{aligned}
      K(\lambda,\v{c})
      &= 
      \frac{
        4m_z(\v{c})^2(m_x(\v{c})^2+m_y(\v{c})^2)
      }{
        3r(\lambda,\v{c})^2
        (m_x(\v{c})^2+m_y^2(\v{c})+(4/3)m_z(\v{c})^2)^2
      }, 
      \\
      H(\lambda,\v{c}) 
      &= 
      -\frac{m_x(\v{c})}{2r(\lambda,\v{c})}
      \cdot
      \frac{
        m_x(\v{c})^2+m_y(\v{c})^2+(8/3)m_z(\v{c})^2
      }{
        (m_x(\v{c})^2+m_y(\v{c})^2+(4/3)m_z(\v{c})^2)^{3/2}
      }, 
    \end{aligned}
  \end{displaymath}
  where $m_\alpha(\v{c})$ is a quantity defined from the chiral index.
  In particular, $c_1 = c_2$, then $m_z(\v{c})
  = 0$, and
  \begin{displaymath}
    K(\lambda,\v{c}) 
    = 
    0, 
    \quad
    H(\lambda,\v{c}) 
    = 
    -\frac{1}{2r(\lambda,\v{c})}
    \cos\frac{C_1}{2}.
  \end{displaymath}
\end{proposition}

\begin{figure}[htp]
  \centering
  \begin{tabular}{ccc|ccc}
    \multicolumn{2}{l}{(a)}
    &\mbox{}
    &\multicolumn{2}{l}{(b)}\\
    \includegraphics[bb=0 0 246 422,scale=0.25]{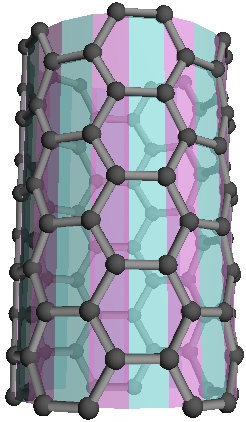}
    &\includegraphics[bb=0 0 151 221,scale=0.3]{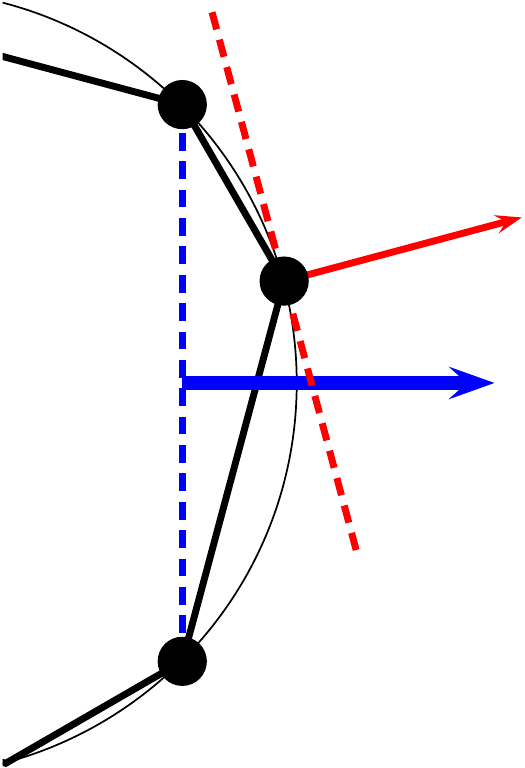}
    &\mbox{}
    &\includegraphics[bb=0 0 246 422,scale=0.25]{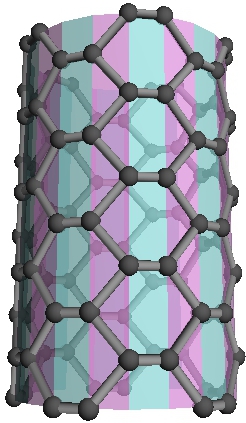}
    &\includegraphics[bb=0 0 157 221,scale=0.3]{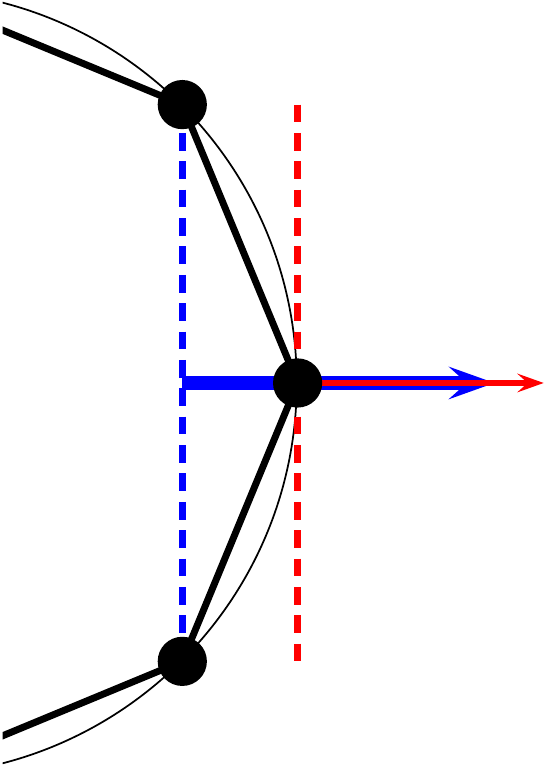}
  \end{tabular}
  \caption{
    (a) $\CNT(\lambda, \v{c})$ satisfies $H \not\equiv 1/r$, 
    whose normal vectors are not parallel to normal vectors of the underlying cylinder.
    (b) small change of $\CNT(\lambda, \v{c})$ satisfies $H \equiv 1/r$, 
    whose normal vectors are parallel to normal vectors of the underlying cylinder
    (Kotani--Naito--Omori \cite{Kotani-Naito-Omori}).
  }
  \label{fig:normalcnt}
\end{figure}

\begin{figure}[htp]
  \centering
  \begin{tabular}{ccc}
    \multicolumn{1}{l}{(a)}
    &\mbox{}
    &\multicolumn{1}{l}{(b)}\\
    \includegraphics[bb=0 0 464 464,height=120pt]{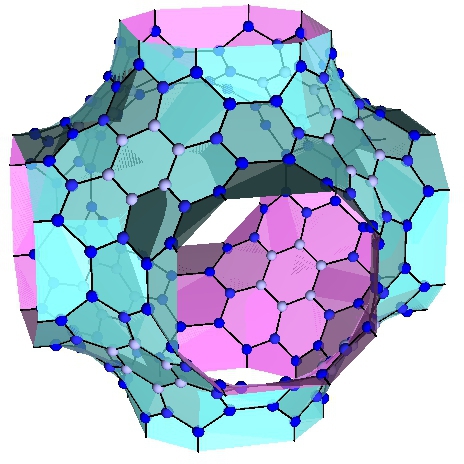}
    &\mbox{}
    &\includegraphics[bb=0 0 464 464,height=120pt]{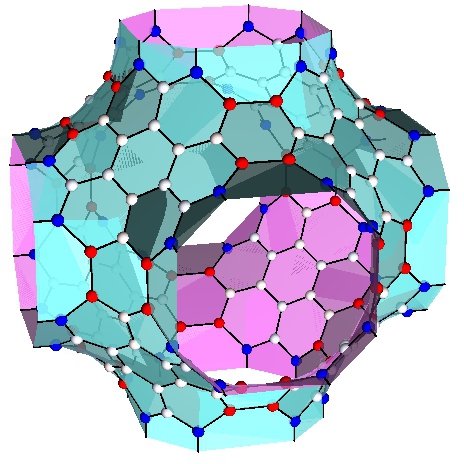}\\
  \end{tabular}
  \caption{
    (a) The Gauss curvatures of Mackay--Terrones' structure, 
    (b) the mean curvatures of Mackay--Terrones' structure.
    By our definition of curvatures, 
    Mackay--Terrones' structure is {\em pointwise} negatively curved.
    Blue (red) vertices are negatively (positively) curved, 
    and color densities expresses relative absolute values of curvatures
    (Kotani--Naito--Omori \cite{Kotani-Naito-Omori})
  }
  \label{fig:mackaycurvature}
\end{figure}

\begin{remark}
  Curvatures of $K_4$ structure satisfy $K > 0$ and $H \equiv 0$.
  In the case of smooth surfaces, the mean curvature $H \equiv 0$ then the Gauss curvature $K \le 0$, 
  since the second fundamental form $\II$ is symmetric.
  On the other hand, 
  in our discrete case, $\II$ is not symmetric, 
  and thus $H \equiv 0$ does not imply $K \le 0$, 
\end{remark}
%%%%% END of document

%%%%% \input{gc.tex}
\subsection{Further problems}
In \cite{Kotani-Naito-Omori}, we discuss a convergence of sequence of ``subdivision'' of a trivalent discrete surface.
In our context, there are no underlying continuous object, 
and trivalent discrete surfaces are essentially discrete object.
For example, Mackay--Terrones' structure is constructed by Schwarz P surface as model, 
but there are no relations between the structure and the surface itself.
We would find a ``limit'' of a sequence of trivalent discrete surfaces, 
and make a relationship a trivalent discrete surface and a continuous surface.
\par
To discuss a convergence theory of trivalent discrete surfaces, 
we should consider how to subdivide a trivalent graph, and how to realize the subdivided graph.
Goldberg--Coxeter subdivision for trivalent graphs defined by Dutour--Deza \cite{MR2429120, MR2035314}
is the good definition to subdivide a trivalent graph (see also Goldberg \cite{1937104} and Omori--Naito--Tate \cite{Omori-Naito-Tate}).
Kotani--Naito--Omori \cite{Kotani-Naito-Omori}, Tao \cite{Tao},  and Kotani--Naito--Tao \cite{Kotani-Naito-Tao} 
discuss convergences of sequences of trivalent discrete surfaces.
\par
On the other hand, 
eigenvalues of the Laplacian of graphs are one of the main objects in discrete geometric analysis (see for example \cite{CRS,  Higuchi-Shirai, MR2039958}).
Eigenvalues of the Laplacian of graphs are also interest from physical and chemical view points (see also Section \ref{sec:electronic}).
Some properties of eigenvalues of Laplacians of 
Goldberg--Coxeter constructions of trivalent graphs are discussed in Omori--Naito--Tate \cite{Omori-Naito-Tate}.

%%% Local Variables:
%%% mode: japanese-latex
%%% TeX-master: t
%%% End:

%%%%% END of document

\appendix
%%%%% \input{appendix.tex}
% -*- coding: utf-8 -*-
\section{Appendix}
%%%%% \input{space_group.tex}
% -*- coding: utf-8 -*-
\subsection{Space groups in $\R^2$ and $\R^3$}
\label{sec:spacegroup}
\begin{definition}
  A discrete finite subgroup $P$ of $O(d)$ is called a {\em point group} of $\R^d$, 
  and 
  a subgroup $\Gamma$ of Euclidean motion group of $\R^d$ is called a {\em space group} of $\R^d$, 
  if $\Gamma$ is discrete subgroup and $\Gamma \cap T \cong \Z^d$, 
  where $T$ is the group of parallel transformations in $\R^d$.
\end{definition}

A $2$-dimensional point group $P$ is extended to a space group, 
if and only if 
$P$ is $3$-, $4$-, or $6$-fold symmetry, 
that is to say, 
$P$ is one of $C_1$, $C_2$, $C_3$, $C_4$, $C_6$, $D_1$, $D_2$, $D_3$, $D_4$, and $D_6$.
Here, 
$C_n$ is the $n$-fold rotation group ($\cong \Z_n$ cyclic group), and
$D_n$ is the $n$-fold rotation and reflection group (dihedral group).
\par
It is a very famous old result that
the number of $2$-dimensional space groups is 17.
$2$-dimensional space groups may contain 
parallel transformations, $2$-, $3$-, or $6$-folds rotations with respect to a point, 
and glide reflections with respect to a line.
Here, a glide reflection is the composition of a reflection and a parallel transformation.

\begin{figure}[htbp]
  \centering
  \begin{tabular}{cc}
    \multicolumn{1}{l}{(a)}
    &\multicolumn{1}{l}{(b)}\\
    \includegraphics[bb=0 0 100 100,height=100pt]{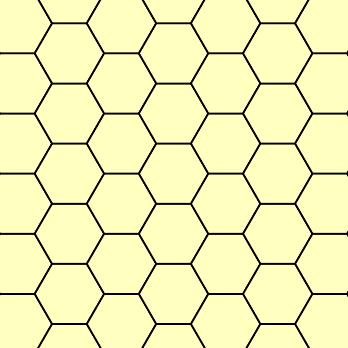}
    &\includegraphics[bb=0 0 100 100,height=100pt]{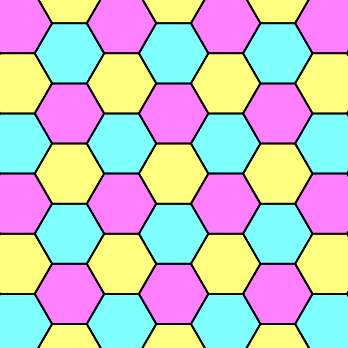}
  \end{tabular}
  \caption{
    Space groups of (a) and (b) are $P6m$ and $P3m1$.
    Since $P3m1$ is a subgroup of $P6m$, 
    we can say that the figure (a) has higher symmetry than figure (b).
  }
  \label{fig:symmetryofhex}
\end{figure}

$3$-dimensional point groups are one of 
$C_n$, $C_{nv}$, $C_{nh}$, $D_n$, $D_{nv}$, $D_{nh}$, $S_{2n}$, 
$T$, $T_d$, $T_h$, $O$, $O_h$, $I$, and $I_h$.
Here, 
$C_{nv}$ ($D_{nv}$) 
and 
$C_{nh}$ ($D_{nh}$)
are $C_n$($D_n$) with additional mirror plane perpendicular (parallel) to the axis to rotation, respectively, 
and $S_{2n}$ is $2n$-fold rotation and reflection axis.
Group $T$, $O$, and $I$ are well-known as polyhedral groups, 
$T_d$ is $T$ with improper rotations, 
$T_h$, $O_h$, and $I_h$ are $T$, $O$, and $I$ with reflections, respectively.
\par
A $3$-dimensional point group $P$ extends a space group, 
if and only if 
$P$ is $3$-, $4$-, or $6$-fold symmetry, 
That is to say
$P$ is one of $C_n$, $C_{nv}$, $C_{nh}$, $D_n$, $D_{nh}$, and $D_{nv}$, ($n = 1,\,2,\,3,\,4,\,6$).
\par
It is also a famous old result that 
number of $3$-dimensional space groups is 230.
$3$-dimensional space groups may contain 
parallel transformations, $2$-, $3$-, or $6$-folds rotations with respect to a point, 
reflections/glide reflections with respect to a line, 
and improper rotation.
Here a improper rotation is 
product with rotation and reflection with respect to a line perpendicular to rotation axis.
\par
The space group of diamond crystals is $F\!d\overline{3}m$, 
where
$F$, $d$, $\overline{3}$, and $m$ mean
that ``Face-centered cubic'' structure, 
the group contains a glide reflection, 
the group contains 3 improper rotations, 
and the group contains a reflection, respectively.
%%%%% END of document

%%%%% \input{carbon_electronic.tex}
% -*- coding: utf-8 -*-
\subsection{Electronic properties of carbon structures}
\label{sec:electronic}
States of electrons of atoms, molecules, and solids follow the Schr\"odinger equation
\begin{equation}
  \label{eq:schrodinger}
  -\Delta \psi + V \psi = E \psi, \quad \text{ in } \R^3, 
\end{equation}
where $V$ is a potential and $E$ is the energy of an electron.
In cases of crystal structures, the state $\psi$ of electrons and the potential $V$ are periodic, 
and hence taking the Fourier transformation of (\ref{eq:schrodinger}), we obtain 
\begin{equation}
  \label{eq:schrodingerF}
  \widehat{H} \widehat{\psi}(\xi)  = E(\xi) \widehat{\psi}(\xi), 
\end{equation}
where $H = -\Delta + V$.
The dispersive relation $\xi \mapsto E(\xi)$ represents energies of electrons with the wave number $\xi$ 
in a crystal.
\par
As an example, we consider electronic states of graphenes.
Considering $\pi$-electrons in graphenes, we obtain eigenvalues of $\widehat{H}$, 
\begin{equation}
  \label{eq:grapheneBand}
  E(\xi) = \pm \sqrt{3 + 2 \cos(\xi_1) + 2 \cos(\xi_2) + 2 \cos(\xi_1 - \xi_2)}.
\end{equation}
As shown in Fig.~\ref{fig:grapheneBand}, 
the lower band (valence band) and the upper band (conductor band) attache 
at $K$ and $K'$ points with the Fermi energy (zero energy), 
and hence graphenes are conductor (metal).
Moreover, at $K$ and $K'$ points, both bands have cone singularities.
Such points are called {\em Dirac points}.
On Dirac points, effective masses of electrons is zero, 
which are very important properties in solid state physics.
Note that crystals whose conductor bands and valence bands intersects are called metals or conductors, 
and crystals whose conductor/valence bands does not intersects but its gap less than about 1\,eV 
are called semi-conductors.
\par
This calculation, which is called the {\em tight-binding approximation}, 
includes only interaction arise from nearest atoms with respect to each atom.
Hence, the tight-binding approximations is based on graph theories, actions of an abelian group, and the Fourier transformations, 
mathematically.
\begin{figure}[htp]
  \centering
  \begin{tabular}{ccccc}
    \multicolumn{1}{l}{(a)}
    &\mbox{}
    &\multicolumn{1}{l}{(b)}
    &\mbox{}
    &\multicolumn{1}{l}{(c)}\\
    \includegraphics[bb=0 0 360 293,scale=0.35]{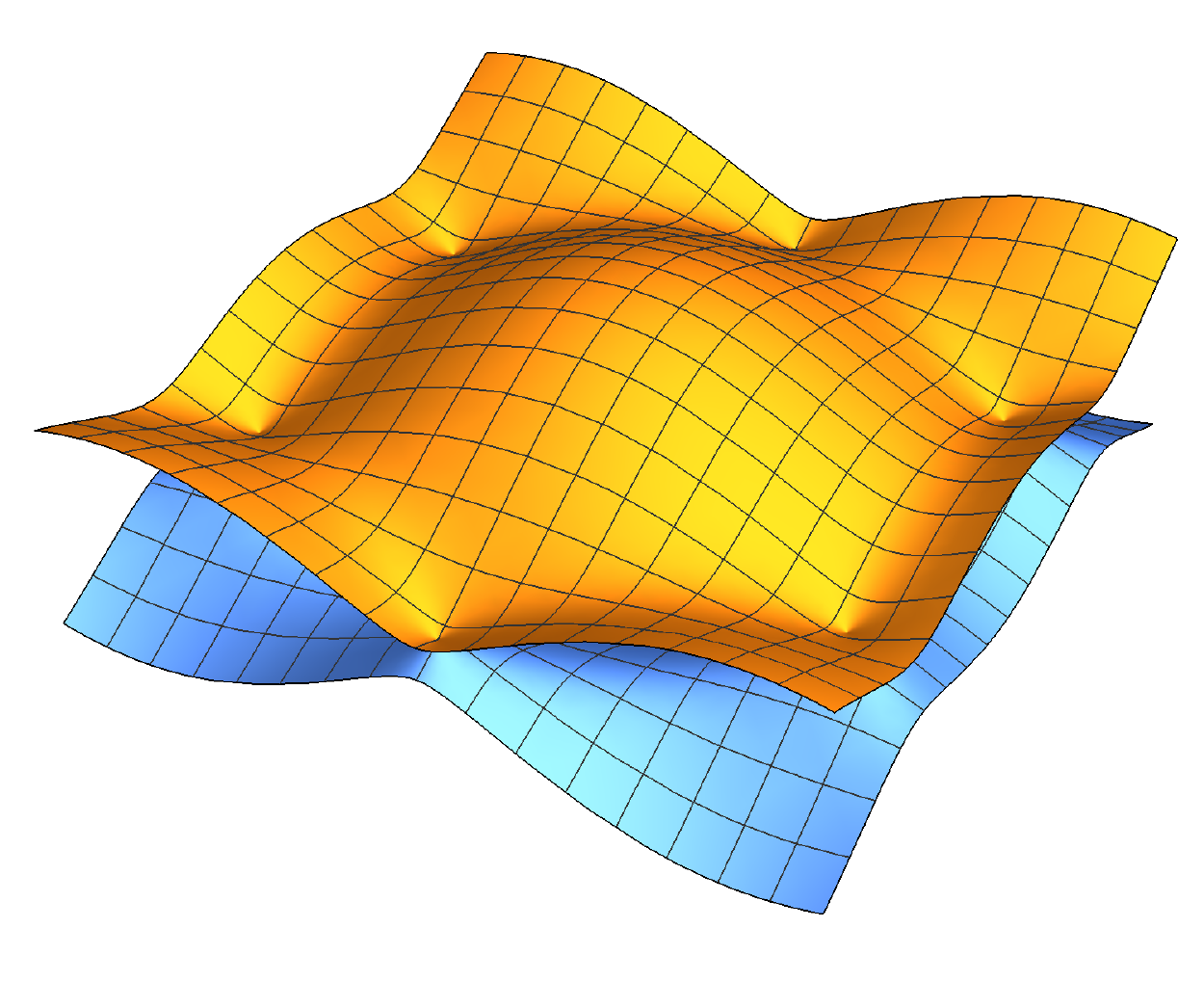} 
    &\mbox{}
    &\includegraphics[bb=0 0 360 293,scale=0.35]{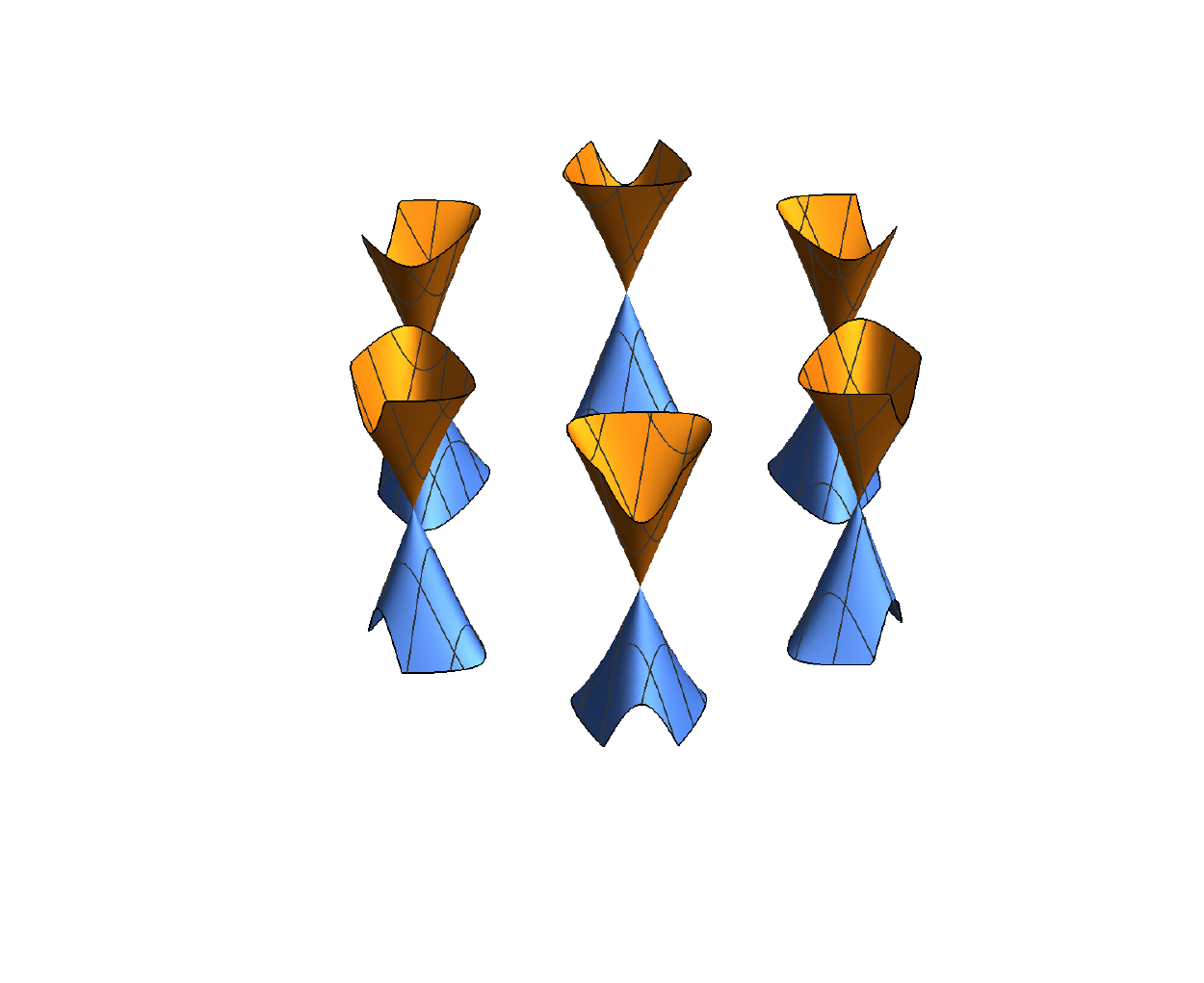} 
    &\mbox{}
    &\includegraphics[bb=0 0 72 67,height=90pt]{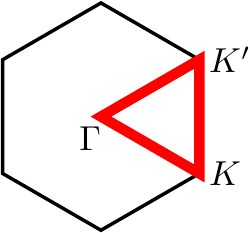} \\
  \end{tabular}
  \caption{
    (a) $E(\xi)$ of graphenes (band structure), 
    (b) closed up $E(\xi)$ near $K$ and $K'$ points.
    (c) highly-symmetric points in the Fourier space.
  }
  \label{fig:grapheneBand}
\end{figure}
\begin{figure}[htp]
  \centering
  \begin{tabular}{ccc}
    \multicolumn{1}{l}{(a)}
    &\mbox{}
    &\multicolumn{1}{l}{(b)}\\
    \includegraphics[bb=0 0 328 192,height=100pt]{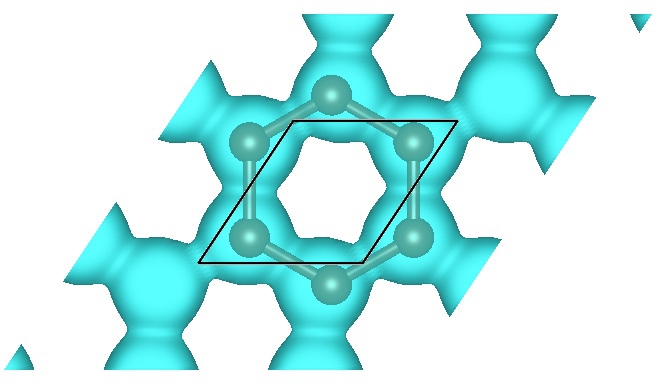}
    &\mbox{}
    &\includegraphics[bb=0 0 328 192,height=100pt]{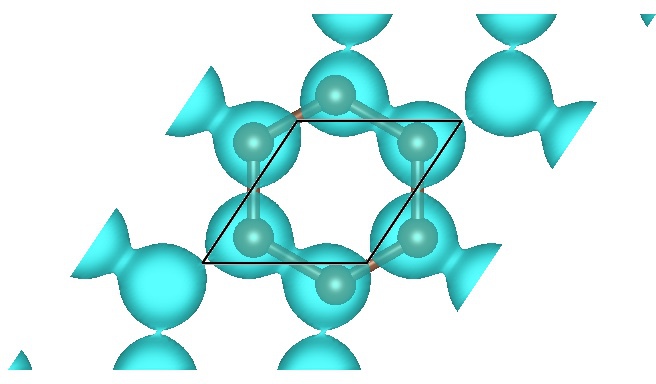}\\
  \end{tabular}
  \caption{
    Iso-surfaces of density of electrons in graphenes by using density functional theory (DFT).
    (a) probability $0.18$, 
    (b) probability $0.30$.
  }
  \label{fig:grapheneIso}
\end{figure}
\par
To calculate of electronic states of $\mathrm{C}_{60}$, we cannot apply the tight-binding approximation, 
since $\mathrm{C}_{60}$ is not a crystal structure, but a molecule structure.
The {\em H\"uckel method} admits calculations of orbitals in hydrocarbon molecules, and is also based on graph theories.
Let $X = (V, E)$ be the graph of a molecule structure, 
that is, $V$ are set of carbon atoms, and $E$ are set of covalent bonds between carbon atoms.
Eigenvectors $\psi$ of the adjacency matrix $A$ of $X$ are molecular orbitals of electrons, 
and their eigenvalues $\lambda$ are energies of orbitals.
\begin{example}
  Consider benzene molecules ($\mathrm{C}_6\mathrm{H}_6$), 
  which contains six carbon atoms, and whose graph is $C_6$ (the cyclic graph with six vertices), 
  and number of electrons of $\pi$-orbitals (not using covalent bonds) is also six.
  Eigenvalues of adjacency matrix of $C_6$ are $\{2, 1, 1, -1, -1, -2\}$, 
  and electrons in $\pi$-orbital occupy orbitals with energies $\{-1, -1, -2\}$ in the ground state, 
  since each orbital can contain two electrons by Pauli's principle.
  We can write the wave function of the ground state as 
  $\psi = \sum_{i=1}^3 c_{ij} \chi_j$ 
  by using eigenvectors $\{\chi_j\}_{j=1}^6$, 
  and we obtain $\sum_{i=1}^3 c_{ij}^2 = 1/2$.
  This means that $\pi$-electrons have equal distributions on each carbon atom.
\end{example}
By using similar calculations, we also obtain $\pi$-electrons have equal distributions on each carbon atom on $\mathrm{C}_{60}$
(see also Fig.~\ref{fig:c60dft}).
\begin{figure}[htp]
  \centering
  \begin{tabular}{cccc}
    \multicolumn{1}{l}{(a)}
    &\multicolumn{1}{l}{(b)}
    &\multicolumn{1}{l}{(c)}
    &\multicolumn{1}{l}{(d)}\\
    \includegraphics[bb=0 0 392 404,height=100pt]{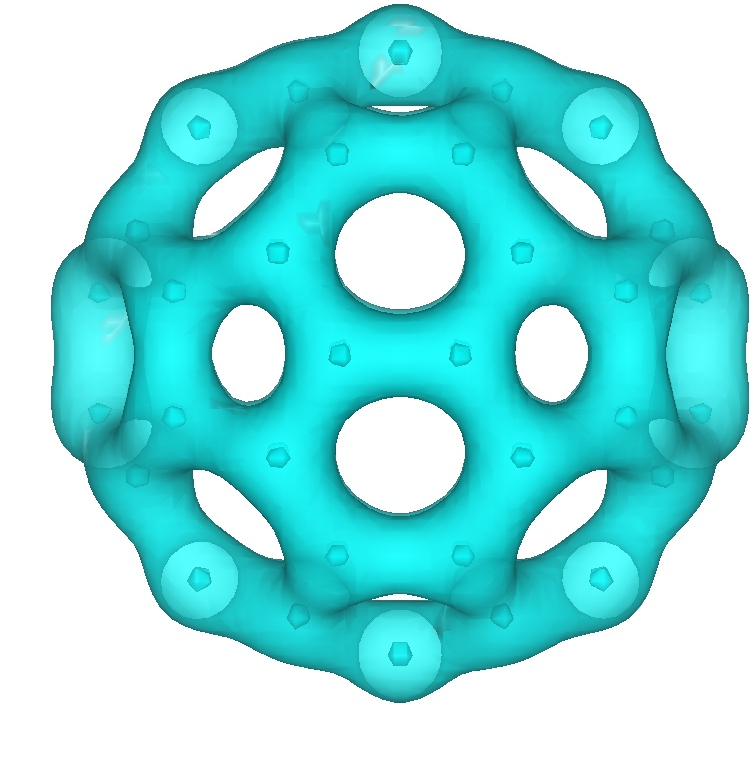}
    &\includegraphics[bb=0 0 392 404,height=100pt]{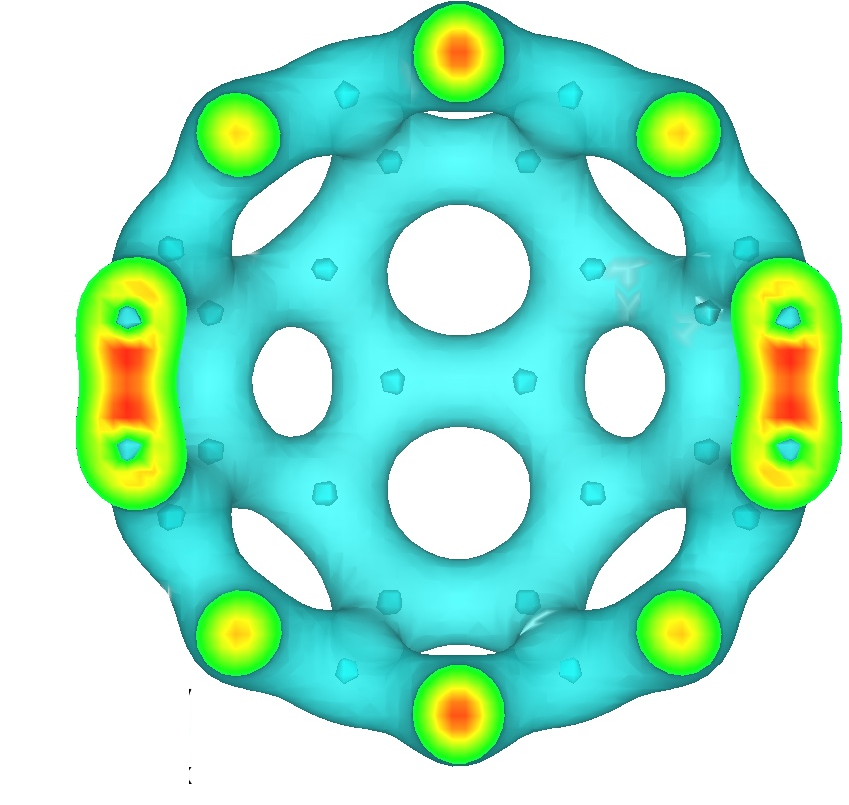}
    &\includegraphics[bb=0 0 392 404,height=100pt]{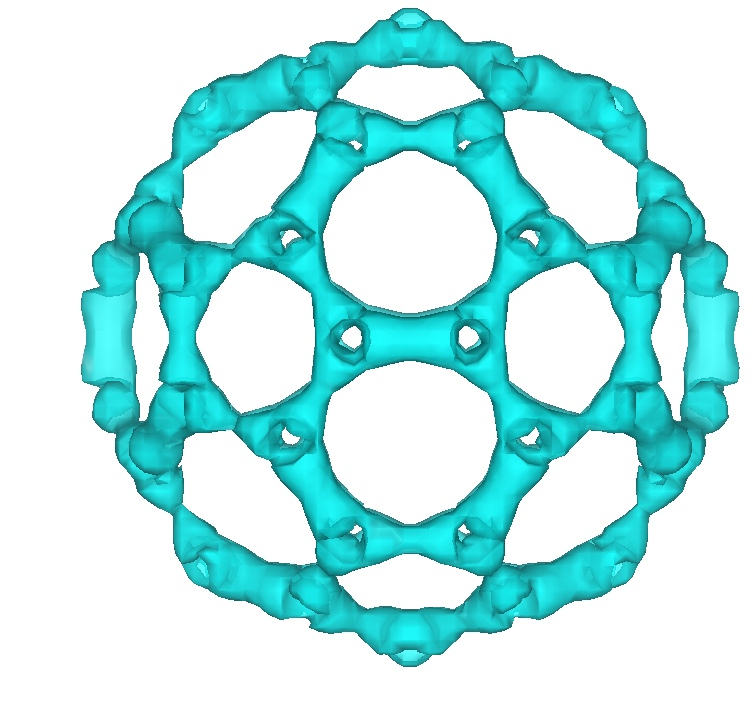}
    &\includegraphics[bb=0 0 392 404,height=100pt]{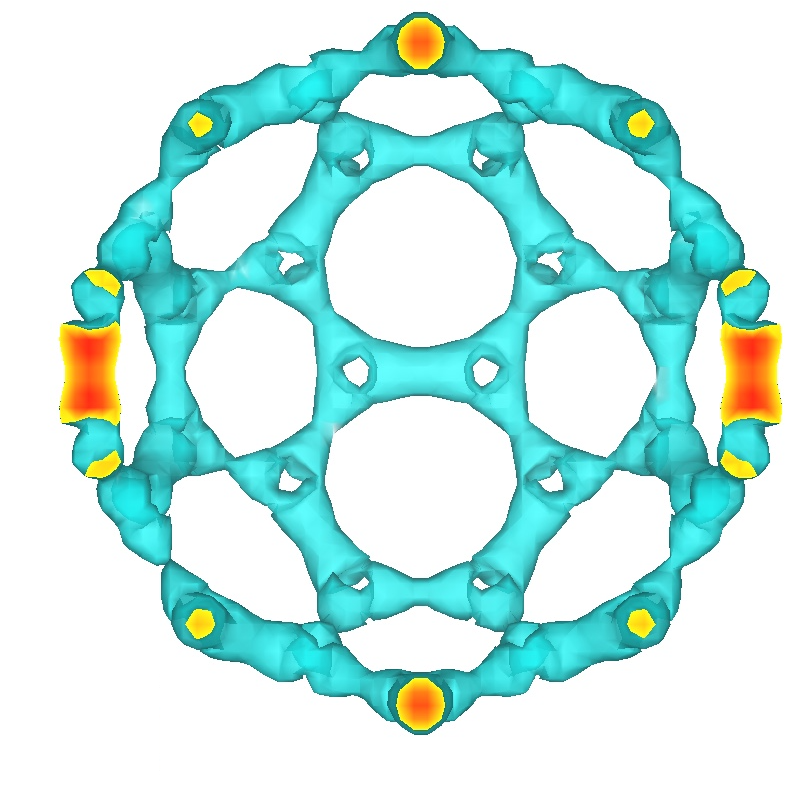}\\
  \end{tabular}
  \caption{Iso-surface of distributions of electrons of $\mathrm{C}_{60}$ by DFT.
    (a) probability $0.15$, (b) its cut-model, 
    (c) probability $0.25$, (d) its cut-model.
  }
  \label{fig:c60dft}
\end{figure}

\begin{figure}[htp]
  \centering
  \begin{tabular}{cc}
    \multicolumn{1}{l}{(a)}
    &\multicolumn{1}{l}{(b)}\\
    \includegraphics[bb=40 0 448 261,height=100pt]{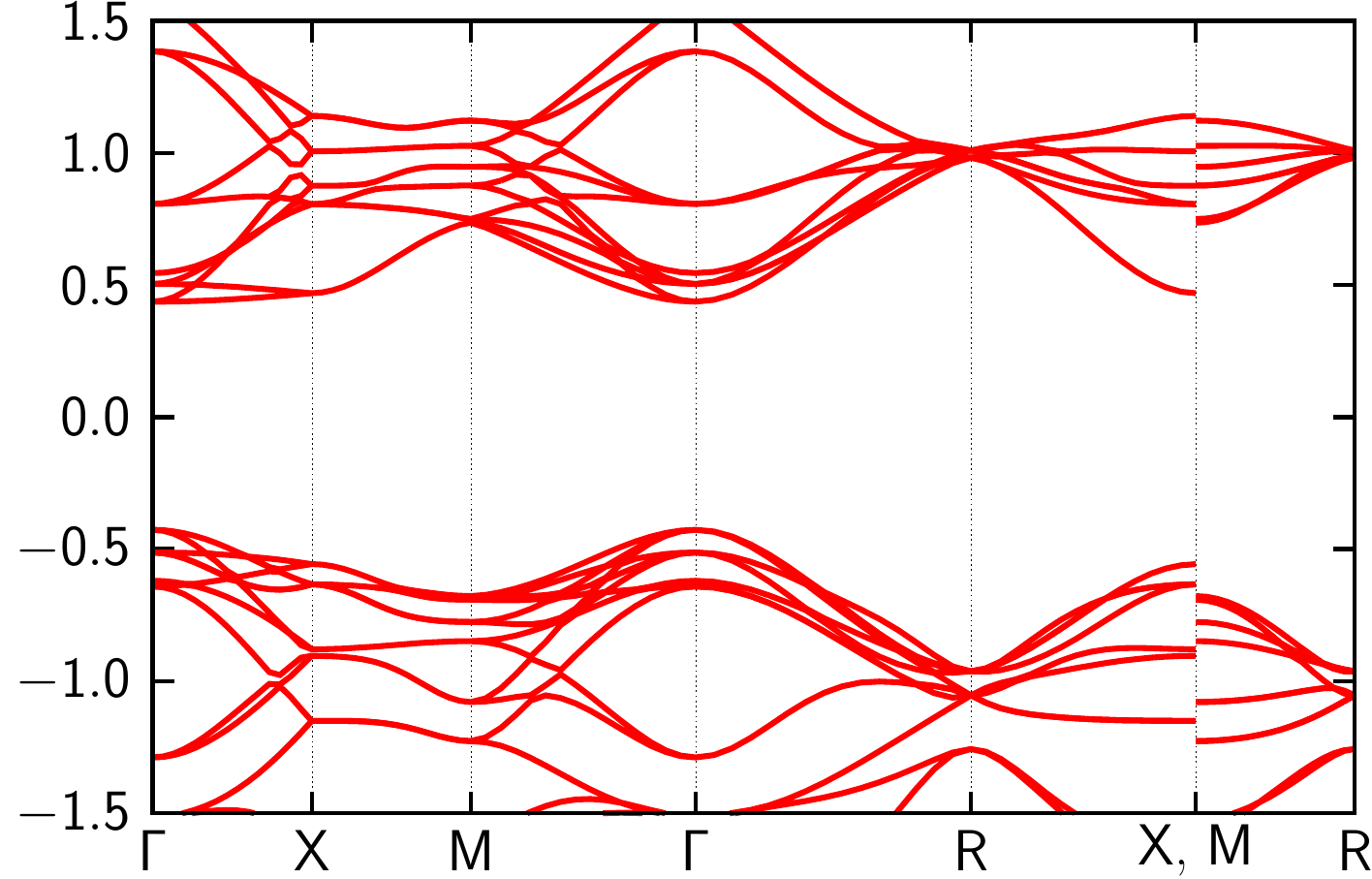}
    &\includegraphics[bb=0 0 815 752,height=100pt]{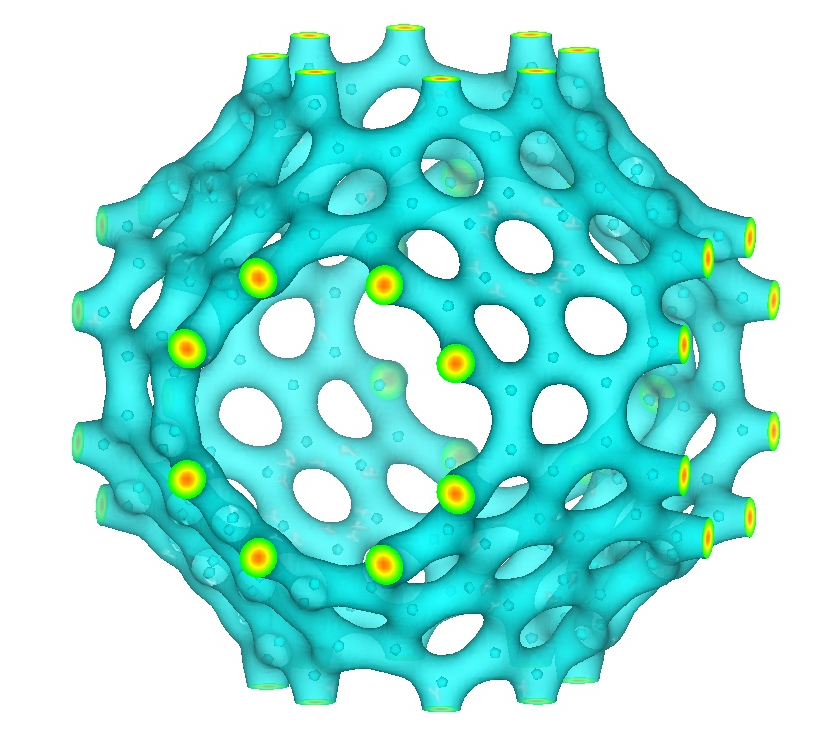} \\
  \end{tabular}
  \caption{
    DFT calculation result for Mackay-Terrones' structure.
    (a) electronic state, band gap is about 1\,eV, and hence Mackay-Terrones' structure is semi-conductor, 
    (b) iso-surface of density of electrons ($p=0.15$).
  }
  \label{fig:mackaydft}
\end{figure}
%%%%% END of document

%%%%% \input{nanotube.tex}
% -*- coding: utf-8 -*-
\subsection{Carbon nanotubes from geometric view points}
\label{sec:nanotube}
Carbon nanotubes are carbon allotropes, whose carbon atoms chemically bind with other three atoms with $sp^2$-orbitals, 
and they are graphenes rolled up in cylinders.
There are many types of carbon nanotubes.
However, in this section, 
we only consider single wall nanotubes (SWNT).
SWNTs have a parameter ({\em chiral index}) $\v{c} = (c_1, c_2)$, which is defined as follows.
\par
Choose a vertex of a regular hexagonal lattice (a graphene) as an origin $(0,0)$, the select fundamental piece, 
whose vertices are $\{\v{v}_i\}_{i=0}^3$ of 
a regular hexagonal lattice (\ref{eq:grapheneV}), 
and translation vectors as (\ref{eq:grapheneA}):
\begin{align}
  \label{eq:grapheneV}
  &\v{v}_0 = (0, 0), 
    \quad
    \v{v}_1 = (0, -1), 
    \quad
    \v{v}_2 = (1/2, \sqrt{3}/2), 
    \quad
    \v{v}_3 = (-1/2, \sqrt{3}/2), 
  \\
  \label{eq:grapheneA}
  &\v{a}_1 = (\sqrt{3}, 0), 
    \quad
    \v{a}_2 = (1/2, -\sqrt{3}/2).
\end{align}
Select $(c_1, c_2)$ with $c_1 \in \N_{>0}$ and $c_2 \in \N_{\ge0}$, 
and set $\v{c} = c_1 \v{a}_1 + c_2 \v{a}_2$, we call $\v{c}$ a {\em chiral vector} (or chiral index)  of SWNT.
On the other hand, $\v{t} = (1/\gcd(c_1, c_2))((c_1 + 2 c_2) \v{a}_1 - (2 c_1 + c_2) \v{a}_2)$, 
which is called a lattice vector, 
satisfies $\inner{\v{c}}{\v{t}} = 0$.
\par
A SWNT with the chiral index $\v{c}$ is the structure identifying $\v{0}$ and $\v{c}$, 
along the line between $\v{0}$ and $\v{t}$.
Note that the fundamental region of the SWNT with the chiral index $\v{c}$ is the rectangle 
with vertices $\v{0}$, $\v{t}$, $\v{t} + \v{c}$, and $\v{c}$.
The diameter of a SWNT with chiral index $\v{c} = (c_1, c_2)$ is 
$L = \sqrt{c_1^2 + c_1 c_2 + c_2^2}$.
\begin{figure}[htp]
  \centering
  \includegraphics[bb=0 0 192 179,height=150pt,clip=true]{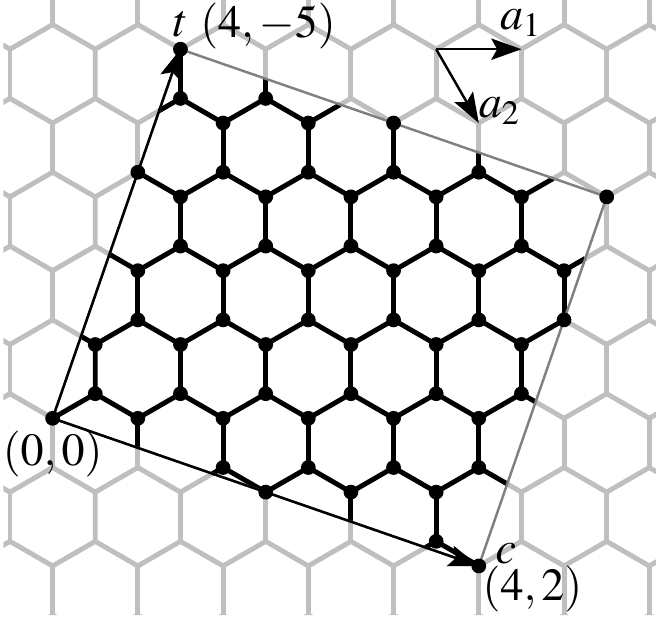}
  \caption{Construction of a SWNT from a regular hexagonal lattice}
  \label{fig:swnt}
\end{figure}
\par
Recently, 
there are several research composing a SWNT with short length using organic chemistry (see for example \cite{Isobe}), 
and hence an index measuring the length of SWNTs is needed.
Matsuno--Naito--Hitosugi--Sato--Kotani--Isobe propose such an index, which is called the {\em length index} of SWNT \cite{Isobe}:
\begin{equation}
  \label{eq:lengthindex}
  \frac{\sqrt{3} | c_1 (a_1 - b_1) - c_2 (a_2 - b_1) |}{2\sqrt{c_1^2 + c_1 c_2 + c_2^2}}, 
  \quad
  \text{ with edge atoms coordinates } (a_1, b_1) \text{ and } (a_2, b_1).
\end{equation}
The index (\ref{eq:lengthindex}) measures how many benzene rings (hexagons) are in the length direction.
\par
SWNTs with $c_1 = c_2$ are called {\em zigzag type}, $c_2 = 0$ are called {\em armchair type}, 
and otherwise are called {\em chiral type}.
These names come from shape of edges of SWNTs (see Fig.~\ref{fig:swnts}).
\begin{figure}[htp]
  \centering
  \begin{tabular}{ccc}
    \multicolumn{1}{l}{(a)}
    &\multicolumn{1}{l}{(b)}
    &\multicolumn{1}{l}{(c)}
    \\
    \includegraphics[bb=0 0 640 640,height=120pt]{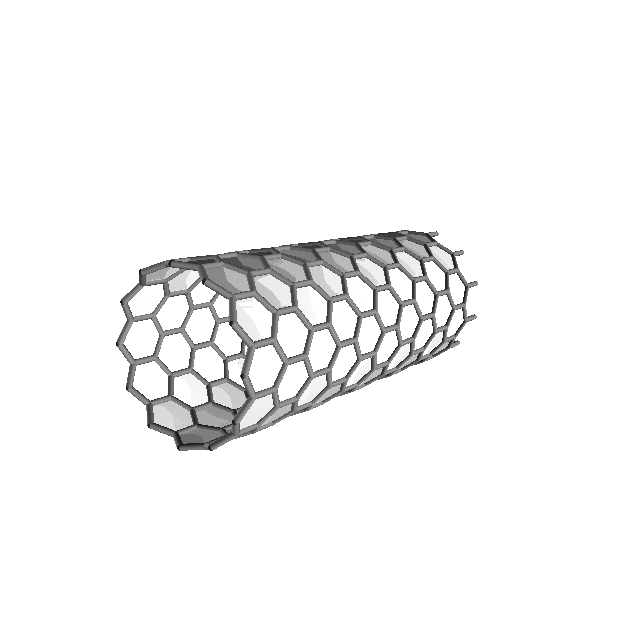}
    & \includegraphics[bb=0 0 640 640,height=120pt]{F/cnt_12_08_5_5_P2D.jpg}
    & \includegraphics[bb=0 0 640 640,height=120pt]{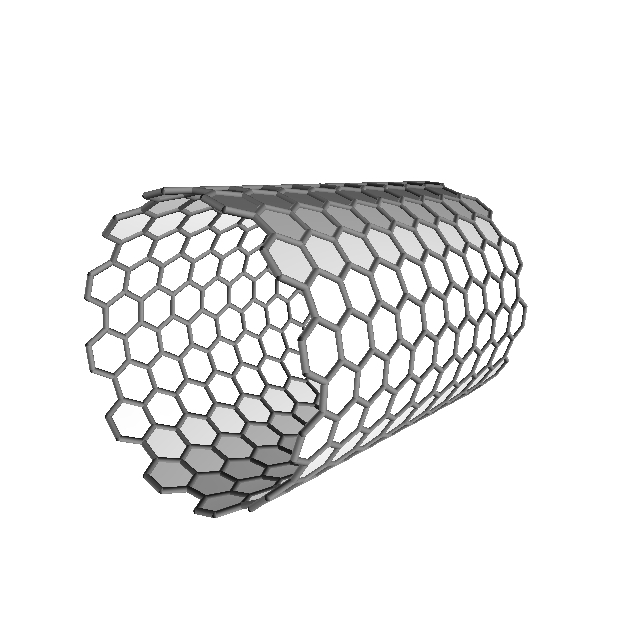}\\
  \end{tabular}
  \caption{
    (a) A zigzag type ($\v{c} = (12, 0)$), 
    (b) a chiral type ($\v{c} = (12, 8)$), 
    (c) an armchair type ($\v{c} = (12, 12)$).
  }
  \label{fig:swnts}
\end{figure}
\par
Electronic properties of SWNTs are also geometric. 
The following result is well-known, and is obtained by tight binding approximations.
If $c_1 \equiv c_2 \pmod{3}$, then SWNTs with $\v{c} = (c_1, c_2)$ are metallic,
if not, then such SWNTs are semi-conductors.
%%%%% END of document

%%%%% END of document

%%%%% \input{acknowledgement.tex}
% -*- coding: utf-8 -*-
\subsection*{Acknowledgement}
This note is based on lectures in 
``Introductory Workshop on Discrete Differential Geometry'' 
at Korea University on January 21--24, 2019.
The author greatly thanks to organizers of the workshop, 
and hospitality of Korea University.
The author thanks to Mr.~Tomoya NAITO (Department of Physics, The University of Tokyo), 
who calculates Figures \ref{fig:grapheneIso}, \ref{fig:c60dft}, and \ref{fig:mackaydft} by Density Functional Theory
by using \textsc{OpenMX} \cite{Science351-aad3000, PRB67-0155108, PRB69-195113, PRB72-0045121} 
and read my manuscript and give the author valuable comments.
The author also thanks to Professor Motoko KOTANI (AIMR, Tohoku University) and 
Dr.~Shintaro AKAMINE (Graduate School of Mathematics, Nagoya University), 
who read my manuscript and give the author valuable comments.
This work is partially supported by KAKENHI 17H06466.
%%%%% END of document

%%%%% \input{bibliography.tex}
% -*- coding: utf-8 -*-

%%%%% END of document


\begin{thebibliography}{10}
\bibitem{Aho}
  A.~V.~Aho, J.~E.~Hopcroft, and J.~D.~Ullman, 
  {\em The Design and Analysis of Computer Algorithms}, 
  Pearson, 1974.
  % 
\bibitem{EPINET}
  The Australian National University, The EPINET Project.
  {http://epinet.anu.edu.au/}
  % 
\bibitem{Bondy}
  J.~A.~Bondy and U.~S.~R.~Murty, 
  {\em Graph Theory}, 
  \newblock Graduate Texts in Mathematics vol.~244, Springer, 2008.
  % 
\bibitem{CRS}
  D.~Cvetkovi\'c, P.~Rowlinson, and S.~Simi\'c, 
  \newblock {\em An Introduction to the Theory of Graph Spectra}, 
  \newblock London Math.~Soc., Student Texts vol.~75, 2010.
  % 
\bibitem{DSV}
  G.~Davidoff, P.~Sarnack, and A.~Valette, 
  \newblock {\em Elementary Number Theory, Group Theory, and Ramanujan Graphs}, 
  \newblock London Math.~Soc., Student Texts vol.~55, 2003.
  % 
\bibitem{Delgado-Friedrichs-OKeeffe}
  O.~Delgado-Friedrichs and M.~O'Keeffe, 
  \newblock Identification of and symmetry computation for crystal nets.
  \newblock {\em Acta Crystallogr.~A}, {\bfseries 59}, 351--360, (2003).
  \newblock {doi:10.1107/S0108767303012017}
  % 
\bibitem{MR2429120}
  M.~Deza and M.~Dutour~Sikiri{\'c}, 
  \newblock {\em Geometry of chemical graphs: polycycles and two-faced maps},
  volume 119 of {\em Encyclopedia of Mathematics and its Applications}.
  \newblock Cambridge University Press, Cambridge, 2008.
  \newblock {doi:10.1017/CBO9780511721311}
  % 
\bibitem{MR2035314}
  M.~Dutour and M.~Deza, 
  \newblock Goldberg-{C}oxeter construction for 3- and 4-valent plane graphs, 
  \newblock {\em Electron.~J.~Combin.}, {\bfseries 11}, Research Paper 20, 49, (2004).
  \newline
  \newblock {http://www.combinatorics.org/Volume\_11/Abstracts/v11i1r20.html}
  % 
\bibitem{EellsSampson}
  J.~Eells and J.~H.~Sampson, 
  \newblock Harmonic mappings of Riemannian manifolds, 
  \newblock {\em Amer.~J.~Math.}, {\bfseries 86}, 109--160, (1964).
  \newblock {http://www.jstor.org/stable/2373037}
  % 
\bibitem{ISI:000391190500047}
  D.~Fujita, Y.~Ueda, S.~Sato, N.~Mizuno, T.~Kumasaka, and M.~Fujita, 
  \newblock Self-assembly of tetravalent Goldberg polyhedra from 144 small components, 
  \newblock {\em Nature}, {\bfseries 540}, 563, (2016).
  \newline
  \newblock {doi:10.1038/nature20771}
  % 
\bibitem{1937104}
  M.~Goldberg, 
  \newblock A class of multi-symmetric polyhedra, 
  \newblock {\em Tohoku Math.~J.}, First Series, {\bfseries 43}, 104--108, (1937).
  \newline
  \newblock 
  {https://www.jstage.jst.go.jp/article/tmj1911/43/0/43\_0\_104/\_pdf/-char/en}
  % 
\bibitem{Higuchi-Shirai}
  Yu.~Higuchi, and T.~Shirai, 
  \newblock Some spectral and geometric properties for infinite graphs, 
  \newblock {\em Contemp.~Math.}, {\bfseries 347}, 29--56, (2004).
  \newblock {doi:10.1090/conm/347/06265}
  % 
\bibitem{K4}
  M.~Itoh, M.~Kotani, H.~Naito, T.~Sunada, Y.~Kawazoe, and T.~Adschiri, 
  \newblock New metallic carbon crystal, 
  \newblock {\em Phys.~Rev.~Lett.}, {\bfseries 102}, 055703, (2009).
  \newblock {doi:10.1103/PhysRevLett.102.055703}
  % 
\bibitem{KajigayaTanaka}
  T.~Kajigaya and R.~Tanaka, 
  \newblock Uniformizing surfaces via discrete harmonic maps, 
  \newblock arXiv:1905.05427.
  % 
\bibitem{Itami}
  K.~Kawasumi, Q.~Zhang, Y.~Segawa, L.~T.~Scott and K.~Itami, 
  \newblock 
  A grossly warped nanographene and the consequences of multiple odd-membered-ring defects, 
  \newblock {\em Nature Chemistry}, {\bfseries 5}, 739--744 (2013).
  \newblock {doi:10.1038/NCHEM.1704}
  % 
\bibitem{Kotani-Naito-Omori}
  M.~Kotani, H.~Naito, and T.~Omori, 
  \newblock A discrete surface theory, 
  \newblock {\em Comput.~Aided Geom.~Design}, {\bfseries 58}, 24--54, (2017).
  \newblock {doi:10.1016/j.cagd.2017.09.002}
  % 
\bibitem{Kotani-Naito-Tao}
  M.~Kotani, H.~Naito, and C.~Tao, 
  \newblock Construction of continuum from a discrete surface by its iterated subdivisions, 
  \newblock {arXiv:1806.03531.}
  % 
\bibitem{MR1743611}
  M.~Kotani and T.~Sunada, 
  \newblock Albanese maps and off diagonal long time asymptotics for the heat kernel, 
  \newblock {\em Comm.~Math.~Phys.}, {\bfseries 209}, 633--670, (2000).
  \newblock {doi:10.1007/s002200050033}
  % 
\bibitem{Kotani-Sunada-Proceedings}
  M.~Kotani and T.~Sunada.
  \newblock A central limit theorem for the simple random walk on a crystal lattice, 
  \newblock Proceedings of the Second ISAAC Newton Congress, vol.~1, 1--6, 2000, Springer.
  \newblock {doi:10.1007/978-1-4613-0269-8\_1}
  % 
\bibitem{MR1748964}
  M.~Kotani and T.~Sunada, 
  \newblock Jacobian tori associated with a finite graph and its abelian covering graphs, 
  \newblock {\em Adv.~in Appl.~Math.}, {\bfseries 24}, 89--110, (2000).
  \newblock {doi:10.1006/aama.1999.0672}
  % 
\bibitem{MR1783793}
  M.~Kotani and T.~Sunada, 
  \newblock Standard realizations of crystal lattices via harmonic maps, 
  \newblock {\em Trans.~Amer.~Math.~Soc.}, {\bfseries 353}, 1--20, (2001).
  \newblock {doi:10.1090/S0002-9947-00-02632-5}
  % 
\bibitem{MR2039958}
  M.~Kotani and T.~Sunada, 
  \newblock Spectral geometry of crystal lattices, 
  \newblock {\em Contemp.~Math.}, {\bfseries 338}, 271--305, (2003).
  \newblock {doi:10.1090/conm/338/06077}
  % 
\bibitem{Science351-aad3000}
  K.~Lejaeghere, \textit{et al.}, 
  \newblock Reproducibility in density functional theory calculations of solids, 
  \newblock {\em Science}, {\bfseries 351}, aad3000, (2016).
  \newblock {doi:10.1126/science.aad3000}
  % 
\bibitem{Lenosky:1992}
  T.~Lenosky, X.~Gonze, M.~Teter, and V.~Elser, 
  \newblock Energetics of negatively curved graphitic carbon, 
  \newblock {\em Nature}, {\bfseries 355}, 333--335, (1992).
  \newblock {doi:10.1038/355333a0}
  % 
\bibitem{Mackay-Terrones:1991}
  A.~Mackay and H.~Terrones, 
  \newblock Diamond from graphite, 
  \newblock {\em Nature}, {\bfseries 352}, 762, (1991).
  \newline
  \newblock {doi:10.1038/352762a0}
  % 
\bibitem{Isobe}
  T.~Matsuno, H.~Naito, S.~Hitosugi, S.~Sato, M.~Kotani, and H.~Isobe, 
  \newblock Geometric measures of finite carbon nanotube molecules: A proposal for length index and filling indexes, 
  \newblock {\em Pure Appl.~Chem.}, {\bfseries 86}, 489--495 (2014).
  \newblock {doi:10.1515/pac-2014-5006}
  % 
\bibitem{Awaga}
  A.~Mizuno, Y.~Shuku, M.~M.~Matsushita, M.~Tsuchiizu, Y.~Hara, N.~Wada, Y.~Shimizu, and K.~Awaga, 
  \newblock 3D spin-liquid state in an organic hyperkagome lattice of mott dimers, 
  \newblock {\em Phys.~Rev.~Lett.}, {\bfseries 119}, 057201 (2017).
  \newblock {doi:10.1103/PhysRevLett.119.057201}
  % 
\bibitem{MR1500145}
  H.~Naito, 
  \newblock Visualization of standard realized crystal lattices, 
  \newblock {\em Contemp.~Math.}, {\bfseries 484}, 153--164, (2009).
  \newblock {doi:10.1090/conm/484/09472}
  % 
\bibitem{Naito:2016}
  H.~Naito, 
  \newblock Construction of negatively curved carbon crystals via standard realizations, 
  \newblock {\em Springer Proc.~Math.~Stat.}, {\bfseries 166}, 83--100, (2016).
  \newblock {doi:10.1007/978-4-431-56104-0\_5}
  % 
\bibitem{PlaneNets}
  M.~O'Keeffe and B.~G.~Hyde, 
  \newblock Plane nets in crystal chemistry, 
  \newblock {\em Philos.~Trans.~R.~Soc.~A}, {\bfseries 295}, 553-618, (1980).
  \newblock {doi:10.1098/rsta.1980.0150}
  % 
\bibitem{Omori-Naito-Tate}
  T.~Omori, H.~Naito, and T.~Tate, 
  \newblock Eigenvalues of the Laplacian on the Goldberg-Coxeter constructions for $3$- and $4$-valent graphs, 
  \newblock {\em Electron.~J.~Combin.}, {\bfseries 26(3)}, (2019), $\sharp$P3.7.
  \newblock {https://www.combinatorics.org/ojs/index.php/eljc/article/view/v26i3p7}
  % 
\bibitem{PRB67-0155108}
  T.~Ozaki, 
  \newblock Variationally optimized atomic orbitals for large-scale electronic structures, 
  \newblock {\em Phys.~Rev.~B}, {\bfseries 67}, 155108, (2003).
  \newblock {doi:10.1103/PhysRevB.67.155108}
  % 
\bibitem{PRB69-195113}
  T.~Ozaki and H.~Kino, 
  \newblock Numerical atomic basis orbitals from H to Kr, 
  \newblock {\em Phys.~Rev.~B}, {\bfseries 69}, 195113, (2004).
  \newblock {doi:10.1103/PhysRevB.69.195113}
  % 
\bibitem{PRB72-0045121}
  T.~Ozaki and H.~Kino, 
  \newblock Efficient projector expansion for the \textit{ab initio} LCAO method, 
  \newblock {\em Phys.~Rev.~B}, {\bfseries 72}, 045121, (2005).
  \newblock {doi:10.1103/PhysRevB.72.045121}
  % 
\bibitem{MatrixCookbook}
  K.~B.~Petersen and M.~S.~Pertersen, 
  \newblock {\em The Matrix Cookbook}, 
  \newline
  \newblock {https://www.math.uwaterloo.ca/{\char"7E}hwolkowi/matrixcookbook.pdf}
  % 
\bibitem{SingerThorpe}
  L.~M.~Singer and J.~A.~Thorpe, 
  \newblock {\em Lecture Notes on Elementary Topology and Geometry}, 
  \newblock Springer, 1967.
  \newblock {https://www.springer.com/gp/book/9780387902029}
  % 
\bibitem{MR2375022}
  T.~Sunada, 
  \newblock Crystals that nature might miss creating, 
  \newblock {\em Notices Amer.~Math.~Soc.}, {\bfseries 55}, 208--215, (2008).
  \newblock {http://www.ams.org/notices/200802/tx080200208p.pdf}
  \newline
  \newblock T.~Sunada, 
  \newblock Correction: ``Crystals that nature might miss creating'', 
  \newline
  \newblock {\em Notices Amer.~Math.~Soc.}, {\bfseries 55}, 343, (2008).
  \newline
  \newblock {https://www.ams.org/journals/notices/200803/tx080300342p.pdf}
  % 
\bibitem{MR2902247}
  T.~Sunada.
  \newblock Lecture on topological crystallography, 
  \newblock {\em Jpn.~J.~Math.}, {\bfseries 7}, 1--39, (2012).
  \newline
  \newblock {doi:10.1007/s11537-012-1144-4}
  % 
\bibitem{MR3014418}
  T.~Sunada, 
  \newblock {\em Topological crystallography}, vol.~6 of {\em Surveys and Tutorials in the Applied Mathematical Sciences}.
  \newblock Springer, Tokyo, 2013.
  \newblock {doi:10.1007/978-4-431-54177-6}
  % 
\bibitem{Tagami:2014}
  M.~Tagami, Y.~Liang, H.~Naito, Y.~Kawazoe, and M.~Kotani, 
  \newblock \scalebox{0.95}[1.0]{Negatively curved cubic} carbon crystals with octahedral symmetry, 
  \newblock {\em Carbon}, {\bfseries 76}, 266--274, (2014).
  \\
  \newblock {doi:10.1016/j.carbon.2014.04.077}
  % 
\bibitem{Tao}
  C.~Tao, 
  \newblock A construction of converging Goldberg-Coxeter subdivisions of a discrete surface, 
  preprint, (2018).
  % 
\bibitem{Tsuchiizu}
  M.~Tsuchiizu, 
  \newblock Three dimensional higher-spin Dirac and Weyl dispersions in strongly isotropic $K_4$ crystal, 
  \newblock {\em Phys.~Rev~ B}, {\bfseries 94}, 195426, (2016).
  \newblock {doi:10.1103/PhysRevB.94.195426}
  % 
\bibitem{Kawazoe}
  S.~Zhang, J.~Zhou, Q.~Wang, X.~Chen, Y.~Kawazoe, and P.~Jena, 
  \newblock \scalebox{0.95}[1.0]{Penta-graphene:} A new carbon allotrope, 
  \newblock {\em Proc.~Nat.~Acad.~Sci.}, {\bfseries 112} 2372--2377, (2015).
  \newblock {doi:10.1073/pnas.1416591112}
  % 
\end{thebibliography}
\end{document}